\documentclass[11pt,leqno]{article}

\usepackage{amssymb,amsmath,amsthm,mysects,url,rotating}
 \usepackage{graphicx,epsfig}
\newcommand{\n}{\noindent}
\newcommand{\ms}{\medskip}
\newcommand{\vp}{\varepsilon}
\newcommand{\bb}[1]{\mathbb{#1}}
\newcommand{\cl}[1]{\mathcal{#1}}
\newcommand{\sst}{\scriptstyle}
\newcommand{\ovl}{\overline}
\newcommand{\intl}{\int\limits}

\theoremstyle{plain}
\newtheorem{thm}{Theorem}[section]
\newtheorem{lem}[thm]{Lemma}
\newtheorem*{lm}{Lemma}
\newtheorem{pro}[thm]{Proposition}

\newtheorem{cor}[thm]{Corollary}

\theoremstyle{definition}

\newtheorem{dfn}[thm]{Definition}
\newtheorem{prbl}[thm]{Problem}

\theoremstyle{remark}
\newtheorem{rem}[thm]{Remark}
\newtheorem*{rmk}{Remark}

\numberwithin{equation}{section}

\def\tilde{\widetilde}

\renewcommand{\tilde}{\widetilde}

\def\RR{\bb R}
\def\CC{\bb C}
\def\KK{\bb K}

\begin{document}

\title{Grothendieck's Theorem, past and present}

\author{by\\
Gilles  Pisier\footnote{Partially supported by NSF grant 0503688}\\
Texas A\&M University\\
College Station, TX 77843, U. S. A.\\
and\\
Universit\'e Paris VI\\
Equipe d'Analyse, Case 186, 75252\\
Paris Cedex 05, France}

\maketitle

\begin{abstract}
Probably the most famous of Grothendieck's contributions to Banach space theory is the result that he himself described as ``the fundamental theorem in the metric theory of tensor products''. That is now commonly referred to as ``Grothendieck's theorem'' (GT in short), or sometimes as ``Grothendieck's inequality''. This had a major impact first in Banach space theory (roughly after 1968), then, later on, in $C^*$-algebra theory, (roughly after 1978). More recently, in this millennium, a new version of GT has been successfully developed in the framework of ``operator spaces'' or non-commutative Banach spaces. In addition, GT independently surfaced in several quite unrelated fields:\ in connection with Bell's inequality in quantum mechanics, in graph theory where the Grothendieck constant of a graph has been introduced and in computer science where the Grothendieck inequality is invoked to replace certain NP hard problems by others that can be treated by ``semidefinite programming' and hence solved in polynomial time. This expository paper (where many proofs are included),  presents a review of all these topics, starting from the original GT. We  concentrate on the more recent developments and merely outline
those of the first Banach space period since   detailed accounts of that are already available, for instance
the author's 1986 CBMS notes.
\end{abstract}

\thispagestyle{empty}

\vfill\eject

\thispagestyle{empty}

\tableofcontents
\vfill\eject

\setcounter{page}{1}

\section{Introduction}

\n{\bf The R\'esum\'e saga } In 1953, Grothendieck published an extraordinary paper \cite{Gr1} entitled ``R\'esum\'e de la th\'eorie m\'etrique des produits tensoriels topologiques,'' now often jokingly referred to as ``Grothendieck's r\'esum\'e''(!). Just like his thesis (\cite{Gr2}), this was devoted to tensor products of topological vector spaces, but in sharp contrast with the thesis devoted to the locally convex case, the ``R\'esum\'e'' was exclusively concerned with Banach spaces (``th\'eorie m\'etrique'').
The central result of this long paper (``Th\'eor\`eme fondamental de la th\'eorie m\'etrique des produits tensoriels topologiques'') is now called Grothendieck's Theorem (or Grothendieck's inequality). We will refer to it as GT. Informally, one could describe GT as a surprising and non-trivial relation between Hilbert space (e.g.\ $L_2$) and the two fundamental Banach spaces $L_\infty, L_1$ (here $L_\infty$ can be replaced by the space $C(S)$ of continuous functions on a compact set $S$). That relationship was expressed by an inequality involving the 3 fundamental tensor norms (projective, injective and Hilbertian),  described  in Theorem \ref{thm1.5} below. 
 The paper went on to investigate the 14 other tensor norms
that can be derived from the first   3 (see Remark \ref{symbol}). When it appeared this astonishing paper 
was virtually ignored. 
 
Although the paper was reviewed in Math Reviews by none less than Dvoretzky, it  seems to have been widely  ignored, until Lindenstrauss and Pe{\l}czy\'nski's 1968 paper \cite{LP} drew attention to it. Many explanations come to mind: it was written in French, published  in a Brazilian journal with very  limited circulation and, in a major reversal from the author's celebrated thesis, it   ignored locally convex questions and concentrated   exclusively on Banach spaces, a move that probably went against the tide at that time.

The situation changed radically (15 years later) when Lindenstrauss and Pe{\l}czy\'nski \cite{LP} discovered the paper's numerous gems, including the solutions to several open questions that had been raised \emph{after} its appearance! Alongside with \cite{LP}, Pietsch's work had just appeared and it contributed to the resurgence of the ``r\'esum\'e'' although Pietsch's emphasis was on spaces of operators (i.e.\ dual spaces) rather than on tensor products (i.e.\ their preduals), see \cite{Pie}. Lindenstrauss and Pe{\l}czy\'nski's beautiful paper completely dissected the ``r\'esum\'e'', rewriting the proof of the fundamental theorem (i.e.\ GT) and giving of it many reformulations, some very elementary ones (see Theorem~\ref{thm1.2bispre} below) as well as some more refined consequences involving absolutely summing operators between ${\cl L}_p$-spaces, a generalized notion of $L_p$-spaces that had just been introduced by Lindenstrauss and Rosenthal. 
Their work also emphasized a very useful factorization of operators from a ${\cl L}_\infty$ space to a Hilbert space (cf.\ also \cite{DPR}) that is much easier to prove that GT itself, and is now commonly called the ``little GT'' (see \S~\ref{sec2}).

Despite all these efforts, the only known proof of GT remained the original one until Maurey \cite{Ma} found the first new proof using an extrapolation method that turned out to be extremely fruitful. After that, several new proofs were given, notably a strikingly short one based on Harmonic Analysis by Pe{\l}czy\'nski and Wojtaszczyk (see \cite[p. 68]{P2}). Moreover Krivine \cite{Kr1,Kr2} managed to improve the original proof and the bound for the Grothendieck constant $K_G$, that remained the best until very recently. Both Krivine's and the original proof of GT are included in \S~\ref{sec1} below.

 In   \S \ref{sec1}   we will give many   different equivalent forms of GT,
 but we need a starting point, so we choose the following most elementary formulation (put forward in \cite{LP}):
 
 \begin{thm}[First statement of GT]\label{thm1.2bispre} 
Let $[a_{ij}]$ be an $n\times n$ scalar matrix  $(n\ge 1)$. Assume   that
for any $n$-tuples of scalars $(\alpha_i)$, $(\beta_j)$ we have
\begin{equation}\label{eq1.2prebispre}
 \quad
\left|\sum a_{ij}\alpha_i\beta_j\right| \le \sup_i|\alpha_i| \sup_j|\beta_j|.\qquad { }\qquad { }\qquad { }\end{equation}
Then for any Hilbert space $H$ and any $n$-tuples $(x_i), (y_j)$ in $H$ we have
\begin{equation}\label{eq1.2bispre+}
 \left|\sum a_{ij}\langle x_i,y_j\rangle\right| \le K\sup\|x_i\| \sup\|y_j\|,
\end{equation}
where $K$  is a numerical constant.
The best $K$ (valid for all $H$ and all $n$) is denoted by $K_G$.
\end{thm}
In this statement (and throughout this paper) the scalars can be either real or complex. But curiously, that affects the constant $K_G$, so
we must distinguish its value in the real case $K^{\bb R}_G$ and in the complex case $K^{\bb C}_G$. To this day, its exact value is still unknown although it is known that $1< K^{\bb C}_G < K^{\bb R}_G \le 1.782$, see \S \ref{sec1+} for more information.

 This   leads one to wonder what \eqref{eq1.2bispre+} means for a matrix (after normalization), i.e. what are the matrices
 such that 
for any Hilbert space $H$ and any $n$-tuples $(x_i), (y_j)$ 
of unit vectors in $H$ we have
\begin{equation}\label{eq1.2bispre?}
 \left|\sum a_{ij}\langle x_i,y_j\rangle\right| \le 1 \ ?
\end{equation}
The answer is another innovation of the R\'esum\'e, an original application of the Hahn--Banach theorem (see
Remark \ref{rem4.3bis}) that leads to a factorization of
the matrix  $[a_{ij}]$\ :  The preceding property \eqref{eq1.2bispre?} holds iff
there is a matrix $[\tilde a_{ij}]$ defining an operator of norm
at most 1 
\footnote{actually in \cite{Gr1} there is an extra  factor 
2, later removed in \cite{Kw1}} on the $n$-dimensional Hilbert space $\ell_2^n$
and numbers $\lambda_i\ge 0$, $\lambda'_j\ge 0$ such that
\begin{equation}\label{eq1.2bisprebis}\sum \lambda^2_i=1,\  \sum{ \lambda'}^2_j=1\quad{\rm and }\quad a_{ij}= \lambda_i \tilde a_{ij} {\lambda'}_j. \end{equation}
Therefore, by homogeneity,  \eqref{eq1.2prebispre} implies  a factorization of the form
 \eqref{eq1.2bisprebis} with $\|[\tilde a_{ij}]\|\le K$.\\
These results hold in much broader generality\ : we can replace
the set of indices $[1,\cdots,n]$ by any 
 compact set $S$, and   denoting by $C(S)$  the space of 
continuous functions on   $S$ equipped with the sup-norm,
we may replace the matrix $[a_{ij}]$ by a bounded bilinear form 
$\varphi$ on $C(S)\times C(S')$ (where $S'$ is any other compact set).
In this setting, GT says that 
there are  probability measures ${\bb P},{\bb P}'$ on $S,S'$ and a bounded bilinear form $\widetilde\varphi\colon \ L_2({\bb P})\times L_2({\bb P}')\to {\bb K}$ with $\| \widetilde\varphi \|\le K$ such that $\widetilde\varphi(x,y)= \varphi(x,y)$ for any $(x,y)$ in $C(S)\times C(S')$. 
 \\In other words, any  bilinear form $\varphi$ that is bounded on $C(S)\times C(S')$  actually ``comes'' from another one  $\widetilde\varphi$ that is bounded on  $L_2({\bb P})\times L_2({\bb P}').$
 
Actually this special factorization through $L_2({\bb P})\times L_2({\bb P}')$ is non-trivial even if we assume in the first place that there is a Hilbert space
$H$ together with  norm 1 inclusions $C(S)\subset H$ and  $C(S')\subset H$ and
a bounded bilinear form $\widetilde\varphi\colon \ H\times H\to {\bb K}$ with $\| \widetilde\varphi \|\le 1$ such that $\widetilde\varphi(x,y)= \varphi(x,y)$ for any $(x,y)$ in $C(S)\times C(S')$. However, it is a much easier
to conclude with this assumption. Thus, the corresponding result
is called the ``little GT".
%\footnote{actually in \cite{Gr1} there is an extra  factor 
% 2, later removed in \cite{Kw1}}
 
 \n {\bf More recent results} 
 The ``R\'esum\'e'' ended with a remarkable list of six problems, on top of which was the famous ``Approximation problem'' solved by Enflo in 1972. By 1981, all of the other problems had been solved (except for the value of the best constant in GT, now denoted by $K_G$).
Our CBMS notes from 1986 \cite{P2} contain a detailed account of the work from that period, so we will not expand on that here. We merely summarize this very briefly in \S~\ref{sec30} below. In the present survey, we choose to focus solely on GT.  One of the six problems was to prove a non-commutative version of GT for bounded bilinear forms on $C^*$-algebras. 

{\bf Non-commutative GT} Recall that a $C^*$-algebra $A$ is a closed self-adjoint subalgebra
of the space $B(H)$ of all bounded operators on a Hilbert space. 
By spectral theory, if 
$A$ is  unital  and commutative (i.e. the operators
in $A$ are all normal and mutually commuting), then $A$ can be identifed with the algebra $C(S)$
of continuous functions   on a  compact space $S$ (it is easy to justify reducing consideration to  the unital case). 
Recall that the operators that form the ``trace class" 
on a Hilbert space $H$ are those operators on $H$ that can
be written as the product of two Hilbert-Schmidt operators. With the associated norm,
the unit ball of the trace class is formed
of products of  two  operators in the  Hilbert-Schmidt unit ball. When a $C^*$-algebra happens to be isometric to a dual Banach space
(for instance $A=B(H)$ is the dual of the trace class), then it can be realized as a weak$^*$-closed
subalgebra of  $B(H)$. Such algebras are called von Neumann algebras (or
$W^*$-algebras). In the commutative case this corresponds
to algebras $L_\infty(\Omega,\mu)$ on some measure space $(\Omega,\mu)$.

Since one of the formulations
of GT (see Theorem \ref{thm1.2} below) was a special factorization for bounded bilinear forms
on $C(S)\times C(S')$ (with compact sets $S,S'$), it was natural for Grothendieck to ask whether
a similar factorization held for  bounded bilinear forms on the product
of two {\it non-commutative} $C^*$-algebras. 
This was proved in \cite{P1} with some restriction and in \cite{H3} in full generality. To give a concrete example, consider  
 the subalgebra $K(H)\subset B(H)$ of all compact operators on $H=\ell_2$ viewed as bi-infinite matrices
 (the reader   may as well replace $K(H)$ by the normed algebra of $n\times n$ complex matrices,
  but then the result
 must be stated in a   ``quantitative form" with uniform constants independent of the dimension $n$). Let us denote by $S_2(H)$ the (Hilbert) space formed
 by the Hilbert-Schmidt operators on $H$.
 Let $\widetilde \varphi$ be a bounded bilinear form on  $S_2(H)$ and let $a,b\in S_2(H)$.
 Then   there are four ``obvious" types
 of bounded bilinear form on  $K(H)\times K(H)$ that can be associated to $\widetilde \varphi$ and $a,b$.
 Those are:\\ \centerline{$\varphi_1(x,y)=\widetilde \varphi(ax, yb),\ \varphi_2(x,y)=\widetilde \varphi(xa,by),\ 
\varphi_3(x,y)=\widetilde \varphi(a x, by),\ 
\varphi_4(x,y)=\widetilde \varphi(xa, yb)$.} The content of the non-commutative  GT in this case is that {\it any} bounded bilinear form $\varphi$ on $K(H)\times K(H)$ can be decomposed as a sum of four forms of each of the four types (see \eqref{eq3.1} and  Lemma \ref{lem3.3} for details).
 In the  general case,  the non-commutative  GT can be  
stated as an inequality satisfied by all bounded bilinear forms on
$C^*$-algebras (see \eqref{eq3.2+}).
Let $K'_G$ (resp.\ $k'_G$) denote the best possible constant in that non-commutative  GT-inequality (resp.\   little GT) reducing to the original GT
  in the commutative case. Curiously, in sharp contrast with the commutative case, the exact values $K'_G=k'_G=2$ are known, following \cite{HI}. We present this in \S~\ref{sec9}.
 
The non-commutative  GT (see \S~\ref{sec3}), or actually the weaker non-commutative little GT (see \S~\ref{sec2}) had a major impact in Operator Algebra Cohomology
(see \cite{SS}), starting with the proof of a conjecture of Ringrose in \cite{P1}.
Both proofs in  \cite{P1,H3} use a certain form of non-commutative Khintchine inequality.
We expand on this   in   \S \ref{sec5}.

{\bf Operator space GT} Although these results all deal with non-commutative $C^*$-algebras, they still belong to classical Banach space theory. However, around 1988, a theory of non-commutative or ``quantum'' Banach spaces emerged with the thesis of Ruan and the work of Effros--Ruan, Blecher and Paulsen. In that theory the spaces remain Banach spaces but the morphisms are different: The familiar space $B(E,F)$ of bounded linear maps between two Banach spaces is replaced by the smaller space $CB(E,F)$ formed of the completely bounded (c.b. in short) ones defined in \eqref{def-cb}  below.
Moreover, each Banach  space $E$ comes equipped with an additional structure in the form of an isometric embedding (``realization'') $E\subset B(H)$ into the algebra of bounded operators on a Hilbert space $H$. 
Thus, by definition,  an operator space
is a Banach space $E$ given together with  an isometric embedding
$E\subset B(H)$ (for some $H$). 
Thus Banach spaces are given a structure resembling that of a $C^*$-algebra, but contrary to $C^*$-algebras which admit a privileged realization, Banach spaces may admit many inequivalent operator space structures. \\
Let us now define the space of morphisms $CB(E,F)$ used
in operator space theory.
Consider a subspace $E\subset B(H)$. Let $M_n(E)$ denote the space
of $n\times n$ matrices with entries in $E$. 
Viewing a matrix with entries in $B(H)$ as an operator acting
on $H\oplus\cdots\oplus H$ in the obvious way, we may clearly equip
this space with the norm induced by $M_n(B(H))=B(H\oplus\cdots\oplus H)$.
Now let $F\subset B({\cl H})$ be another operator space and let $M_n(F)$
be the associated normed space.
We say that a linear map $u\colon\ E\to F$ is completely bounded  (c.b. in short)
if the mappings $u_n \colon\ M_n(E)\to M_n(F)$ are bounded uniformly over $n$,
and we define
\begin{equation}\label{def-cb} \|u\|_{cb}=\sup\nolimits_{n\ge 1} \|u_n\| .
\end{equation}
We give a very brief outline of that theory in \S~\ref{sec11} and \S \ref{ssec1.8}. Through the combined efforts of Effros--Ruan and Blecher--Paulsen, an analogue of Grothendieck's program was developed for operator spaces, including
a specific duality theory,  analogues of the injective and projective tensor products and the approximation property. 
One can also define similarly a notion of completely bounded  (c.b. in short)
bilinear form $\varphi\colon\ E\times F \to \bb C$ on the product of two operator spaces.  
Later on, a bona fide analogue of Hilbert space, i.e.\ a unique self-dual object among operator spaces, (denoted by $OH$) was found (see \S \ref{ssec1.8}) with factorization properties matching exactly the classical ones (see \cite{P6}). Thus it became natural to search for an analogue of GT for operator spaces.  This came in several steps: \cite{JP,PS,HM2}
 described below in \S \ref{sec13bis} and \S \ref{sec14}.   One intermediate step
  from \cite{JP} was a factorization and extension theorem   for  c.b.\ bilinear forms on the product of two \emph{exact} operator spaces, $E,F\subset B(H)$, for instance two subspaces of     $K(H)$. The precise definition
  of the term ``exact" is slightly technical.
Roughly an operator space $E$ is \emph{exact} if all its finite dimensional subspaces can be realized  
as   subspaces of finite dimensional $B(H)$'s with uniform isomorphism constants. 
   This is described in \S~\ref{sec13}
and \S~\ref{sec13bis}. That result was surprising because it had no  counterpart in the Banach space context, where nothing like that holds for subspaces of the space $c_0$ of scalar sequences tending to zero (of which  $K(H)$ is a non-commutative analogue). 
%The main corollary was  the existence of  at least two inequivalent $C^*$-norms on $B(H)\otimes B(H)$.
However, the main result was somewhat hybrid: it assumed complete boundedness but only concluded to the existence of a bounded extension from $E\times F$ to $B(H)\times B(H)$. This   was resolved in \cite{PS}. There a characterization was found for  c.b.\ bilinear forms on $E\times F$ with $E,F$ exact. 
Going back to our earlier example,  when $E= F=K(H)$, the c.b.\ bilinear forms on $K(H)\times K(H)$
are those that can be written as a sum of only two forms of the first and second  type.
Curiously however, the associated factorization did not go through the canonical self-dual space $OH$ ---as one would have expected---but instead through a different Hilbertian space denoted by $R\oplus C$.
The space $R$ (resp. $C$) is the space
of all row (resp. column) matrices in $B(\ell_2)$,
and the space $R\oplus C$ is simply defined
as the subspace $R\oplus C\subset B(\ell_2)\oplus B(\ell_2)$ 
where  $B(\ell_2)\oplus B(\ell_2)$ is viewed as a $C^*$-subalgebra
(acting diagonally)  of $B(\ell_2\oplus \ell_2)$. The spaces
$R$, $C$ and $R\oplus C$ are examples of exact operator spaces.
The operator space GT from  \cite{ PS} says that, assuming $E,F$ exact, any c.b. linear map
$u\colon \ E\to F^*$ factors (completely boundedly) through $R\oplus C$.
 In case $E,F$ were $C^*$-algebras the result established a conjecture formulated 10 years earlier by Effros--Ruan and Blecher. This however was restricted to exact $C^*$-algebras (or to suitably approximable bilinear forms). But in \cite{HM2}, Haagerup and Musat found a new approach that removed all restriction. Both \cite{HM2,PS} 
have in common that they crucially use 
a kind of non-commutative probability space defined on von~Neumann algebras that do not admit any non-trivial trace. These are called
``Type III'' von~Neumann algebras.
In \S~\ref{sec14}, we give an almost self-contained proof of the operator space GT, based on \cite{HM2} but assuming no knowledge of Type III and hopefully much more accessible to a non-specialist. We also manage to incorporate
in this approach the case  of c.b.
bilinear forms on $E\times F$ with $E,F$ exact operator spaces (from \cite{PS}), which was not
covered in \cite{HM2}.

{\bf Tsirelson's bound} In \S \ref{sec15}, we describe Tsirelson's discovery of the close relationship between Grothendieck's inequality (i.e.\ GT) and Bell's inequality. The latter was crucial to put to the test the Einstein--Podolsky--Rosen (EPR) framework of ``hidden variables'' proposed as a sort of substitute to quantum mechanics. Using Bell's ideas, experiments were made (see \cite{As1,As2}) to verify the presence of a certain ``deviation'' that invalidated the EPR conception. What Tsirelson observed is that the Grothendieck constant could be interpreted as an upper bound for the ``deviation'' in the (generalized) Bell inequalities.
Moreover, there would be no deviation  if the Grothendieck constant was equal to 1 !
 This corresponds to an experiment with essentially two independent (because very distant) observers, and hence to the tensor product of two spaces. When three very far apart 
(and hence independent) observers are present, the quantum setting leads to a triple tensor product, whence the question whether there is a trilinear version of GT.
We present the recent counterexample from \cite{J+}
to this  ``trilinear GT"  in \S \ref{sec16}. We   follow the same route as  \cite{J+}, but by using a different
technical ingredient, we are able to include a rather short self-contained  proof.

Consider an $n\times n$ matrix $[a_{ij}]$ with   real entries. Following Tsirelson \cite{Ts3}, we say that $[a_{ij}]$ is a quantum correlation matrix if there are  self-adjoint operators $A_i,B_j$ on a Hilbert space $H$ with $\|A_i\|\le 1$, $\|B_j\|\le 1$ and $\xi$ in the unit sphere of $H\otimes_2 H$ such that
\begin{equation}\label{eqg4}
\forall i,j=1,\ldots, n\qquad\qquad a_{ij} = \langle (A_i\otimes B_j)\xi,\xi\rangle.
\end{equation}
 If in addition   the operators 
 $\{A_i\mid 1\le i\le n\}$ and $\{ B_j\mid 1\le j\le n\}$ all  commute
  then $[a_{ij}]$ is called a classical correlation matrix. In that case it is easy to see that there is a ``classical'' probability space $(\Omega,{\cl A},{\bb P})$ and   real  valued random variables $A_i,B_j$ in the unit ball of $L_\infty$ such that
\begin{equation}\label{eqg5}
 a_{ij} = \int A_i(\omega) B_j(\omega)\ d{\bb P}(\omega).
\end{equation}
As observed by Tsirelson,  GT implies that any real matrix of the form \eqref{eqg4} can be written in the form \eqref{eqg5} \emph{after division} by $K^{\bb R}_G$ and this is the best possible constant (valid for all $n$).  This is precisely what \eqref{eq0.3} below says in the real case:
Indeed, \eqref{eqg4} (resp. \eqref{eqg5}) holds iff the norm of  $\sum a_{ij} e_i\otimes e_j$  in $\ell_\infty^n  \otimes_{H} \ell_\infty^n$
(resp.  $\ell_\infty^n \stackrel{\wedge}\otimes \ell_\infty^n$) is less than 1
(see the proof of Theorem~\ref{ts1} below for the identification of \eqref{eqg4}  with the unit ball of $\ell_\infty^n  \otimes_{H} \ell_\infty^n$).  In \cite{Ts3}, Tsirelson, observing that in \eqref{eqg4}, $A_i\otimes 1$ and $1\otimes B_j$ are commuting operators on ${\cl H} = H\otimes_2 H$, considered the following generalization of \eqref{eqg4}:
\begin{equation}\label{eqg6}
 \forall i,j=1,\ldots, n\qquad \qquad a_{ij} = \langle  X_iY_j\xi,\xi\rangle
\end{equation}
where $X_i,Y_j\in B({\cl H})$ with $\|X_i\|\le 1$ $\|Y_j\|\le 1$ are
self-adjoint operators such that $X_iY_j = Y_jX_i$  for all $i,j$ and $\xi$ is in the unit sphere of ${\cl H}$. Tsirelson \cite[Th. 1]{Ts1} or  \cite[Th. 2.1]{Ts2} proved that \eqref{eqg4} and \eqref{eqg6} are   the same
(for real matrices).  He observed that since either set of matrices determined by  \eqref{eqg4} or \eqref{eqg6}
is closed and convex, it suffices to prove that the polar sets coincide. This is precisely
what is proved in Theorem~\ref{ts1} below.
In \cite{Ts3}, Tsirelson went further and   claimed without proof the equivalence of
an extension of  \eqref{eqg4} and \eqref{eqg6}
to the case when $A_i,B_j$  and $X_i,Y_j$ are replaced by certain operator valued probability measures
on an arbitrary finite set. However, he 
later on  emphasized that he completely overlooked a serious approximation difficulty, and he advertised
this  as   problem 33 (see \cite{Ts0})  on a website (http://www.imaph.tu-bs.de/qi/problems/) devoted to quantum information theory, see \cite{SW} as a substitute for the website.

 {\bf  The Connes-Kirchberg  problem} As it turns out (see \cite{JuPa,Fr})   Tsirelson's problem   is
 (essentially)  equivalent to one of the most famous ones   in von~Neumann algebra theory going back to Connes's paper \cite{Co}. The Connes problem can be stated as follows:\\
The non-commutative analogue of a probability measure on a  von Neumann algebra
$M\subset B(H)$ (assumed weak$^*$-closed) is a weak$^*$-continuous positive linear
functional $\tau\colon M \to \CC$ of norm 1, such that $\tau(1)=1$ and that is ``tracial", i.e. such that
$\tau (xy)=\tau (yx)$ for all $x,y\in M$. We will call this a non-commutative probability
on $M$. The Connes problem
asks whether any such  non-commutative probability
can be approximated by (normalized) matricial traces. More precisely, consider two unitaries $U,V$ in   $M$, can we find nets $(U^\alpha)$ $(V^\alpha)$ of unitary matrices of size $N(\alpha)\times N(\alpha)$ such that 
\[
 \tau(P(U,V)) = \lim_{\alpha\to \infty} \frac1{N(\alpha)} \text{ tr}(P(U^\alpha,V^\alpha))
\]
for any polynomial $P(X,Y)$ (in noncommuting variables $X,Y$)?

\n Note we can restrict to pairs of unitaries by a
well known matrix trick (see \cite[Cor. 2]{WW}).

In \cite{Kir}, Kirchberg found many striking equivalent reformulations of this problem, 
involving the unicity of certain $C^*$-tensor products. A Banach algebra norm $\alpha$ on the algebraic tensor product $A\otimes B$ of two $C^*$-algebras
(or on any  $*$-algebra) is called a  $C^*$-norm
if $\alpha(T^*)=\alpha(T)$ and $\alpha(T^*T)=\alpha(T)^2$ for any $T\in A\otimes B$.
Then the completion of $(A\otimes B, \alpha)$ is a $C^*$-algebra. 
It is known that there is a minimal and a maximal $C^*$-norm on $A\otimes B$. The associated $C^*$-algebras (after completion) are denoted
by $A\otimes_{\min} B$ and $A\otimes_{\max}  B$.
Given a discrete group $G$, there is a maximal $C^*$-norm on the group algebra ${\bb C}[G]$ and, after completion, this gives rise the (``full" or ``maximal") $C^*$-algebra of $G$.
Among several of Kirchberg's deep equivalent reformulations of the Connes problem,  this one stands out: \ Is there a unique $C^*$-norm on the tensor product $C\otimes C$ when $C$ is the (full) $C^*$-algebra of the free group ${\bb F}_n$ with $n\ge 2$ generators? The connection with GT comes through the generators:\ if $U_1,\ldots, U_n$ are the generators of ${\bb F}_n$ viewed as sitting in $C$, then $E = \text{span}[U_1,\ldots, U_n]$ is ${\bb C}$-isometric to ($n$-dimensional) $\ell_1$ and GT tells us that the minimal and maximal $C^*$-norms of $C\otimes C$ are $K^{\bb C}_G$-equivalent on $E\otimes E$.\\
 In addition to this link
with $C^*$-tensor products, the operator space version of GT has led to the first proof in \cite{JP} that $B(H) \otimes B(H)$ admits at least 2 inequivalent $C^*$-norms.
We describe some of these results connecting GT to $C^*$-tensor products and the Connes--Kirchberg problem in \S \ref{sec10}.

{ \bf GT in Computer Science} Lastly, in \S \ref{sec17} we briefly describe the recent surge of interest in GT among computer scientists, apparently triggered by   the idea  (\cite{AMMN}) to attach a Grothendieck inequality
(and hence a Grothendieck constant) to any  (finite) graph. The original GT
corresponds to the case of bipartite graphs. The motivation for the extension lies in   various algorithmic applications of the related computations. Here is a rough glimpse of the connection with GT:
When dealing with certain ``hard"  optimization problems of a  specific kind (``hard" here means time consuming), computer scientists
have a way to replace them by a companion problem that can be solved much faster using semidefinite programming. The companion problem
is then called the semidefinite ``relaxation"   of the original one. 
For instance, consider a finite graph $G=(V,E)$, 
and a real matrix $[a_{ij}]$ indexed by $V\times V$. We propose to compute
$$(I)=\max\{\sum\nolimits_{\{i,j\}\in E} a_{ij} s_i s_j\mid\ s_i=\pm 1, s_j=\pm 1\}.$$
In general, computing such a  maximum is hard, but the relaxed companion problem
is   to compute
$$(II)=\max\{\sum\nolimits_{\{i,j\}\in E} a_{ij} \langle x_i, y_j\rangle\mid\ x_i\in B_H, y_j\in B_H\}$$
where $B_H$ denotes the unit ball in   Hilbert space $H$. The latter is much easier: It
can be solved (up to an arbitrarily small additive error) in polynomial time by a well known method called the ``ellipsoid method" (see \cite{GLS1}). \\
The Grothendieck constant of the graph is defined as the best $K$
such that $(II)\le K  (I)  $.
Of course $(I)\le (II)$. Thus the Grothendieck
constant is precisely the maximum ratio $\frac {\rm relaxed (I)}{(I)}$. 
When $V$ is the disjoint union of two copies  $S'$ and $S''$ of $  [1,\ldots, n]$ and $E$ is the union of $S'\times S''$ and $S''\times S'$ (``bipartite" graph),
then 
GT (in the real case) says precisely that 
$(II)\le K_G (I)  $ (see Theorem \ref{thm1.2bis}), so the constant $K_G$ is the maximum Grothendieck
 constant for all bipartite graphs.
 Curiously, the value of these constants can also be connected
 to the P=NP problem. We merely glimpse into that aspect 
in  \S \ref{sec17}, and refer the reader   to the references
 for a more serious exploration. \\

  \n {\bf General background and notation}
  A Banach space is a complete normed space over ${\bb R}$ or ${\bb C}$. 
The 3 fundamental  Banach spaces in \cite{Gr1}  (and this paper) are $L_2$ (or any Hilbert space), $L_\infty$ and $L_1$.
By ``an $L_p$-space" we mean any space of the form $L_p(\Omega,{\cl A}, \mu)$
associated to some measure space $(\Omega,{\cl A}, \mu)$. Thus $\ell_\infty$
(or its $n$-dimensional analogue denoted by $\ell_\infty^n$) is an $L_\infty$-space. 
We denote by $C(S)$ the space of continuous functions on a compact set $S$
equipped with the sup-norm.
Any $L_\infty$-space is isometric to   a $C(S)$-space but not conversely.
However if $X=C(S)$ then for any $\vp >0$,  $X$ can be written as the closure of the union
of an increasing net  of finite dimensional subspaces $X_i \subset X$ 
such that each $X_i $ is $(1+\vp)$-isometric to a finite dimensional 
$\ell_\infty$-space. 
Such spaces are called ${\cl L}_{\infty,1}$-spaces (see \cite{LP}). 
From  this finite dimensional viewpoint,
a $C(S)$-space behaves like an $L_\infty$-space. This explains why
  many   statements below hold for either class.

Any Banach space embeds isometrically into a $C(S)$-space
(and into an $L_\infty$-space): just consider the mapping taking an element
to the function it defines on the dual unit ball. Similarly,
any Banach space is a quotient of an $L_1$-space. Thus $L_\infty$-spaces
(resp. $L_1$-spaces) are ``universal" respectively for  embeddings (resp.  quotients).
In addition, they possess
a special extension (resp. lifting) property: Whenever    $X\subset X_1$
is a subspace of a Banach space $X_1$, any operator $u\colon\ X\to L_\infty$
extends to an operator    $u_1\colon\ X_1\to L_\infty$ with the same norm.
The lifting property for   $\ell_1$-spaces  is similar.
Throughout the paper  $\ell^n_p$ designates ${\bb K}^n$ 
 (with ${\bb K}={\bb R}$ or ${\bb C}$)  equipped with the norm $\|x\| = (\Sigma|x_j|^p)^{1/p}$ and $\|x\| = \max|x_j|$ when $p=\infty$.  Then, when
$S = [1,\ldots, n]$, we have $C(S) = \ell^n_\infty$ and $C(S)^* = \ell^n_1$. 

More generally, for $1\le p\le \infty$ and  $1\le \lambda< \infty$, a Banach space $X$
is called a ${\cl L}_{p,\lambda}$-space if it can be rewritten, for each fixed $\vp>0$, as \begin{equation}\label{script}X =\ovl{\bigcup\limits_\alpha X_\alpha}\end{equation} where $(X_\alpha)$ is a net (directed by inclusion) of finite dimensional subspaces of $X$ such that, for each $\alpha,X_\alpha$ is $(\lambda+\vp)$-isomorphic to $\ell^{N(\alpha)}_p$ where $N(\alpha) = \dim(X_\alpha)$.
Any space $X$  that is   a ${\cl L}_{p,\lambda}$-space for some $1\le \lambda< \infty$
is called a ${\cl L}_{p}$-space. See \cite{JL} for more background.

There is little doubt that Hilbert spaces are central, but Dvoretzky's famous theorem
that any infinite dimensional Banach space contains almost isometric copies
of any finite dimensional Hilbert space makes it all the more striking.
As for  $L_\infty$ (resp. $L_1$), their central r\^ole is attested by their universality and   their extension (resp. lifting) property.
 Of course, 
by $L_2$, $L_\infty$ and $L_1$ we think here of infinite dimensional spaces, but
actually, the discovery that the finite dimensional setting is crucial to the understanding of many features of the structure of infinite dimensional Banach spaces is another visionary innovation of the r\'esum\'e
(Grothendieck even conjectured explicitly Dvoretzky's 1961 theorem in the shorter article \cite[p. 108]{Gr1+} that follows
the r\'esum\'e).

\section{Classical GT}\label{sec1}

In this and the next section, we take the reader on a tour of the many reformulations of GT.
\begin{thm}[Classical GT/factorization]\label{thm1.1}
Let $S,T$ be   compact sets. For any bounded bilinear form $\varphi\colon \ C(S)\times C(T)\to {\bb K}$ (here ${\bb K} = {\bb R}$ or ${\bb C}$) there are probabilities $\lambda$ and $\mu$, respectively on $S$ and $T$, such that
\begin{equation}\label{eq1.1}
\forall(x,y)\in C(S)\times C(T)\quad |\varphi(x,y)|\le K\|\varphi\|\left(\int_S|x|^2 d\lambda\right)^{1/2} \left(\int_T |y|^2 d\mu\right)^{1/2}
\end{equation}
where $K$ is a numerical constant, the best value of which is denoted by $K_G$, more precisely by $K^{\bb R}_G$ or $K^{\bb C}_G$ depending whether ${\bb K} = {\bb R}$ or ${\bb C}$.\\
Equivalently, the linear map $\tilde \varphi\colon \ C(S)\to C(T)^*$ associated to $\varphi$ admits a factorization of the form
$\tilde\varphi = J_{\mu}^*
   u  J_{\lambda}$ where $J_{\lambda}\colon \ C(S)\to L_2(\lambda)$
   and $J_{\mu}\colon \ C(T)\to L_2(\mu)$ are the canonical (norm 1) inclusions and $u\colon \ L_2(\lambda)\to L_2(\mu)^*$ is a bounded linear operator
   with $\|u\|\le K \| \varphi\|$.
\end{thm}

For any   operator $v\colon\ X \to Y$, we denote
\begin{equation}\label{eq0.2.1}
 \gamma_2(v) = \inf\{\|v_1\| \|v_2\|\}
\end{equation}
where the infimum runs over all Hilbert spaces $H$ and all possible factorizations of $  v$ through $H$:
\[
 v\colon \ X \overset{\sst v_2}{\longrightarrow} H \overset{\sst v_1}{\longrightarrow} Y
\]
 with $v = v_1v_2$.\\
 Note that any $L_\infty$-space is isometric to $C(S)$ for some $S$,
and any $L_1$-space embeds isometrically into its bidual,
and hence embeds into a space of the form $C(T)^*$. 
Thus we may state:

\begin{cor}\label{cor1} Any bounded linear map $v\colon \ C(S)\to C(T)^*$ or any bounded linear map $v\colon \ L_\infty \to L_1$
(over arbitrary measure spaces)  factors through a Hilbert space. More precisely, we have
$$\gamma_2(v)\le \ell \|v\|$$
where $\ell$ is a numerical constant with $\ell\le K_G$. 
\end{cor}

By a Hahn--Banach type argument (see \S \ref{sechb}), the preceding theorem is equivalent to the following one:

\begin{thm}[Classical GT/inequality]\label{thm1.2}
For any $\varphi\colon \ C(S)\times C(T) \to {\bb K}$ and for any finite sequence $(x_j,y_j)$ in $C(S)\times C(T)$ we have 
\begin{equation}\label{eq1.2}
 \left|\sum \varphi(x_j,y_j)\right| \le K\|\varphi\| \left\|\left(\sum |x_j|^2\right)^{1/2}\right\|_\infty \left\|\left( \sum |y_j|^2\right)^{1/2}\right\|_\infty.
\end{equation}
(We denote $\|f\|_\infty = \sup\limits_S |f(s)|$ for $f\in C(S)$.) Here again 
\[
K_{\text{\rm best}} = K_G.
\]
\end{thm}
Assume $S=T= [1,\ldots, n]$. Note that $C(S) = C(T) = \ell^n_\infty$. Then we obtain the formulation put forward by Lindenstrauss and Pe\l czy\'nski in \cite{LP}, perhaps the most ``concrete'' or elementary of them all:

\begin{thm}[Classical GT/inequality/discrete case]\label{thm1.2bis}
Let $[a_{ij}]$ be an $n\times n$ scalar matrix  $(n\ge 1)$ such that
\begin{equation}\label{eq1.2prebis}
\forall \alpha,\beta\in {\bb K}^n\quad
\left|\sum a_{ij}\alpha_i\beta_j\right| \le \sup_i|\alpha_i| \sup_j|\beta_j|.\qquad { }\qquad { }\qquad { }\end{equation}
Then for any Hilbert space $H$ and any $n$-tuples $(x_i), (y_j)$ in $H$ we have
\begin{equation}\label{eq1.2bis}
 \left|\sum a_{ij}\langle x_i,y_j\rangle\right| \le K\sup\|x_i\| \sup\|y_j\|.
\end{equation}
Moreover the best $K$ (valid for all $H$ and all $n$) is equal to $K_G$.
\end{thm}

\begin{proof}
We will prove that this is equivalent to
the preceding Theorem \ref{thm1.2}. Let $S=T=[1,\ldots,n]$. Let $\varphi\colon \ C(S)\times C(T)\to {\bb K}$ be the bilinear form associated to $[a_{ij}]$. Note that (by our assumption) $\|\varphi\|\le 1$. 
We may clearly assume that $\dim(H)<\infty$. Let $(e_1,\ldots, e_d)$ be an orthonormal basis of $H$. Let
\[
X_k(i) = \langle x_i,e_k\rangle \quad\text{and}\quad Y_k(j) = \ovl{\langle y_j,e_k\rangle}.
\]
Then
\begin{gather*}
\sum a_{ij}\langle x_i,y_j\rangle = \sum_k \varphi(X_k,Y_k),\\
\sup\|x_i\| = \left\|\left(\sum |X_k|^2\right)^{1/2}\right\|_\infty\quad \text{and}\quad \sup\|y_j\| = \left\|\left(\sum |Y_k|^2\right)^{1/2}\right\|_\infty.
\end{gather*}
Then it is clear that Theorem \ref{thm1.2} implies Theorem \ref{thm1.2bis}. The converse is also true. To see that one should view $C(S)$ and $C(T)$ as ${\cl L}_\infty$-spaces (with constant 1), i.e. ${\cl L}_{\infty,1}$-spaces as in \eqref{script}.\end{proof}

In harmonic analysis, the classical Marcinkiewicz--Zygmund inequality says that any bounded linear map $T\colon \ L_p(\mu)\to L_p(\mu')$ satisfies the following inequality (here $0<p\le \infty$)
\begin{equation}
\left\|\left(\sum |T x_j|^2\right)^{1/2}\right\|_p \le \|T\| \left\|\left(\sum |x_j|^2\right)^{1/2}\right\|_p. \tag*{$\forall n~~\forall x_j\in L_p(\mu)~~ (1\le j\le n)$}
\end{equation}
Of course $p=2$ is trivial. Moreover,
the case $p=\infty$ (and hence $p=1$ by duality)
is obvious  because we have the following linearization of the ``square function norm":
$$\left\|\left(\sum |x_j|^2\right)^{1/2}\right\|_\infty=
\sup\left\{ \left\|\sum  a_j x_j\right\|_\infty \mid a_j\in \KK, \ \sum |a_j|^2\le 1\right\}.$$
The remaining case $1<p<\infty$ is an easy consequence of Fubini's Theorem and the isometric linear embedding of $\ell_2$ into $L_p$ provided by independent standard Gaussian variable (see \eqref{eq1.6+} below): Indeed, if $(g_j)$ is an independent, identically distributed (i.i.d. in short) sequence of Gaussian normal
variables relative to  a probability ${\bb P}$, we have for any scalar sequence $(\lambda_j)$ in $\ell_2$
\begin{equation}\label{eqgauss}
(\sum |\lambda_j|^2)^{1/2}=\|g_1\|^{-1}_p \|\sum g_j \lambda_j\|_p.
\end{equation}
     Raising this  to the $p$-th power (set $\lambda_j=x_j(t)$) and integrating with respect to $\mu(dt)$,
we find for any
 $(x_j)$ in $L_p(\mu)$
$$\left\|\left(\sum |x_j|^2\right)^{1/2}\right\|_p=\|g_1\|_p^{-1}\left(\int \left\|\sum g_j(\omega) x_j\right\|^p_p d{\bb P}(\omega)\right)^{1/p}.$$

The preceding result can be reformulated in a similar fashion:

\begin{thm}[Classical GT/Marcinkiewicz--Zygmund style]\label{thm1.3}
For any pair of measure spaces $(\Omega,\mu), (\Omega',\mu')$ and any bounded linear map $T\colon \ L_\infty(\mu) \to L_1(\mu')$ we have
 $\forall n~\forall x_j\in L_\infty(\mu)$ $(1\le j\le n)$
\begin{equation}\label{eq1.3}
 \left\|\left(\sum |T x_j|^2\right)^{1/2}\right\|_1 \le K\|T\| \left\|\left(\sum |x_j|^2\right)^{1/2} \right\|_\infty.
\end{equation}
Moreover here again $K_{\text{\rm best}} = K_G$.
\end{thm}

\begin{proof}
The modern way to see that Theorems \ref{thm1.2} and \ref{thm1.3} are equivalent is to note that both results can be reduced by a routine technique to the finite case, i.e.\ the case $S=T=[1,\ldots, n] = \Omega = \Omega'$ with $\mu=\mu'=$ counting measure. Of course the constant $K$ should \emph{not} depend on $n$. In that case we have $C(S)= L_\infty(\Omega)$ and $L_1(\mu') = C(T)^*$ isometrically, so that \eqref{eq1.2} and \eqref{eq1.3} are immediately seen to be identical using the isometric identity (for vector-valued functions) $C(T;\ell^n_2)^* = L_1(\mu';\ell^n_2)$. Note that the reduction to the finite case owes a lot to the illuminating notion of ${\cl L}_\infty$-spaces (see \eqref{script} and \cite{LP}).\end{proof}
Krivine \cite{Kr3}  observed the following generalization of Theorem \ref{thm1.3} (we state this as a corollary, but it is clearly equivalent to the theorem and hence to GT).

\begin{cor}\label{cor1.4}
For any pair $\Lambda_1,\Lambda_2$ of Banach lattices and for any bounded linear $T\colon \ \Lambda_1\to \Lambda_2$ we have
\begin{equation}\label{eq1.4}
 \forall n~~\forall x_j\in \Lambda_1\quad (1\le j\le n)\qquad \left\|\left(\sum | T x_j|^2\right)^{1/2}\right\|_{\Lambda_2} \le K\|T\| \left\|\left(\sum |x_j|^2\right)^{1/2}\right\|_{\Lambda_1}.
\end{equation}
Again the best $K$ (valid for all pairs $(\Lambda_1,\Lambda_2)$ is equal to $K_G$.
\end{cor}
Here the reader may assume that $\Lambda_1,\Lambda_2$ are 
``concrete'' Banach lattices of functions over measure spaces say $(\Omega,\mu)$, $(\Omega',\mu')$ so that the notation $\|(\sum|x_j|^2)^{1/2}\|_{\Lambda_1}$ can be understood as the norm in ${\Lambda_1}$ of the function $\omega\to (\sum|x_j(\omega)|^2)^{1/2}$. But actually this also makes sense in the setting of ``abstract'' Banach lattices (see \cite{Kr3}).

Among the many applications of GT, the following \emph{isomorphic} characterization of Hilbert spaces is rather puzzling in view of the many related questions
(mentioned below)  that remain open to this day.

It will be convenient to use the Banach--Mazur ``distance'' between two Banach spaces $B_1,B_2$, defined as follows:
\[
 d(B_1,B_2) = \inf\{\|u\|\ \|u^{-1}\|\}
\]
where the infimum runs over all isomorphisms $u\colon \ B_1\to B_2$, and we set $d(B_1,B_2) = +\infty$ if there is no such isomorphism.

\begin{cor}\label{cor1.20}
 The following properties of a Banach space $B$ are equivalent:
\begin{itemize}
 \item[\rm (i)] Both $B$ and its dual $B^*$ embed isomorphically into an $L_1$-space.
\item[\rm (ii)] $B$ is isomorphic to a Hilbert space. 
\end{itemize}
More precisely, if $X\subset L_1$, and $Y\subset L_1$ are $L_1$-subspaces, with $B\simeq X$ and  $B^*\simeq Y$; then there is a Hilbert space $H$ such that
\[
 d(B,H) \le K_G d(B,X) d(B^*,Y).
\]
Here $K_G$ is $K^{\bb R}_G$ or $K^{\bb C}_G$ depending whether ${\bb K} = {\bb R}$ or ${\bb C}$.

\end{cor}

\begin{proof}
Using the Gaussian isometric embedding \eqref{eqgauss} with $p=1$ it is immediate that any Hilbert space say $H = \ell_2(I)$ embeds isometrically into $L_1$ and hence, since $H\simeq H^*$ isometrically, (ii) $\Rightarrow$~(i) is immediate. Conversely, assume (i). Let $v\colon \ B\hookrightarrow L_1$ and $w\colon \ B^* \hookrightarrow L_1$ denote the isomorphic embeddings. We may apply GT to the composition
\[
u = wv^*\colon \ L_\infty \overset{v^*}{\longrightarrow} B^* \overset{w}{\longrightarrow} L_1.
\]
By Corollary \ref{cor1}, $wv^*$ factors through a Hilbert space. Consequently, since $w$ is an embedding, $v^*$ itself must factor through a Hilbert space and hence, since $v^*$ is onto, $B^*$ (a fortiori $B$) must be isomorphic to a Hilbert space. The last assertion is then easy to check.
\end{proof}

Note that in (ii) $\Rightarrow$ (i), even if we assume $B,B^*$ both \emph{isometric} to subspaces of $L_1$, we only conclude that $B$ is \emph{isomorphic} to Hilbert space. The question was raised already by Grothendieck himself in \cite[p. 66]{Gr1}  whether one could actually conclude that $B$ is \emph{isometric} to a Hilbert space. Bolker also asked the same question
in terms of zonoids (a symmetric convex body is   a zonoid iff the 
normed space admitting the polar body for its unit ball embeds isometrically in $L_1$.)
This was answered \emph{negatively} by Rolf Schneider \cite{Sc} but only in the \emph{real case}:\ He produced $n$-dimensional counterexamples  over ${\bb R}$ for any $n\ge3$. Rephrased in geometric language, 
there are (symmetric) zonoids in ${\bb R}^n$ whose polar is a zonoid but that are not ellipsoids. Note that in the real case the dimension 2 is exceptional because any 2-dimensional space embeds isometrically into $L_1$ so there are obvious 2D-counterexamples (e.g.\ 2-dimensional $\ell_1$). But apparently (as pointed out by J. Lindenstrauss) the infinite dimensional case, both for ${\bb K} = {\bb R}$ or ${\bb C}$ is   open, and also
the complex case seems open in all dimensions.

%%%appli

\section{Classical GT with tensor products}\label{sec1bis}

Before Grothendieck, Schatten and von Neumann
 had already worked on   tensor products of Banach spaces (see \cite{Sch}). But
although Schatten did lay the foundation for the Banach case in   \cite{Sch},   
it is probably fair to say that   Banach space tensor products  really took off only after Grothendieck.

There are many norms that one can define on the \emph{algebraic} tensor product $X\otimes Y$ of two Banach spaces $X,Y$. Let $\alpha$ be such a norm. Unless one of $X,Y$ is finite dimensional, $(X\otimes Y,\alpha)$ is not complete, so we denote by $X\widehat\otimes_\alpha Y$ its completion.  We need to restrict attention to norms that have some minimal compatibility with tensor products, so we always impose $\alpha(x\otimes y) = \|x\|\|y\|$ for all $(x,y)$ in $X\times Y$ (these are called ``cross norms'' in \cite{Sch}). 
 We will mainly consider 3 such norms: $\|\ \|_\wedge,\|\ \|_\vee$ and $\|\ \|_H$,
 defined as follows.
 
 By the triangle inequality, there is  obviously a largest cross norm defined  
   for any
\begin{align}\label{eq0.1}
 t &= \sum\nolimits^n_1 x_j\otimes y_j \in X\otimes Y\\
\label{eq0.1-}
  {\rm by}\qquad \|t\|_\wedge &= \inf\left\{\sum \|x_j\|\|y_j\|\right\}\quad \text{(``projective norm'')}\\
\label{eq0.1+} {\rm or\  equivalently}\quad
 \|t\|_\wedge &=\inf\left\{(\sum \|x_j\| ^2)^{1/2} (\sum\|y_j\|^2)^{1/2}\right\} \end{align}
where the infimum runs over all possible representations of the form \eqref{eq0.1}.\\
Given two Banach spaces $X,Y$, the completion of $(X\otimes Y, \|\cdot\|_\wedge)$
is denoted by $X\widehat\otimes Y$. 
Grothendieck called it
  the projective tensor product of $X,Y$.

Its characteristic property (already in \cite{Sch}) is the isometric identity
\[
 (X\widehat\otimes Y)^* = {\cl B}(X\times Y)
\]
where ${\cl B}(X\times Y)$ denotes the space of bounded bilinear forms on $X\times Y$. 

Furthermore the norm
\begin{equation}\label{eq0.1p}
\|t\|_\vee = \sup\left\{\left|\sum x^*(x_j)y^*(y_j)\right| \ \Bigg| \ x^*\in B_{X^*}, y^*\in B_{Y^*}\right\} \quad \text{(``injective norm'')}
\end{equation}
is the smallest one over all norms that are cross-norms as well
as their dual norm (Grothendieck called those ``reasonable" norms).
Lastly we define
\begin{equation}\label{eq0.1h} \|t\|_H = \inf\left\{\sup_{x^*\in B_{X^*}} \left(\sum |x^*(x_j)|^2\right)^{1/2} \sup_{y^*\in B_{Y^*}} \left(\sum|y^*(y_j)|^2\right)^{1/2}\right\}
\end{equation}
where again the infimum runs over all possible representions \eqref{eq0.1}.
We have
\begin{equation}\label{eq0.0} 
 \|t\|_\vee \le \|t\|_H \le \|t\|_\wedge
\end{equation}
 Let $\widetilde t\colon \ X^*\to Y$ be the linear mapping associated to $t$, so that $\widetilde t(x^*) = \sum x^*(x_j)y_j$. Then  
\begin{equation}\label{eq0.2}
\|t\|_\vee = \|\widetilde t\|_{B(X,Y)}\quad {\rm and}\quad
 \|t\|_H = \gamma_2(\widetilde t),\end{equation}
where $\gamma_2$ is as defined in \eqref{eq0.2.1}.  

One of the great methodological innovations of ``the R\'esum\'e'' was the systematic use of the duality of tensor norms (already considered in \cite{Sch}):\ Given a norm $\alpha$ on $X\otimes Y$ one defines $\alpha^*$ on $X^*\otimes Y^*$ by setting
\begin{equation}
\alpha^*(t') = \sup\{|\langle t,t'\rangle|\mid t\in X\otimes Y, \ \alpha(t)\le 1\}.\tag*{$\forall t'\in X^*\otimes Y^*$}
\end{equation}
In the case $\alpha(t) = \|t\|_H$, Grothendieck studied the dual norm $\alpha^*$ and used the notation $\alpha^*(t) = \|t\|_{H'}$.
We have
$$\|t\|_{H'}=   \inf\left\{(\sum \|x_j\| ^2)^{1/2} (\sum\|y_j\|^2)^{1/2}\right\} $$
where the infimum runs over all  finite sums $ \sum_1^n  x_j\otimes y_j\in X\otimes Y$ such that
$$|\langle t, x^*\otimes y^*\rangle|\le (\sum |x^*(x_j)|^2)^{1/2}  (\sum|y^*(y_j)|^2)^{1/2}. \leqno \forall (x^*, y^*)\in X^*\times Y^*$$

 It is easy to check that if $\alpha(t)=\|t\|_\wedge$ then $\alpha^*(t') = \|t'\|_\vee$. Moreover, if either $X$ or $Y$ is finite dimensional and $\beta(t) = \|t\|_\vee$ then $\beta^*(t') = \|t'\|_\wedge$. So, at least in the finite dimensional setting the projective and injective norms $\|~~\|_\wedge$ and $\|~~\|_\vee$ are in perfect duality (and so are
 $\|~~\|_H$ and $\|~~\|_{H'}$).\\

  Let us return to the case when $S=[1,\cdots,n]$.
  Let us denote by $(e_1,\ldots, e_n)$ the canonical basis of $\ell^n_1$ and by $(e^*_1,\ldots, e^*_n)$ the biorthogonal basis in $\ell^n_\infty = (\ell^n_1)^*$.
  Recall $C(S) = \ell^n_\infty= (\ell^n_1)^*$ and $C(S)^* = \ell^n_1$.
Then $t\in C(S)^*\otimes C(S)^*$ (resp.\ $t'\in C(S)\otimes C(S))$ can be identified with a matrix $[a_{ij}]$ (resp.\ $[a'_{ij}]$) by setting $$t = \Sigma a_{ij}e_i\otimes e_j \quad {\rm (resp.}\ t' = \Sigma a'_{ij} \otimes e^*_i\otimes e^*_j). $$
One then checks easily from the definitions that
 (recall  $\KK=\RR$ or $\CC$)
\begin{equation}\label{eq0.4}
 \|t\|_\vee = \sup\left\{\left|\sum a_{ij}\alpha_i \beta_j\right|\ \bigg|\ \alpha_i,\beta_j\in {\bb K}, \ \sup\nolimits_i|\alpha_i| \le 1, \sup\nolimits_j|\beta_j|\le 1\right\}.
\end{equation}
Moreover,
\begin{equation}\label{eq0.5}
 \|t'\|_H = \inf\{\sup_i\|x_i\|\sup_j\|y_j\|\}
\end{equation}
where the infimum runs over all Hilbert spaces $H$ and all $x_i,y_j$ in $H$ such that $a'_{ij} = \langle x_i,y_j\rangle$ for all $i,j=1,\ldots, n$. By duality, this implies
\begin{equation}\label{eq0.6}
 \|t\|_{H'}= \sup\left\{\left|\sum a_{ij} \langle x_i,y_j\rangle\right|\right\}
\end{equation}
where the supremum is over all Hilbert spaces $H$ and all $x_i,y_j$ in the unit ball of $H$.

 With this notation, GT in the form  \eqref{eq1.2bispre+} can be restated as follows:\ there is a constant $K$ such that for any $t$ in $L_1\otimes L_1$ (here $L_1$ means $\ell_1^n$) we have
 \begin{equation}\label{eq0.3b}
 \|t\|_{H'} \le K\|t\|_\vee.
\end{equation}
Equivalently by duality the theorem says that  
for any $t'$ in $C(S)\otimes C(S)$ (here $C(S)$ means $\ell_\infty^n$) we have
\begin{equation}\label{eq0.3}
 \|t'\|_\wedge \le K\|t'\|_H.
\end{equation}
The best constant in either \eqref{eq0.3} (or its   dual form \eqref{eq0.3b}) is   the Grothendieck constant $K_G$. \\
Lastly, although we restricted to the finite dimensional case for maximal simplicity,  \eqref{eq0.3b}  (resp.   \eqref{eq0.3}) remains valid for any $t\in X \otimes Y$   (resp.  $t'\in X \otimes Y$ ) when $X,Y$ are arbitrary $L_1$-
spaces (resp. arbitrary $L_\infty$-spaces or $C(S)$
for a     compact set $S$), or even more generally for arbitrary ${\cl L}_{1,1}$-spaces
spaces  (resp. arbitrary ${\cl L}_{\infty,1}$-spaces) in the sense of \cite{LP} (see \eqref{script}).
Whence    the following dual reformulation of  Theorem \ref{thm1.2}:

\begin{thm}[Classical GT/predual formulation]\label{thm1.5}
For any $F$ in $C(S)\otimes C(T)$ we have 
\begin{equation}\label{eq1.5}
 \|F\|_\wedge \le K\|F\|_H
\end{equation}
and, in $C(S)\otimes C(T)$, \eqref{eq0.1h} becomes 
\begin{equation}\label{eq1.6}
 \|F\|_H = \inf\left\{\left\|\left(\sum |x_j|^2\right)^{1/2}\right\|_\infty \left\|\left(\sum |y_j|^2\right)^{1/2}\right\|_\infty\right\}
\end{equation}
with the infimum running over all $n$ and all possible representations of $F$ of the form
\[
 F = \sum\nolimits^n_1 x_j\otimes y_j,\quad (x_j,y_j)\in C(S)\times S(T).
\]
\end{thm}

\begin{rmk} By \eqref{eq0.0},
\eqref{eq0.3b} implies
\[
 \|t\|_{H'} \le K\|t\|_H.
\]
but   the latter inequality 
is much easier to prove than Grothendieck's, so that
it is often called ``the little GT"
 and for it the best constant denoted by $ k_G $ is known: it is equal to $\pi/2$  in the real case
 and $4/\pi$ in the complex one, see \S \ref{sec2}.
\end{rmk}

We will now prove Theorem \ref{thm1.5}, and hence all the preceding equivalent formulations of GT. Note that both $\|\cdot\|_\wedge$ and $\|\cdot\|_H$ are Banach algebra norms on $C(S)\otimes C(T)$, with respect to the pointwise product on $S\times T$, i.e.\ we have for any $t_1,t_2$ in $C(S)\otimes C(T)$
\begin{equation}\label{eq1.6-}
\|t_1\cdot t_2\|_\wedge \le \|t_1\|_\wedge \|t_2\|_\wedge\quad\text{and}\quad \|t_1\cdot t_2\|_H \le\|t_1\|_H \|t_2\|_H.
\end{equation}
Let $H=\ell_2$. Let $\{g_j\mid j\in {\bb N}\}$ be an i.i.d.\ sequence of standard Gaussian random variable on a probability space $(\Omega, {\cl A}, {\bb P})$. For any $x = \sum x_je_j$ in $\ell_2$ we denote $G(x)= \sum x_jg_j$. Note that
\begin{equation}\label{eq1.6+}
 \langle x,y\rangle_H= \langle G(x), G(y)\rangle_{L_2(\Omega,{\bb P})} .
\end{equation}
Assume ${\bb K}={\bb R}$. The following formula is crucial both to Grothendieck's original proof and to Krivine's: if $\|x\|_H=\|y\|_H =1$
\begin{equation}\label{eq1.7}
 \langle x,y\rangle = \sin\left(\frac\pi2 \langle\text{sign}(G(x)), \text{ sign}(G(y))\rangle\right).
\end{equation}

\begin{proof}[Grothendieck's proof of Theorem \ref{thm1.5} with $K={\rm sh}(\pi/2)$]
This is in essence the original proof. Note that Theorem \ref{thm1.5} and Theorem \ref{thm1.2} are obviously equivalent by duality. We already saw that Theorem \ref{thm1.2} can be reduced to Theorem \ref{thm1.2bis} by an approximation argument (based on ${\cl L}_\infty$-spaces). Thus it suffices to prove Theorem \ref{thm1.5} in case $S,T$ are finite subsets of the unit ball of $H$.

Then we may as well assume $H=\ell_2$. Let $F\in C(S)\otimes C(T)$. We view $F$ as a function on $S\times T$. Assume $\|F\|_H<1$. Then by definition of $\|F\|_H$, we can find elements $x_s,y_t$ in $\ell_2$ with $\|x_s\|\le 1$, $\|y_t\|\le 1$ such that
\begin{equation}
 F(s,t) = \langle x_s,y_t\rangle.\tag*{$\forall(s,t)\in S\times T$}
\end{equation}
By adding mutually orthogonal parts to the vectors $x_s$ and $y_t$, $F(s,t)$ does not change and we may assume $\|x_s\|= 1$, $\|y_t\|= 1$.
By \eqref{eq1.7} $F(s,t) = \sin(\frac\pi2\int \xi_s\eta_t\ d{\bb P})$ where $\xi_s = \text{sign}(G(x_s))$ and $\eta_t = \text{sign}(G(y_t))$. 

Let $k(s,t) = \int \xi_s\eta_t\ d{\bb P}$. Clearly $\|k\|_{C(S)\hat\otimes C(T)} \le 1$ this follows by approximating the integral by sums (note that, $S,T$ being finite, all norms are equivalent on $C(S)\otimes C(T)$ so this approximation is easy to check). Since $\|~~\|_\wedge$ is a Banach algebra norm (see \eqref{eq1.6-}) , the elementary identity
\[
 \sin(z) = \sum\nolimits^\infty_0 (-1)^m \frac{z^{2m+1}}{(2m+1)!}
\]
is valid for all $z$ in $C(S)\widehat\otimes C(T)$, and we have
\[
 \|\sin(z)\|_\wedge \le \sum\nolimits^\infty_0 \|z\|^{2m+1} ((2m+1)!)^{-1} = {\rm sh}(\|z\|_\wedge).
\]
Applying this to $z=(\pi/2)k$, we obtain
\[
 \|F\|_\wedge \le {\rm sh}(\pi/2).\qquad \qed
\]
\renewcommand{\qed}{}\end{proof}

\begin{proof}[Krivine's proof of Theorem \ref{thm1.5} with $K=\pi(2 \text{ Log}(1+\sqrt 2))^{-1}$]
Let $K = \pi/2a$ where $a>0$ is chosen so that ${\rm sh}(a)=1$, i.e.\ $a = \text{Log}(1+\sqrt 2)$. We will prove Theorem \ref{thm1.5} (predual form of Theorem \ref{thm1.2}). From what precedes, it suffices to prove this when $S,T$ are arbitrary \emph{finite} sets.

Let $F\in C(S)\otimes C(T)$. We view $F$ as a function on $S\times T$. Assume $\|F\|_H<1$. We will prove that $\|F\|_\wedge\le K$. Since $\|~~\|_H$ also is a Banach algebra norm
(see \eqref{eq1.6-})  we have
\[
 \|\sin(aF)\|_H \le {\rm sh}(a\|F\|_H) < {\rm sh}(a).
\]
By definition of $\|~~\|_H$ (after a slight correction, as before, to normalize the vectors), this means that there are $(x_s)_{s\in S}$ and $(y_t)_{t\in T}$ in the  unit sphere of $H$ such that
\[
 \sin(aF(s,t)) = \langle x_s,y_t\rangle.
\]
By \eqref{eq1.7} we have
\[
\sin(aF(s,t)) = \sin\left(\frac\pi2 \int \xi_s\eta_t\ d{\bb P}\right)
\]
where $\xi_s = \text{sign}(G(x_s))$ and $\eta_t = \text{sign}(G(y_t))$. Observe that $|F(s,t)| \le \|F\|_H <1$, and a fortiori, since $a<1$, $|aF(s,t)|<\pi/2$. Therefore we must have
\[
 aF(s,t) = \frac\pi2 \int \xi_s\eta_t\ d{\bb P}
\]
and hence $\|aF\|_\wedge \le \pi/2$, so that we conclude $\|F\|_\wedge \le \pi/2a$.
\end{proof}

We now give the Schur multiplier formulation of GT. By a Schur multiplier, we mean here a bounded linear map $T\colon \ B(\ell_2)\to B(\ell_2)$ of the form $T([a_{ij}])  = [\varphi_{ij}a_{ij}]$ for some (infinite) matrix $\varphi = [\varphi_{ij}]$. We then denote $T=M_\varphi$. For example, if $\varphi_{ij} =z'_iz''_j$ with $|z'_i|\le 1$, $|z''_j|\le 1$ then $\|M_\varphi\|\le 1$ (because $M_\varphi(a) =D'aD''$ where $D'$ and $D''$ are the diagonal matrices with entries $(z'_i)$ and $(z''_j)$).
Let us denote by ${\cl S}$ the class of all such ``simple'' multipliers. Let $\ovl{\text{conv}}({\cl S})$ denote the pointwise closure (on ${\bb N}\times{\bb N}$) of the convex hull of ${\cl S}$. Clearly $\varphi\in {\cl S}\Rightarrow\|M_\varphi\|\le 1$. Surprisingly, the converse is essentially true (up to a constant). This is one more form of GT:

\begin{thm}[Classical GT/Schur multipliers]\label{thm2.7bis}
If $\|M_\varphi\|\le 1$ then $\varphi\in K_G \ \ovl{\text{\rm conv}}({\cl S})$ and $K_G$ is the best constant satisfying this (by $K_G$ we mean $K^{\bb R}_G$ or $K^{\bb C}_G$ depending whether all matrices involved have real or complex entries).
\end{thm}

We will use the following (note that, by \eqref{eq0.5}, 
this implies in particular that $\|~~\|_H$ is a Banach algebra norm
on $\ell_\infty^n\otimes \ell_\infty^n$ since we obviously have $\|M_{\varphi\psi}\|=\|M_{\varphi}M_{\psi}\|\le  \|M_\varphi\|\|M_\psi\|$):

\begin{pro}\label{pro2.7ter}
We have $\|M_\varphi\|\le 1$ iff there are $x_i,y_j$ in the unit ball of Hilbert space such that $\varphi_{ij} = \langle x_i,y_j\rangle$.
\end{pro}
\begin{rmk}
Except for precise constants, both Proposition~\ref{pro2.7ter} and Theorem~\ref{thm2.7bis} are due to Grothendieck, but they were rediscovered several times, notably by John Gilbert and Uffe Haagerup. They can be essentially deduced from \cite[Prop.~7, p.~68]{Gr1} in the R\'esum\'e, but the specific way in which Grothendieck uses duality there introduces some extra numerical factors (equal to 2) in the constants involved, that were removed later on (in \cite{Kw1}).
\end{rmk}

\begin{proof}[Proof of Theorem \ref{thm2.7bis}]
Taking the preceding Proposition for granted, it is easy to complete the proof. 
Assume $\|M_\varphi\|\le 1$. 
Let $\varphi^n$ be the matrix equal to $\varphi$
in the upper $n\times n$ corner and to $0$ elsehere. 
Then $\|M_{\varphi^n}\|\le 1$ and obviously  
$\varphi^n\to \varphi$ pointwise. Thus, it suffices to show $\varphi \in K_G \text{ conv}({\cl S})$ when $\varphi$ is an $n\times n$ matrix. Then let $t' = \sum^n_{i,j=1} \varphi_{ij}e_i\otimes e_j \in \ell^n_\infty\otimes \ell^n_\infty$. If $\|M_\varphi\|\le 1$, the preceding Proposition implies by \eqref{eq0.5} that $\|t'\|_H\le 1$ and hence by \eqref{eq0.3} that $\|t'\|_{\wedge} \le K_G$. But by \eqref{eq0.1-}, $\|t'\|_{\wedge} \le K_G$ iff $\varphi\in K_G \text{ conv}({\cl S})$. 
A close examination of the proof shows that the best constant in Theorem~\ref{thm2.7bis} is equal to the best one in \eqref{eq0.3}.
\end{proof}

\begin{proof}[Proof of Proposition \ref{pro2.7ter}]
The if part is easy to check. To prove the converse, we may again assume (e.g. using ultraproducts) that $\varphi$ is an $n\times n$ matrix. Assume $\|M_\varphi\|\le 1$. By duality it suffices to  show that $|\sum \varphi_{ij}\psi_{ij}|\le 1$ whenever $\|\sum \psi_{ij} e_i\otimes e_j\|_{\ell^n_1\otimes_{H'}\ell^n_1}\le 1$. Let $t = \sum \psi_{ij} e_i\otimes e_j\in \ell^n_1 \otimes \ell^n_1$. We will use the fact that $\|t\|_{H'}\le 1$ iff $\psi$ admits a factorization of the form $\psi_{ij}= \alpha_i a_{ij}\beta_j$ with $[a_{ij}]$ in the unit ball of $B(\ell^n_2)$ and $(\alpha_i)$, $(\beta_j)$ in the unit ball of $\ell^n_2$.  Using this fact we obtain
\[
 \left|\sum \psi_{ij}\varphi_{ij}\right| = \left|\sum \lambda_i a_{ij}\varphi_{ij}\mu_j\right| \le \|[a_{ij}\varphi_{ij}]\| \le \|M_\varphi\| \|[a_{ij}]\|\le 1.
\]
The preceding factorization of $\psi$ requires a Hahn--Banach argument that we provide in Remark~\ref{rem4.3bis} below.
\end{proof}

\begin{thm}\label{thm1.8}
The constant   $K_G$ is the best constant $K$ such that, for any Hilbert space $H$, there is a probability space $(\Omega, {\bb P})$ and  functions
\[
 \Phi\colon \ H\to L_\infty(\Omega,{\bb P}),\quad \Psi\colon \ H\to L_\infty(\Omega,{\bb P})
\]
such that
\begin{equation}
\|\Phi(x)\|_\infty\le \|x\|,\quad \|\Psi(x)\|_\infty\le \|x\| \tag*{$\forall x\in H$}
\end{equation}
and
\begin{equation}\label{eq1.8}
\forall x,y\in H\qquad\qquad \frac1K\langle x,y\rangle = \int \Phi(x) \ovl{\Psi(y)}\ d{\bb P}
\end{equation} 
Note that depending whether $\KK=\RR$ or $\CC$
we use real or complex valued $L_\infty(\Omega,{\bb P})$.  Actually, we can find a compact set $\Omega$
equipped with a Radon probability measure and functions $\Phi,\Psi$ as above but taking values in the space $C(\Omega)$ of continuous (real or complex) functions on $\Omega$.
\end{thm}

\begin{proof}
We will just prove that this holds for the constant $K$ appearing in Theorem \ref{thm1.5}. It suffices to prove the existence of functions $\Phi_1,\Psi_1$ defined on the unit sphere of $H$ satisfying the required properties. Indeed, we may then set $\Phi(0) = \Psi(0)=0$ and
\[
\Phi(x) = \|x\|\Phi_1(x\|x\|^{-1}),\quad \Psi(y) = \|y\| \Psi_1(y\|y\|^{-1})
\]
and we obtain the desired properties on $H$. Our main ingredient is this:\ let ${\cl S}$ denote the unit sphere of $H$, let $S\subset {\cl S}$ be a finite subset, and let
\begin{equation}
F_S(x,y) = \langle x,y\rangle.\tag*{$\forall(x,y)\in S\times S$}
\end{equation}
Then $F_S$ is a function on $S \times S$ but 
we may view it as an element of $C(S)\otimes C(S)$. 
We will obtain \eqref{eq1.8} for all $x,y$ in $S$ and then use a limiting argument.

Obviously $\|F_S\|_H\le 1$. Indeed let $(e_1,\ldots, e_n)$ be an orthonormal basis of the span of $S$. We have
\[
 F_S(x,y) = \sum\nolimits^n_1 x_j\bar y_j
\]
and $\sup_{x\in S}(\sum|x_j|^2)^{1/2} = \sup_{y\in S}(\sum|y_j|^2)^{1/2}=1$ (since $(\sum|x_j|^2)^{1/2}=\|x\|=1$ for all $x$ in $\cl S$). Therefore, $\|F\|_H\le 1$ and Theorem \ref{thm1.5} implies
\[
 \|F_S\|_\wedge \le K.
\]
Let $C$ denote the unit ball of $\ell_\infty({\cl S})$ equipped with the weak-$*$ topology. Note that $C$ is compact and for any $x$ in ${\cl S}$, the mapping $f\mapsto f(x)$ is continuous on $C$. We claim that for any $\vp>0$ and any finite subset $S\subset {\cl S}$ there is a probability $\lambda$ on $C\times C$ (depending on $(\vp,S)$) such that
\begin{equation}\label{eq1.9}
 \forall x,y\in S\qquad\qquad \frac1{K(1+\vp)} \langle x,y\rangle = \intl_{C\times C} f(x) \bar g(y) \ d\lambda (f,g).
\end{equation}
Indeed, since $\|F_S\|_\wedge < K(1+\vp)$ by homogeneity we can rewrite
$ \frac1{K(1+\vp)} F_S$ as a finite sum
\[
 \frac1{K(1+\vp)} F_S = \sum \lambda_m \ f_m\otimes \bar g_m
\]
where $\lambda_m\ge 0$, $\sum\lambda_m =1$, $\|f_m\|_{C(S)} \le 1$, $\|g_m\|_{C(S)}\le 1$. Let $\tilde f_m\in C$, $\tilde g_m\in C$ denote the extensions of $f_m$ and $g_m$ vanishing outside $S$. Setting $\lambda = \sum \lambda_m \delta_{(\tilde f_m,\tilde g_m)}$, we obtain the announced claim. We view $\lambda = \lambda(\vp,S)$ as a net, where we let $\vp\to 0$ and $S\uparrow {\cl S}$. Passing to a subnet, we may assume that $\lambda = \lambda(\vp,S)$ converges weakly to a probability ${\bb P}$ on $C\times C$. Let $\Omega = C\times C$. Passing to the limit in \eqref{eq1.9} we obtain
\begin{equation}
 \frac1K \langle x,y\rangle = \intl_\Omega f(x)\bar g(y)\ d{\bb P}(f,g).\tag*{$\forall x,y\in {\cl S}$}
\end{equation}
Thus if we set $\forall \omega=(f,g)\in \Omega$
\[
 \Phi(x)(\omega) = f(x) \quad \text{and}\quad \Psi(y)(\omega) = g(y)
\]
we obtain finally \eqref{eq1.8}.

To show that the best $K$ in Theorem \ref{thm1.8} is equal to $K_G$, it suffices to show that Theorem \ref{thm1.8} implies Theorem \ref{thm1.2bis} with the same constant $K$. Let $(a_{ij}), (x_i), (y_j)$ be as in Theorem \ref{thm1.2bis}. We have
\[
 \langle x_i,y_j\rangle = K \int \Phi(x_i) \bar\Psi(y_j)\ d{\bb P}
\]
and hence
\begin{align*}
\left|\sum a_{ij}\langle x_i,y_j\rangle\right| &= K\left|\int \left(\sum a_{ij} \Phi(x_i) \bar\Psi(y_j)\right)\ d{\bb P}\right|\\
&\le K \sup\|x_i\| \sup\|y_j\|
\end{align*}
where at the last step we used the assumption on $(a_{ij})$ and $|\Phi(x_i)(\omega)| \le \|x_i\|$, $|\Psi(y_j)(\omega)|\le \|y_j\|$ for almost all $\omega$. Since $K_G$ is the best $K$ appearing in either Theorem \ref{thm1.5} or Theorem \ref{thm1.2bis}, this completes the proof, except for the last assertion, but that follows
from the isometry $L_\infty(\Omega,{\bb P})\simeq C(S)$ for some compact set $S$ (namely the spectrum of $L_\infty(\Omega,{\bb P})$)
and that allows to replace the integral with respect to  ${\bb P}$ by a Radon measure on $S$.
\end{proof}

\begin{rem} 
 The functions $\Phi,\Psi$ appearing above are highly \emph{non-linear}. In sharp contrast, we have
\begin{equation}
 \langle x,y\rangle = \int G(x)\ovl{G(y)}\ d{\bb P}\tag*{$\forall x,y\in\ell_2$}
\end{equation}
 and for any $p<\infty$ (see \eqref{eqgauss})
\[
 \|G(x)\|_p = \|x\|\quad \gamma(p)
\]
where $\gamma(p) = ({\bb E}|g_1|^p)^{1/p}$. Here $x\mapsto G(x)$ is linear but since $\gamma(p)\to \infty$ when $p\to\infty$, this formula does not produce a uniform bound for the norm in $C(S) \widehat\otimes C(S)$ of $F_S(x,y) = \langle x,y\rangle$ with $S\subset {\cl S}$ finite. 
\end{rem}

 \begin{rem}\label{rem-kso} It is natural to wonder whether one can take $\Phi=\Psi$ in Theorem \ref{thm1.8} (possibly with a  larger $K$). Surprisingly the answer is negative, \cite{KS}.
More precisely, Kashin and Szarek estimated the best constant $K(n)$ with the following property:\ for any   $x_1,\ldots, x_n$ in
the unit ball of $H$ there are functions $\Phi_i$ in the unit ball of $ {L_\infty(\Omega,{\bb P})}$ on a probability space $(\Omega,{\bb P})$ (we can easily reduce consideration to the Lebesgue interval) such that
\begin{equation}
 \frac1{K(n)} \langle x_i,x_j\rangle = \int \Phi_i\Phi_j d{\bb P}.\tag*{$\forall i,j=1,\ldots, n$}
\end{equation}
They also consider the best constant $K'(n)$ such that the preceding can be done
but only for distinct pairs $i\not=j$.
They showed that $K'(n)$ grows at least like $(\log n)^{1/2}$, but the exact order of growth 
 $
 K'(n)\approx  {\log } n$ was only obtained in \cite{AMMN}.
The fact that $K'(n)\to \infty$ answered 
a question raised by Megretski (see \cite{Meg}) in connection with 
possible electrical engineering applications.
As observed in \cite{KS1}, the logarithmic growth of $K(n)$
 is much easier :
\begin{lm}[\cite{KS1}]
There are constants $\beta_1,\beta_2>0$ so that $\beta_1\log n\le K(n)\le \beta_2\log n$ for all $n>1$.
\end{lm}

\begin{proof} We restrict ourselves to  the case of real scalars for simplicity.
Using Gaussian random variables it is easy to see that $K(n) \in O(\log n)$. Indeed, consider $x_1,\ldots, x_n$ in the unit ball of $\ell_2$, let $W_n = \sup_{j\le n}|G(x_j)|$ and let $\Phi_j = W^{-1}_nG(x_j)$. We have then $\|\Phi_j \|_{\infty}\le 1$ and for all $i,j=1,\ldots, n$:
\[
 \langle x_i,y_j\rangle  = {\bb E} G(x_i) G(x_j) = {\bb E}(\Phi_i\Phi_jW^2_n)
\]
but, by a well known   elementary estimate,  there is a constant $\beta$ such that ${\bb E}W^2_n \le \beta \log n$ for all $n>1$, so replacing ${\bb P}$ by the probability ${\bb Q }= ({\bb E}W^2_n)^{-1} W^2_n{\bb P}$,
 we obtain $K(n)\le {\bb E}W^2_n\le \beta\log n$.\\
Conversely, let $A$ be a $(1/2)$-net in the unit ball of $\ell^n_2$. We can choose $A$ with card$(A)\le d^n$ where $d>1$ is an integer independent of $n$. We will prove that $K(n+d^n)\ge n/4$. For any $x$ in $\ell^n_2$ we have
\begin{equation}\label{Dix5eq19}
\left(\sum |x_k|^2\right)^{1/2} \le 2 \sup\{|\langle x,\alpha\rangle|\ \big|\ \alpha\in A\}.
\end{equation}
Consider the set $A' = \{e_1,\ldots, e_n\}\cup A$. Let $\{\Phi(x)\mid x\in A'\}$ be in $L_\infty(\Omega,{\bb P})$ with $\sup_{x\in A'}\|\Phi(x)\|_\infty  \le K(n+d^n)^{1/2}$ and such that $\langle x,y\rangle  =\langle\Phi(x), \Phi(y)\rangle$ for any $x,y$ in $A'$. Then obviously $\|\sum \alpha(x)x\|^2 = \|\sum \alpha(x)\Phi(x)\|^2$ for any $\alpha\in {\bb R}^{A'}$. In particular, since $\|\sum \alpha_je_j-\alpha\| = 0$ for any $\alpha$ in $A$ we must have $\sum \alpha_j\Phi(e_j)-\Phi(\alpha) = 0$. Therefore
\begin{equation}
 \left|\sum \alpha_j \Phi(e_j)\right| = |\Phi(\alpha)|\le K(n+d^n)^{1/2}.\tag*{$\forall\alpha\in A$}
\end{equation}
By \eqref{Dix5eq19} this implies $\sum|\Phi(e_j)|^2 \le 4 K(n+d^n)$ and hence after integration we obtain finally
\[
n = \sum\|e_j\|^2 = \sum\|\Phi(e_j)\|^2_2 \le 4 K(n+d^n).\qquad\qed
\]
\renewcommand{\qed}{}\end{proof}
\end{rem}
  
\begin{rem} Similar questions were also considered long before in Menchoff's work on orthogonal series of functions.
 In  that direction, if we assume, in  addition
to $\|x_j\|_H\le 1$, that
\begin{equation}
 \left\|\sum \alpha_jx_j\right\|_H \le \left(\sum |\alpha_j|^2\right)^{1/2} \tag*{$\forall(\alpha_j)\in {\bb R}^n$}
\end{equation}
then it is unknown whether this modified version of $K(n)$ remains bounded when $n\to\infty$. This problem (due to Olevskii, see \cite{Olev}) is a well known open question in the theory of bounded orthogonal systems (in connection with a.s.\ convergence  of the associated series).
\end{rem}
\begin{rem}\label{symbol} 
I strongly suspect that Grothendieck's favorite formulation of GT (among the many in his paper) is this one (see \cite[p. 59]{Gr1}): for any pair $X,Y$ of Banach spaces
\begin{equation}\label{Dix5eq21+}
 \forall T\in X\otimes Y\qquad \quad \|T\|_H \le \|T\|_{/\wedge\backslash}\le K_G\|T\|_H.
\end{equation}
This is essentially the same as \eqref{eq0.3}, but let us explain
   Grothendieck's cryptic notation ${/\alpha\backslash}$. More generally,
for any   norm
$\alpha$,  assumed defined
on $X \otimes Y$ for any pair of Banach spaces $X , Y$, he introduced the   norms
$/\alpha$ and $\alpha\backslash$ as follows.
 Since any Banach space embeds isometrically into   a $C(S)$-space,
 for some compact  set $S$,
    we have isometric embeddings
 $ X\subset X_1$ with $X_1= C(S)$, and  $Y\subset Y_1$ with $Y_1= C(T)$ (for suitable compact sets $S,T$). Consider $t\in X \otimes Y$. Since we may view $t$ as sitting in
 a larger space such e.g. $X_1\otimes Y_1$
we denote 
by $\alpha(t,X_1 \otimes Y_1)$ the resulting norm.
Then, by definition $/\alpha(t,X \otimes Y)=\alpha(t,X_1 \otimes Y)$,
$\alpha\backslash(t,X \otimes Y)=\alpha(t,X \otimes Y_1)$ and combining both
$/\alpha\backslash(t,X \otimes Y)=\alpha(t,X_1 \otimes Y_1)$.
When $\alpha=\|\ \|_\wedge$, this leads to $\|\ \|_{/\wedge\backslash}$. In duality with this procedure, Grothendieck also defined $\backslash\alpha,\alpha/ $
using the fact that any Banach space is a quotient of an $L_1$ (or $\ell_1$) space. He then completed (see \cite[p. 37]{Gr1}) the task of identifying
{\it all} the distinct norms that one obtains by applying duality and  these
basic operations starting from either $\wedge$ or $\vee$,
and he found...14 different norms ! See \cite{DFS} for an exposition of this.
One advantage of the above formulation \eqref{Dix5eq21+}
is that the best constant corresponds to the case when $t$ represents the identity on   Hilbert space ! More precisely, we have
\begin{equation}\label{Dix5eq19w} K_G=\lim_{d\to \infty} \{ \|t_d\|_{/\wedge\backslash} \}
\end{equation}
where $t_d$ is the tensor representing the identity on the $d$-dimensional
(real or complex) Hilbert space $\ell_2^d$, for which obviously $\|t_d\|_H=1$. More precisely, with the notation
in \eqref{eq2+D} below,
 we have $K^{\bb K}_G(d)= \|t_d\|_{/\wedge\backslash}$. 
\end{rem}

\section{The Grothendieck constants}\label{sec1+}

The constant $K_G$ is ``the Grothendieck constant.'' Its exact value is still unknown! \\ Grothendieck \cite{Gr1} proved that $\pi/2 \le K^{\bb R}_G \le {\rm sh}(\pi/2)=2.301..$. Actually, his argument
(see \eqref{eq2.1pre+} below for a proof) shows that $\|g\|^{-2}_1 \le K_G$ where $g$ is a standard Gaussian variable such that ${\bb E}g=0$ and ${\bb E}|g|^2=1$. In the real case $\|g\|_1 = {\bb E}|g| = (2/\pi)^{1/2}$. In the complex case $\|g\|_1 = (\pi/4)^{1/2}$, and hence $K^{\bb C}_G \ge 4/\pi$. It is known (cf.\ \cite{P1}) that $K^{\bb C}_G < K^{\bb R}_G$. After some progress
by Rietz \cite{Rie}, Krivine (\cite{Kr1}) proved that \begin{equation}\label{eq1+}K^{\bb R}_G \le \pi/(2 \text{ Log}(1+\sqrt 2))= 1.782\ldots\end{equation} and conjectured that this is the exact value.
He also proved that $K^{\bb R}_G \ge 1.66$ (unpublished, see also \cite{Ree}). His conjecture remained open    until the very recent paper
  \cite{BMMN} that proved that his bound is not optimal.

 In the complex case the corresponding bounds are due to Haagerup and Davie    \cite{Dav, H4}.
Curiously, there are   difficulties to ``fully" extend Krivine's proof to the complex case. 
   Following this path, Haagerup  \cite{H4}  finds
\begin{equation}\label{eq2+} K_G^{\bb C}\le  \frac 8 {\pi (K_0+1)} =1.4049...
\end{equation}
where $K_0$ is is the unique solution in $(0,1)$ of the equation
$$K\int_0^{\pi/2}  \frac{\cos^{2}t} {(1-K^2 \sin^2t)^{1/2} }  dt   = \frac\pi 8(K+1).$$
Davie (unpublished)  improved the lower bound to $K_G^{\bb C}>1.338$.

Nevertheless, Haagerup  \cite{H4} conjectures that  a ``full"  analogue of Krivine's argument should yield
a slightly better upper bound, he {\it conjectures}  that:
$$ K_G^{\bb C}\le \left(\int_0^{\pi/2} \frac{\cos^{2}t} {(1+ \sin^2t)^{1/2} }  dt   \right)^{-1}=1.4046...!$$
Define $\varphi\colon\ [-1,1]\to [-1,1]$     by
$$\varphi(s)= s \int_0^{\pi/2}  \frac{\cos^{2}t} {(1-s^2 \sin^2t)^{1/2} }  dt.$$
Then (see  \cite{H4}), $\varphi^{-1}\colon\ [-1,1]\to [-1,1]$ exists and admits a convergent power series expansion
$$\varphi^{-1}(s)=\sum_{k=0}^\infty \beta_{2k+1} s^{2k+1} ,$$
with $\beta_1=4/\pi$ and $\beta_{2k+1}\le 0$ for $k\ge 1$.
On one hand $K_0$ is the solution of the equation $2\beta_1 \varphi(s) =1+s$,
or equivalently $\varphi(s)= \frac\pi 8(s+1),$ but on the other hand
 Haagerup's  already mentioned {\it conjecture}   in \cite{H4}  is that, extending  $\varphi$ analytically, we should have
$$K^{\bb C}_G\le |\varphi(i)|= \left(\int_0^{\pi/2} \frac{\cos^{2}t} {(1+ \sin^2t)^{1/2} }  dt   \right)^{-1}.$$

K\"onig \cite{Ko1} pursues this and obtains several bounds for the finite dimensional analogues
of $K^{\bb C}_G$, which we now define.

Let $K^{\bb R}_G(n,d)$ and $K^{\bb C}_G(n,d)$ be the best constant $K$ such that \eqref{eq1.2bis} holds (for all $n\times n$-matrices $[a_{ij}]$ satisfying \eqref{eq1.2prebis}) for all $n$-tuples $(x_1,\ldots, x_n)$, $(y_1,\ldots, y_n)$ in a $d$-dimensional Hilbert space $H$ on ${\bb K} = {\bb R}$ or ${\bb C}$ respectively. Note that $K^{\bb K}_G(n,d) \le K^{\bb K}_G(n,n)$ for any $d$ (since we can replace $(x_j)$, $(y_j)$ by $(Px_j)$ , $(Py_j)$, $P$ being the orthogonal projection onto $\text{span}[x_j]$). Clearly, we have
\begin{equation}\label{eq2+Dpre}
K^{\bb K}_G=\sup\nolimits_{n,d\ge 1} K^{\bb K}_G(n,d) =\sup\nolimits_{n\ge 1} K^{\bb K}_G(n,n).
\end{equation}
We set 
\begin{equation}\label{eq2+D}
 K^{\bb K}_G(d) = \sup\nolimits_{n\ge 1} K^{\bb K}_G(n,d).
\end{equation} Equivalently,  $K^{\bb K}_G(d) $ is the best constant
in \eqref{eq1.2bis} when one restricts the left hand side to
$d$-dimensional unit vectors $(x_i,y_j)$ (and $n$ is arbitrary).

Krivine \cite{Kr2} proves that $K^{\bb R}_G(2) = \sqrt 2$, $K^{\bb R}_G(4)\le \pi/2$ and
obtains numerical upper bounds  for   $ K^{\bb R}_G(n)$  for each $n$. The fact that $K^{\bb R}_G(2) \ge \sqrt 2$ is   obvious since, in the real case, the $2$-dimensional $L_1$- and  $L_\infty$-spaces (namely $\ell_1^2$ and
$\ell_\infty^2$) are isometric, but at Banach-Mazur distance $\sqrt 2$ from $\ell_2^2$.
 The assertion $K^{\bb R}_G(2) \le \sqrt 2$ equivalently means that for any ${\bb R}$-linear $T\colon \ L_\infty(\mu;{\bb R}) \to L_1(\mu';{\bb R})$ the complexification $T^{\bb C}$ satisfies
\[
 \|T^{\bb C}\colon \ L_\infty(\mu;{\bb C})\to L_1(\mu';{\bb C})\| \le \sqrt 2\|T\colon \ L_\infty(\mu; {\bb R})\to L_1(\mu';{\bb R})\|.
\]

In the complex case, refining Davie's lower bound, K\"onig \cite{Ko1} 
obtains two sided bounds (in terms of the function $\varphi$) for $K^{\bb C}_G(d)$, for example he
proves that $1.1526<K^{\bb C}_G(2)<1.2157$.
He also computes the $d$-dimensional version of the preceding Haagerup  conjecture
on $K^{\bb C}_G$.
See   \cite{Ko1,Ko2} for more details, and \cite{FR} for some explicitly computed examples.
See also \cite{Dav} for various remarks on the norm
$\|[a_{ij}]\|_{(p)}$ defined as the supremum of the left hand side of \eqref{eq1.2bis} over all
unit vectors $x_i,y_j$ in an $p$-dimensional complex Hilbert space. In particular,
Davie, Haagerup and Tonge (see \cite{Dav,Ton})  independently proved that 
 \eqref{eq1.2bis}  restricted
to $2 \times 2$ complex matrices holds with $K=1$.

Recently, in \cite{BOV3}, the following variant of 
$K_G^{\bb R}(d)$ was introduced and studied: let $K_G[d,m]$ denote the smallest constant $K$ such that for any real matrix $[a_{ij}]$ of arbitrary size
$$\sup|\sum a_{ij} \langle x_i,y_j\rangle |\le K
 \sup|\sum a_{ij} \langle x'_i,y'_j\rangle |$$
 where the supremum on the left
 (resp. right) runs over all unit vectors $(x_i,y_j)$
 (resp.  $(x'_i,y'_j)$ ) in an $d$-dimensional (resp. $m$-dimensional) Hilbert space. Clearly we have
$K_G^{\bb R}(d)=  K_G[d,1]$.
 
 The best value $\ell_{best}$ of the constant in Corollary \ref{cor1} seems unknown  in both the real and complex case.
Note that, in the real case, we have obviously
$\ell_{best}\ge \sqrt 2$ because, as we just mentioned, the $2$-dimensional $L_1$ and $L_\infty$
are isometric. 

\section{The ``little'' GT}\label{sec2}

The next result, being much easier to prove that GT has been nicknamed the ``little GT.'' 
\begin{thm}\label{thm2.1pre} Let $S,T$ be  compact sets. There is an absolute constant $k$ such that, for any pair of bounded linear operators $u\colon \ C(S)\to H$, $v\colon \ C(T)\to H$ into an arbitrary Hilbert space $H$ and 
for any finite sequence $(x_j,y_j)$ in $C(S)\times C(T)$ we have 
\begin{equation}\label{eq2.1pre}
 \left|\sum \langle u (x_j),v(y_j)\rangle \right| \le k\|u\|\|v\| \left\|\left(\sum |x_j|^2\right)^{1/2}\right\|_\infty \left\|\left( \sum |y_j|^2\right)^{1/2}\right\|_\infty.
\end{equation}
Let $k_G$ denote the best possible such $k$. We have $k_G = \|g\|^{-2}_1$ (recall $g$ denotes a standard $N(0,1)$ Gaussian variable) or more explicitly
\begin{equation}\label{eq2.1pre+}
 k^{\bb R}_G = \pi/2\quad k^{\bb C}_G=4/\pi.
\end{equation}

\end{thm}
Although he used a different formulation and denoted $k^{\bb R}_G$ by $\sigma$,
Grothendieck did prove that $k^{\bb R}_G = \pi/2$ (see \cite[p.51]{Gr1}). His proof extends immediately
to the complex case.  
 
 To see that Theorem \ref{thm2.1pre} is ``weaker'' than GT (in the formulation of Theorem \ref{thm1.2}), let $\varphi(x,y) = \langle u(x), v(y)\rangle$. Then
  $\|\varphi\| \le\|u\|\|v\| $ and hence 
  \eqref{eq1.2}   implies 
 Theorem \ref{thm2.1pre} with
  \begin{equation}\label{kK} k_G\le K_G.\end{equation}

Here is a very useful reformulation of the little GT:

\begin{thm}\label{thm2.1}
 Let $S$ be a compact set. For any bounded linear operator $u\colon \ C(S)\to H$, the following holds:
\begin{itemize}
\item[\rm (i)] There is a probability measure $\lambda$ on $S$ such that
\end{itemize}
\begin{equation}
 \|ux\| \le \sqrt{k_G}\|u\| \left(\int|x|^2 d\lambda\right)^{1/2}.\tag*{$\forall x\in C(S)$}
\end{equation}
\begin{itemize}
 \item[\rm (ii)] For any finite set $(x_1,\ldots, x_n)$ in $C(S)$ we have
\[
\left(\sum\|ux_j\|^2\right)^{1/2} \le \sqrt{k_G}\|u\| \left\|\left(\sum |x_j|^2\right)^{1/2}\right\|_\infty.
\]
\end{itemize}
Moreover $ \sqrt{k_G} $ is the best possible constant in either (i) or (ii).
\end{thm}
The equivalence of (i) and (ii) is a particular case of Proposition \ref{hb3}.
By Cauchy-Schwarz,   Theorem \ref{thm2.1} obviously implies Theorem \ref{thm2.1pre}. Conversely, applying Theorem \ref{thm2.1pre} to the   form $(x,y)\mapsto \langle u(x),u(y)\rangle $ and taking $x_j=y_j$ in \eqref{eq2.1pre} we obtain  Theorem \ref{thm2.1}.

 The next Lemma will serve to minorize $k_G$ (and hence $K_G$).
 \begin{lem}\label{lem2.1}
Let $(\Omega,   \mu)$ be an arbitrary measure space. Consider   $g_j\in L_1(\mu)$
and $x_j\in L_\infty(\mu)$ ($1\le j\le N$) and  positive numbers $a,b$. Assume that
$(g_j)$ and $(x_j)$ are biorthogonal (i.e.
$\langle g_i,x_j \rangle =0$ if $i\not=j$ and $=1$
otherwise) and such that
$$\forall (\alpha_j)\in \KK^N\quad a(\sum |\alpha_j|^2)^{1/2}\ge \|\sum \alpha_j g_j\|_1\quad{\rm and}\quad \left\|\left(\sum |x_j|^2\right)^{1/2}\right\|_\infty\le b\sqrt{N}.$$
Then
$ b^{-2}a^{-2}\le k_G$ (and a fortiori $ b^{-2}a^{-2}\le K_G$).
\end{lem}
\begin{proof} Let $H=\ell_2^N$. Let $u\colon\ L_\infty\to H$ be defined by $u(x)=\sum \langle x, g_i\rangle e_i$. Our assumptions imply   $\sum \|u(x_j)\|^2=N$,
and $\|u\|\le a$. By (ii) in Theorem \ref{thm2.1}
we obtain $b^{-1}a^{-1}\le \sqrt{k_G }$.
\end{proof}
 
It turns out that Theorem \ref{thm2.1} can be proved directly  as a consequence of Khintchine's inequality in $L_1$ or its Gaussian analogue. The Gaussian case yields the best possible $k$.
 
Let $(\vp_j)$ be an i.i.d.\ sequence of $\{\pm1\}$-valued random variables with ${\bb P}(\vp_j = \pm1) = 1/2$. Then for any scalar sequence $(\alpha_j)$   we have
\begin{equation}\label{eq2.1}
 \frac1{\sqrt 2} \left(\sum |\alpha_j|^2\right)^{1/2} \le \left\|\sum \alpha_j\vp_j\right\|_1 \le \left(\sum |\alpha_j|^2\right)^{1/2}.
\end{equation}
Of course the second inequality is trivial since $\|\sum \vp_j\alpha_j\|_2 = (\sum|\alpha_j|^2)^{1/2}$.

The Gaussian analogue of this equivalence is the following equality:
\begin{equation}\label{eq2.2}
 \left\|\sum \alpha_jg_j\right\|_1 = \|g\|_1 \left(\sum |\alpha_j|^2\right)^{1/2}.
\end{equation}
(Recall that $(g_j)$ is an i.i.d.\ sequence of copies of $g$.) Note that each of these assertions corresponds to an embedding of $\ell_2$ into $L_1(\Omega,{\bb P})$, an isomorphic one in case of \eqref{eq2.1}, isometric in case of \eqref{eq2.2}.

A typical application of \eqref{eq2.2}  is as follows:\ let $v\colon\ H\to L_1(\Omega',\mu')$ be a bounded operator. Then for any finite sequence $(y_j)$ in $H$
\begin{equation}\label{eq2.3}
 \|g\|_1 \left\|\left(\sum |v(y_j)|^2\right)^{1/2}\right\|_1 \le \|v\| \left(\sum \|y_j\|^2\right)^{1/2}.
\end{equation}
Indeed, we have by \eqref{eq2.2}
\begin{align*}
\|g\|_1 \left\|\left(\sum |v(y_j)|^2\right)^{1/2}\right\|_1 &= \left\|\sum g_jv(y_j)\right\|_{L_1({\bb P}\times\mu')} = \int \left\|v\left(\sum g_j(\omega)y_j\right)\right\|_1 d{\bb P}(\omega)\\
&\le \|v\|\left(\int\left\|\sum g_j(\omega)y_j\right\|^2 d{\bb P}(\omega)\right) = \|v\| \left(\sum \|y_j\|^2\right)^{1/2}.
\end{align*}
Consider now the adjoint $v^*\colon \ L_\infty(\mu') \to H^*$. Let $u=v^*$. Note $\|u\|=\|v\|$. By duality, it is easy to deduce from \eqref{eq2.3} that for any finite subset $(x_1,\ldots, x_n)$ in $L_\infty(\mu')$ we have
\begin{equation}\label{eq2.4}
 \left(\sum\|ux_j\|^2\right)^{1/2} \le \|u\| \|g\|^{-1}_1 \left\|\left(\sum |x_j|^2\right)^{1/2}\right\|_\infty.
\end{equation}
This leads us to:
\begin{proof}[Proof of Theorem \ref{thm2.1} (and Theorem \ref{thm2.1pre})]
By the preceding observations,
it remains to prove (ii) and justify the value of $k_G$. Here again by suitable approximation (${\cl L}_\infty$-spaces) we may reduce to the case when $S$ is a finite set. Then $C(S)=\ell_\infty(S)$ and so that if we apply \eqref{eq2.4} to $u$ we obtain (ii) with $\sqrt{k}\le \|g\|^{-1}_1$. This shows that $k_G\le \|g\|^{-2}_1$. To show that $k_G\ge \|g\|^{-2}_1$, 
we will use Lemma \ref{lem2.1}. 
Let $x_j=c_N N^{1/2} g_j (\sum |g_j|^2)^{-1/2}$ with $c_N$ adjusted 
so that $(x_j)$ is biorthogonal to $(g_j)$, i.e.
we set $N^{-1/2}  c_N^{-1}=\int |g_1|^2(\sum |g_j|^2)^{-1/2}$. Note $\int |g_j|^2(\sum |g_j|^2)^{-1/2}=\int |g_1|^2(\sum |g_j|^2)^{-1/2}$ for any $j$ and hence 
$  c_N^{-1}= N^{-1/2} \sum_j \int |g_j|^2(\sum |g_j|^2)^{-1/2}= N^{-1/2} \int (\sum |g_j|^2)^{1/2}$. 
Thus by the strong law  of large numbers (actually here the weak law suffices), since ${\bb E}|g_j|^2=1$,
$  c_N^{-1}\to 1$ when $N\to \infty$.
Then by Lemma \ref{lem2.1} (recall \eqref{eq2.2}) we find
$  c_N^{-1} \|g\|^{-1}_1\le \sqrt{k_G}$. Thus
letting $N\to \infty$ we
obtain $\|g\|^{-1}_1\le \sqrt{k_G}$.
\end{proof}

\section{Banach spaces satisfying GT}\label{sec30}

It is natural to try to extend GT to Banach spaces other than $C(S)$ or $L_\infty$ (or their duals). Since any Banach space is isometric to a subspace of $C(S)$ for some compact $S$, one can phrase the question like this:\ What are the pairs of subspaces $X\subset C(S)$, $Y\subset C(T)$ such that any bounded bilinear form $\varphi\colon \ X\times Y\to {\bb K}$ satisfies the conclusion of Theorem \ref{thm1.1}?
Equivalently, this property means that $\|\cdot\|_H$ and $\|~~\|_\wedge$ are equivalent on $X\otimes Y$. This reformulation shows that the property does not depend on the choice of the embeddings $X\subset C(S)$, $Y\subset C(T)$.
Since the reference \cite{P2} contains a lot of information on this question, we will merely briefly outline what is known and point to a few more recent sources.

\begin{dfn}\label{dfn30.1}
(i) A pair of (real or complex) Banach spaces $(X,Y)$ will be called a GT-pair if any bounded bilinear form $\varphi\colon \ X\times Y\to {\bb K}$ satisfies the conclusion of Theorem \ref{thm1.1}, with say $S,T$ equal respectively to the weak$^*$ unit balls of $X^*,Y^*$. Equivalently, this means that there is a constant $C$ such that
$$\forall t\in X\otimes Y\quad     \|t\|_{\wedge}\le C \|t\|_{H}.$$\\
(ii) A Banach space $X$ is called a GT-space (in the sense of \cite{P2}) if $(X^*,C(T))$ is a GT-pair for any compact set $T$.
\end{dfn}

GT tells us that the pairs $(C(S),C(T))$ or
$(L_\infty(\mu),L_\infty(\mu'))$ are GT-pairs, and that
all abstract $L_1$-spaces (this includes $ C(S)^*$)
are GT-spaces.

\begin{rem}\label{gtext} 
If $(X,Y)$ is a GT-pair and we have isomorphic embeddings
$X\subset X_1$ and $Y\subset Y_1$, with arbitrary $X_1,Y_1$, then any bounded bilinear form on $ X\times Y$
extends to one on $ X_1\times Y_1$.
\end{rem}

Let us say that a pair $(X,Y)$ is a ``Hilbertian pair'' if every bounded linear operator from $X$ to $Y^*$ factors through a Hilbert space. Equivalently this means that $X\widehat\otimes Y$ and $X\widehat\otimes_{H'}Y$ coincide with equivalent norms. Assuming $X,Y$ infinite dimensional,
 it is easy to see (using Dvoretzky's theorem for the only if part) that $(X,Y)$ is a GT-pair iff it is a Hilbertian pair and moreover each operator from $X$ to $\ell_2$ and from $Y$ to $\ell_2$ is 2-summing (see \S \ref{sechb} for $p$-summing operators).

It is known (see \cite{P2}) that $(X,C(T))$ is a GT-pair for all $T$ iff $X^*$ is a GT-space. See \cite{Kis3} for a discussion of GT-pairs. The main examples are pairs $(X,Y)$ such that $X^\bot \subset C(S)^*$ and $Y^\bot\subset C(T)^*$ are both reflexive (proved independently by Kisliakov and the author in 1976), also the pairs $(X,Y)$ with $X=Y=A(D)$ (disc algebra) and $X=Y=H^\infty$ (proved by Bourgain in 1981 and 1984). 
In particular, Bourgain's result shows that if (using boundary values) we view  $H^\infty$ over the disc as
isometrically embedded in the space $L_\infty$ over the unit circle, then
any bounded bilinear form on $H^\infty\times H^\infty $ extends to a bounded bilinear one
on  $L^\infty\times L^\infty $. See \cite{P22} for a proof that the pair $(J,J)$ is a Hilbertian pair, when $J$ is the classical James space of codimension 1 in its bidual.

On a more ``abstract'' level, if $X^*$ and $Y^*$ are both GT-spaces of cotype 2 (resp. both of cotype 2) and if one of them has the approximation property then $(X,Y)$ is a GT pair (resp. a Hilbertian pair). See  Definition \ref{cotype} below for the precise definition of ''cotype 2". This 
``abstract'' form of GT was crucially used in the author's construction
of {\it infinite dimensional} Banach spaces $X$ such that    $X\stackrel{\wedge}{\otimes} X=X\stackrel{\vee}{\otimes} X$,
  i.e. $\|\ \|_\wedge$ and $\|\ \|_\vee$ are equivalent 
  norms on $X \otimes X$.
See \cite{P2} for more precise references on all this. 

As for more recent results
on the same direction  see \cite{Jo} for examples of pairs of infinite dimensional Banach spaces $X,Y$ such that
any {\it compact} operator from $X$ to $Y$  is nuclear.
 Note that there is still no nice characterization of Hilbertian pairs. See \cite{Leung} for a counterexample to a conjecture in \cite{P2}  on this. 

We refer the reader to  \cite{Kis2} and \cite{GKis} for more recent updates on GT in connection with Banach spaces of analytic functions and uniform algebras. 

Grothendieck's work was a major source of inspiration
in    the development of Banach space Geometry   in the last 4 decades. We refer the reader to \cite{Kw1,Pie,TJ2,DeF, DFS} for more on this development.

\section{Non-commutative GT}\label{sec3}

Grothendieck himself conjectured (\cite[p. 73]{Gr1}) a non-commutative version of Theorem \ref{thm1.1}. This was proved in \cite{P1} with an additional approximation assumption,
and in  \cite{H3} in general. Actually, a posteriori, by    \cite{H3}, the assumption 
needed in   \cite{P1}   always holds. Incidentally it is amusing to note that Grothendieck overlooked the difficulty related to the approximation property, since he asserts without proof that the problem can be reduced to finite dimensional $C^*$-algebras.

In the optimal form proved in \cite{H3}, the result is as follows.

\begin{thm}\label{thm3.1}
Let $A,B$ be $C^*$-algebras. Then for any bounded bilinear form $\varphi\colon \ A\times B\to {\bb C}$ there are states $f_1,f_2$ on $A$, $g_1,g_2$ on $B$ such that
\begin{equation}\label{eq3.1}
\forall(x,y)\in A\times B\qquad \quad |\varphi(x,y)| \le \|\varphi\| (f_1(x^*x) + f_2(xx^*))^{1/2} (g_1(yy^*) + g_2(y^*y))^{1/2}.
\end{equation}
Actually, just as    \eqref{eq1.1} and \eqref{eq1.2}, \eqref{eq3.1} is equivalent to the following  inequality valid for all finite sets $(x_j,y_j)$ $1\le j\le n$ in $A\times B$
\begin{equation}\label{eq3.2}
 \left|\sum \varphi(x_j,y_j)\right| \le \|\varphi\| \left(\left\|\sum x^*_jx_j\right\| + \left\|\sum x_jx^*_j\right\|\right)^{1/2} \left(\left\|\sum y^*_jy_j\right\| + \left\|\sum y_jy^*_j\right\|\right)^{1/2}.
\end{equation}
A fortiori we have
\begin{equation}\label{eq3.2+}
 \left|\sum \varphi(x_j,y_j)\right| \le 2
 \|\varphi\| \max\left\{ \left\|\sum x^*_jx_j\right\|^{1/2}, \left\|\sum x_jx^*_j\right\|^{1/2}\right\} \max\left\{ \left\|\sum y^*_jy_j\right\|^{1/2}, \left\|\sum y_jy^*_j\right\|^{1/2}\right\}.
\end{equation}
\end{thm}
For further reference, we state the following obvious consequence:
\begin{cor}\label{cor3.1} Let $A,B$ be $C^*$-algebras and let $u\colon \ A \to H$ and $v\colon \ B \to H$ be bounded linear operators into a Hilbert space $H$. Then for all finite sets $(x_j,y_j)$ $1\le j\le n$ in $A\times B$ we have
\begin{equation}\label{eq3.2++}
  |\sum \langle u(x_j),v(y_j)\rangle | \le 2
 \|u\| \|v\|\max \{  \|\sum x^*_jx_j \|^{1/2},  \|\sum x_jx^*_j \|^{1/2} \} \max \{  \|\sum y^*_jy_j \|^{1/2},  \|\sum y_jy^*_j \|^{1/2} \}.
\end{equation}
\end{cor}

It was proved in \cite{HI} that the constant 1 is best possible in either \eqref{eq3.1} or \eqref{eq3.2}. Note however that ``constant $=1$'' is a bit misleading since if $A,B$ are both commutative, \eqref{eq3.1} or \eqref{eq3.2} yield \eqref{eq1.1} or \eqref{eq1.2} with $K=2$.

\begin{lem}\label{lem3.3}
Consider $C^*$-algebras $A,B$ and subspaces $E\subset A$ and $F\subset B$. Let $u\colon \ E\to F^*$ a linear map with associated bilinear form $\varphi(x,y) = \langle ux,y\rangle$ on $E\times F$. Assume that there are states $f_1,f_2$ on $A,g_1,g_2$ on $B$ such that
\begin{equation}
|\langle ux,y\rangle| \le (f_1(xx^*) + f_2(x^*x))^{1/2} (g_1(y^*y) + g_2(yy^*))^{1/2}.\tag*{$\forall (x,y)\in E\times F$}
\end{equation}
Then $\varphi\colon \ E\times F\to {\bb C}$ admits a bounded bilinear extension $\widetilde\varphi\colon \ A\times B\to {\bb C}$ that can be decomposed as a sum
\[
 \widetilde\varphi = \varphi_1 + \varphi_2 + \varphi_3 + \varphi_4
\]
where $\varphi_j\colon \ A\times B\to {\bb C}$ are bounded bilinear forms satisfying the following: For all $(x,y)$ in $A\times B$
\begin{align}\label{3.3_1}
|\varphi_1(x,y)| &\le (f_1(xx^*) g_1(y^*y))^{1/2}\tag*{$(6.3)_1$}\\
\label{3.3_2}
|\varphi_2(x,y)| &\le (f_2(x^*x) g_2(yy^*))^{1/2}\tag*{$(6.3)_2$}\\
|\varphi_3(x,y)| &\le (f_1(xx^*) g_2(yy^*))^{1/2}\tag*{$(6.3)_3$}\\
|\varphi_4(x,y)| &\le (f_2(x^*x) g_1(y^*y))^{1/2}.\tag*{$(6.3)_4$}
\end{align}
A fortiori, we can write $u = u_1 + u_2 + u_3 + u_4$ where $u_j\colon \ E\to F^*$ satisfies the same bound as $\varphi_j$ for  $(x,y)$ in $E\times F$.
\end{lem}

\addtocounter{equation}{1}

\begin{proof}
Let $H_1$ and $H_2$ (resp.\ $K_1$ and $K_2$) be the Hilbert spaces obtained from $A$ (resp.\ $B$) with respect to the (non Hausdorff) inner products $\langle a,b\rangle_{H_1} = f_1(ab^*)$ and $\langle a,b\rangle_{H_2} = f_2(b^*a)$ (resp.\ $\langle a,b\rangle_{K_1} = g_1(b^*a)$ and $\langle a,b\rangle_{K_2} = g_2(ab^*))$. Then our assumption can be rewritten as $|\langle ux,y\rangle| \le  \|(x,x)\|_{H_1\oplus H_2} \|(y,y)\|_{K_1\oplus K_2}$. Therefore using the orthogonal projection from $H_1\oplus H_2$ onto $\overline{\rm span}[(x\oplus x)\mid x\in E]$ and similarly for $K_1\oplus K_2$, we find an operator $U\colon \ H_1\oplus H_2 \to (K_1\oplus K_2)^*$ with $\|U\|\le 1$ such that
\begin{equation}\label{eq3.4}
 \forall(x,y)\in E\times F\qquad\quad \langle ux,y\rangle = \langle U(x\oplus x), (y\oplus y)\rangle.
\end{equation}
Clearly we have contractive linear ``inclusions'' 
\[
 A\subset H_1, \quad A\subset H_2,\quad B\subset K_1,\quad B\subset K_2
\]
so that the bilinear forms $\varphi_1,\varphi_2,\varphi_3,\varphi_4$ defined on $A\times B$ by the identity
\[
 \langle  U(x_1\oplus x_2), (y_1\oplus y_2)\rangle  = \varphi_1(x_1,y_1) + \varphi_2(x_2,y_2) + \varphi_3(x_1,y_2) + \varphi_4(x_2,y_1)
\]
must satisfy the inequalities \ref{3.3_1} and after. By \eqref{eq3.4}, we have $(\varphi_1 + \varphi_2 + \varphi_3 +\varphi_4)_{|E\times F} = \varphi$. Equivalently,  if $u_j\colon \ E\to F^*$ are the linear maps associated to $\varphi_j$, we have $u = u_1 + u_2 + u_3 + u_4$.
\end{proof}

\section{Non-commutative ``little GT''}\label{sec4}

The next result was first proved in \cite{P1} with a larger constant. Haagerup \cite{H3} obtained the constant 1, that was shown to be optimal in \cite{HI} (see \S \ref{sec9} for details).

\begin{thm}\label{thm4.1}
 Let $A$ be a $C^*$-algebra, $H$ a Hilbert space. Then for any bounded linear map $u\colon \ A\to H$ there are states $f_1,f_2$ on $A$ such that
\begin{equation}
 \|ux\|\le \|u\|(f_1(x^*x) + f_2(xx^*))^{1/2}.\tag*{$\forall x\in A$}
\end{equation}
Equivalently (see Proposition \ref{hb3}), for any finite set $(x_1,\ldots, x_n)$ in $A$ we have
\begin{equation}\label{eq4.1}
 \left(\sum\|ux_j\|^2\right)^{1/2} \le \|u\|\left(\left\|\sum x^*_jx_j\right\| + \left\|\sum x_jx^*_j\right\|\right)^{1/2},\end{equation}   and a fortiori  
\begin{equation}\label{eq4.1bis}
 \left(\sum\|ux_j\|^2\right)^{1/2} \le \sqrt{2}\|u\|\max\left\{\left\|\sum x^*_jx_j\right\|^{1/2} , \left\|\sum x_jx^*_j\right\|^{1/2}\right\}.
\end{equation}
\end{thm}

\begin{proof}
 This can be deduced from Theorem \ref{thm3.1} (or Corollary \ref{cor3.1} exactly as in \S\ref{sec2}, by considering the bounded bilinear (actually sesquilinear) form $\varphi(x,y) = \langle ux,uy\rangle$.
\end{proof}

Note that $\|(\sum x^*_jx_j)^{1/2}\| = \|\sum x^*_jx_j\|^{1/2}$. At first sight, the reader may view $\|(\sum x^*_jx_j)^{1/2}\|_A$ as the natural generalization of the norm $\|(\sum |x_j|^2)^{1/2}\|_\infty$ appearing in Theorem \ref{thm1.1} in case $A$ is commutative. There is however a major difference:\ if $A_1,A_2$ are commutative $C^*$-algebras, then for any bounded $u\colon \ A_1\to A_2$ we have for any $x_1,\ldots, x_n$ in $A_1$
\begin{equation}\label{eq4.2}
 \left\|\left(\sum u(x_j)^*u(x_j)\right)^{1/2}\right\| \le \|u\| \left\|\left( \sum x^*_jx_j\right)^{1/2}\right\|.
\end{equation}
The simplest way to check this is to observe that
\begin{equation}\label{eq4.3}
\left\|\left(\sum |x_j|^2\right)^{1/2}\right\|_\infty = \sup\left\{\left\|\sum \alpha_jx_j\right\|_\infty \ \Big| \  \alpha_j\in {\bb K} \sum |\alpha_j|^2\le 1\right\}.
\end{equation}
Indeed, \eqref{eq4.3} shows that the seemingly non-linear expression $\|(\sum|x_j|^2)^{1/2}\|_\infty$ can be suitably ``linearized'' so that \eqref{eq4.2} holds. But this is no longer true when $A_1$ is not-commutative. In fact, let $A_1 = A_2 = M_n$ and $u\colon \ M_n\to M_n$ be the transposition of $n\times n$ matrices, then if we set $x_j = e_{1j}$, we have $ux_j = e_{j1}$ and we find $\|(\sum x^*_jx_j)^{1/2}\| = 1$ but $\|(\sum (ux_j)^*(ux_j))^{1/2}\| = \sqrt n$.
This shows that there is no non-commutative analogue of the ``linearization'' \eqref{eq4.3} (even as a two-sided equivalence). The ``right'' substitute seems to be the following corollary.\ms 

\n {\bf Notation.} Let $x_1,\ldots, x_n$ be a finite set in  a $C^*$-algebra $A$. We denote
\begin{equation}\label{eq4.4}
 \|(x_j)\|_C  = \left\|\left(\sum x^*_jx_j\right)^{1/2}\right\|,\quad \|(x_j)\|_R = \left\|\left( \sum x_jx^*\right)^{1/2}\right\|, \  {\rm and:} \end{equation}
\begin{equation}
\label{eq4.5}
\|(x_j)\|_{RC}  = \max\{\|(x_j)\|_R, \|(x_j)\|_C\}.
\end{equation}
It is easy to check that $\| . \|_C, \|.\|_R$ and hence also $\|.\|_{RC}$
are norms on the space of finitely supported sequences   in $A$.

\begin{cor}\label{cor4.2}
Let $A_1,A_2$ be $C^*$-algebras. Let $u\colon \ A_1\to A_2$ be a bounded linear map. Then for any finite set $x_1,\ldots, x_n$ in $A_1$ we have
\begin{equation}\label{eq4.6}
\|(ux_j)\|_{RC} \le \sqrt 2\|u\| \|(x_j)\|_{RC}.
\end{equation}
\end{cor}

\begin{proof}
Let $\xi$ be any state on $A_2$. Let $L_2(\xi)$ be the Hilbert space obtained from $A_2$ equipped with the inner product $\langle a,b\rangle = \xi(b^*a)$. (This is the so-called ``GNS-construction''). We then have a canonical inclusion $j_\xi\colon \ A_2\to L_2(\xi)$. Then \eqref{eq4.1} applied to the composition $j_\xi u$ yields
\[
 \left(\sum \xi((ux_j)^*(ux_j))\right)^{1/2} \le \|u\|\sqrt 2\|(x_j)\|_{RC}.
\]
Taking the supremum over all $\xi$ we find
\[
 \|(ux_j)\|_C \le \|u\|\sqrt 2\|(x_j)\|_{RC}.
\]
Similarly (taking $\langle a,b\rangle = \xi(ab^*)$ instead) we find
\[
\|(ux_j)\|_R \le \|u\|\sqrt 2\|(x_j)\|_{RC}.\qquad \qed
\]
\renewcommand{\qed}{}\end{proof}

\begin{rem} The problem of extending the non-commutative GT from
$C^*$-algebras to $JB^*$-triples was considered notably by Barton and Friedman 
around 1987, but seems to be still incomplete, see 
\cite{Pera1,Pera2} for a discussion
and more precise   references.
\end{rem}

\section{Non-commutative Khintchine inequality}\label{sec5}

In analogy with \eqref{eq4.3}, it is natural to expect that there is a ``linearization'' of $[(x_j)]_{RC}$ that is behind \eqref{eq4.6}. This is one of the applications of the Khintchine inequality in non-commutative $L_1$, i.e.\ the non-commutative version of \eqref{eq2.1} and \eqref{eq2.2}.

Recall first that by ``non-commutative $L_1$-space,'' one usually means a Banach space $X$ such that its dual $X^*$ is a von Neumann algebra. (We could equivalently say ``such that $X^*$ is (isometric to) a $C^*$-algebra'' because a $C^*$-algebra that is also a dual must be a von Neumann algebra isometrically.) Since the commutative case is almost always the basic example for the theory it seems silly to exclude it, so we will say instead that $X$ is a generalized (possibly non-commutative) $L_1$-space. When $M$ is a commutative von Neumann algebra, we have $M\simeq L_\infty(\Omega,\mu)$ isometrically for some abstract measure space $(\Omega,\mu)$ and hence if $M=X^*$, $X\simeq L_1(\Omega,\mu)$.

When $M$ is a non-commutative von Neumann algebra the measure $\mu$ is replaced by a trace, i.e.\ an additive, positively homogeneous functional $\tau\colon \ M_+\to [0,\infty]$, such that $\tau(xy) = \tau(yx)$ for all $x,y$ in $M_+$. The trace $\tau$ is usually assumed ``normal'' (this is equivalent to $\sigma$-additivity, i.e.\ $\sum\tau(P_i) = \tau(\sum P_i)$ for any family $(P_i)$ of mutually orthogonal self-adjoint projections $P_i$ in $M$) and ``semi-finite'' (i.e.\ the set of $x\in M_+$ such that $\tau(x)<\infty$ generates $M$ as a von Neumann algebra). One also assumes $\tau$ ``faithful'' (i.e.\ $\tau(x)=0\Rightarrow x=0$). One can then mimic the construction of $L_1(\Omega,\mu)$ and construct the space $X = L_1(M,\tau)$ in such a way that $L_1(M,\tau)^*=M$ isometrically. This analogy explains why one often sets $L_\infty(M,\tau) = M$ and one denotes by $\|~~\|_\infty$ and $\|~~\|_1$ respectively the norms in $M$ and $L_1(M,\tau)$. 

For example if $M=B(H)$ then the usual trace $x\mapsto \text{tr}(x)$ is semi-finite and $L_1(M,\text{tr})$ can be identified with the trace class (usually denoted by $S_1(H)$ or simply $S_1$) that is formed of all the compact operators $x\colon \ H\to H$ such that $\text{tr}(|x|)<\infty$ with $\|x\|_{S_1}  = \text{tr}(|x|)$. (Here of course by convention $|x| = (x^*x)^{1/2}$.) It is classical that $S_1(H)^* \simeq B(H)$ isometrically.

Consider now $(M,\tau)$ as above  with $\tau$ normal, faithful and  semi-finite. Let $y_1,\ldots, y_n\in L_1(M,\tau)$. The following equalities are  easy to check:
\begin{equation}\label{rcdual}
\tau\left(\left(\sum y^*_jy_j\right)^{1/2}\right)  = \sup\left\{\left|\sum \tau(x_jy_j)\right|\ \Big| \ \|(x_j)\|_R\le 1\right\}\\
\end{equation}
\begin{equation}\label{crdual}
\tau\left(\left(\sum y_jy^*_j\right)^{1/2}\right)  = \sup\left\{\left|\sum \tau(x_jy_j)\right|\ \Big| \ \|(x_j)\|_C \le 1\right\}.
\end{equation}
In other words the norm $(x_j)\mapsto \|(\sum x^*_jx_j)^{1/2}\|_\infty = \|(x_j)\|_C$ is naturally in duality with the norm $(y_j)\mapsto \|(\sum y_jy^*_j)^{1/2}\|_1$, and similarly for $(x_j)\mapsto \|(x_j)\|_R$. Incidentally the last two equalities show that 
$(y_j)\mapsto \|(\sum y_jy^*_j)^{1/2}\|_1$ and
$(y_j)\mapsto \|(\sum y_j^*y_j)^{1/2}\|_1$  
satisfy the triangle inequality and hence are indeed norms.

The von~Neumann algebras $M$ that admit a trace $\tau$ as above are called ``semi-finite''
( ``finite" if  $\tau(1)<\infty$), but 
although the latter case is fundamental, 
as we will see in \S \ref{sec13bis} and \ref{sec14}, there are many important algebras that are not semi-finite. To cover that case too in the sequel we make the following convention. If $X$ is any generalized (possibly non-commutative) $L_1$-space, with $M=X^*$ possibly non-semifinite, then for any $(y_1,\ldots, y_n)$ in $X$ we set by definition
\begin{align*}
\|(y_j)\|_{1,R} &= \sup\left\{\left|\sum \langle x_j,y_j\rangle\right| \ \Big| \ x_j \in M,\ \|(x_j)\|_C \le 1\right\}\\
\|(y_j)\|_{1,C} &= \sup\left\{\left|\sum \langle x_j,y_j\rangle\right|\ \Big| \ x_j\in M,\ \|(x_j)\|_R \le 1\right\}. 
\end{align*}
Here of course $\langle \cdot,\cdot\rangle$ denotes the duality between $X$ and $M = X^*$. Since $M$ admits a unique predual (up to isometry) it is customary to set $M_*=X$.\ms 

\n {\bf Notation.} For $(y_1,\ldots, y_n)$ is $M_*$ we set
\[
|||(y_j)|||_1 = \inf\{\|(y'_j)\|_{1,R} + \|(y''_j)\|_{1,C}\}
\]
where the infimum runs over all possible decompositions of the form $y_j = y'_j + y''_j$, $j=1,\ldots, n$. By an elementary duality argument, one deduces from \eqref{rcdual} and \eqref{crdual}
that for all $(y_1,\ldots, y_n)$ is $M_*$ 
\begin{equation}\label{c+rdual}
|||(y_j)|||_1 = \sup\left\{\left|\sum \langle x_j,y_j\rangle\right| \ \Big| \ x_j \in M,\ \max\{\|(x_j)\|_C ,\|(x_j)\|_R\}\le 1\right\}.
\end{equation}
We will denote by $(g^{\bb R}_j)$ (resp.\ $(g^{\bb C}_j)$) an independent  sequence of real (resp.\ complex) valued Gaussian random variables with mean zero and $L_2$-norm 1. 
We also denote by $(s_j)$ an independent   sequence of  complex valued variables,
each one uniformly distributed over the unit circle $\bb T$. This is the complex analogue
(``Steinhaus variables") of
the sequence $(\vp_j)$.
We can now state the Khintchine inequality for (possibly) non-commutative $L_1$-spaces, and its Gaussian counterpart:

\begin{thm}\label{thm5.1}
There are constants $c_1,c^{\bb C}_1, {\tilde c}_1$ and ${\tilde c}^{\bb C}_1$ such that for any $M$ and any finite set $(y_1,\ldots, y_n)$ in $M_*$ we have (recall $\|\cdot\|_1 = \|~~\|_{M_*}$)
\begin{align}\label{eq5.1}
 \frac1{c_1} |||(y_j)|||_1 &\le \int\left\|\sum \vp_j(\omega)y_j\right\|_1 d{\bb P}(\omega) \le |||(y_j)|||_1\\
 \label{eq5.1+}
 \frac1{c^{\bb C}_1} |||(y_j)|||_1 &\le \int\left\|\sum s_j(\omega)y_j\right\|_1 d{\bb P}(\omega) \le |||(y_j)|||_1\\
\label{eq5.2}
\frac1{{\tilde c}_1}|||(y_j)|||_1 &\le \int\left\|\sum g^{\bb R}_j(\omega)y_j\right\|_1 d{\bb P}(\omega) \le |||(y_j)|||\\
\label{eq5.3}
\frac1{{\tilde c}^{\bb C}_1}|||(y_j)|||_1 &\le \int \left\|\sum g^{\bb C}_j(\omega)y_j\right\|_1 d{\bb P}(\omega) \le |||(y_j)|||_1.
\end{align}
\end{thm}
This result was first proved in \cite{LPP}. Actually \cite{LPP} contains two proofs of it, one that derives it from the ``non-commutative little GT'' and the so-called cotype 2 property of $M_*$, another one based on the factorization of functions in the Hardy space of $M_*$-valued functions $H_1(M_*)$. With a little polish   (see \cite[{p. 347 }]{P3}), the second proof yields \eqref{eq5.1+} with $c^{\bb C}_1=2$, and hence  ${\tilde c}^{\bb C}_1=2$ by an easy central limit argument. More recently, Haagerup and Musat (\cite{HM1}) found a proof of \eqref{eq5.2}  and \eqref{eq5.3} with ${ c}^{\bb C}_1={\tilde c}^{\bb C}_1=\sqrt2$, and by \cite{HI} these are the best constants here (see   Theorem \ref{thm9.2} below). They also proved that $c_1\le \sqrt 3$ (and hence ${\tilde c}_1\le \sqrt 3$ by the central limit theorem) but the best values of $c_1$ and ${\tilde c}_1$ remain apparently unknown.

To prove \eqref{eq5.1}, we will use the following.

\begin{lem}[\cite{P1, H3}]\label{lem5.2}
Let $M$ be a $C^*$-algebra. Consider $x_1,\ldots, x_n\in M$. Let $S = \sum \vp_k x_k$ and let $\tilde S = \sum   s_k x_k$. Assuming $x^*_k=x_k$, we have 
\begin{equation}\label{eq5.4}
 \left\|\left(\int S^4\ d{\bb P}\right)^{1/4}\right\| \le 3^{1/4} \left\|\left(\sum x^2_k\right)^{1/2}\right\|.
\end{equation}
No longer assuming $(x_k)$ self-adjoint, we set $T= \left(\begin{matrix}{0}\ \ { \tilde S}\\
{{\tilde S}^* }\ \ {0}\  
\end{matrix}\right)$ so that $T=T^*$. We have then
\begin{equation}\label{eq5.4+}
 \left\|\left(\int T^4\ d{\bb P}\right)^{1/4}\right\| \le 2^{1/4} \|(x_k)\|_{RC}.
\end{equation}
\end{lem}

\begin{rmk}
By the central limit theorem, $3^{1/4}$ and $2^{1/4}$ are clearly  the best constants here, because  ${\bb E}|g^{\bb R}|^4=3$ and ${\bb E}|g^{\bb C}|^4=2$.
\end{rmk}

\begin{proof}[Proof of Theorem \ref{thm5.1}]
By duality \eqref{eq5.1} is clearly equivalent to:\ $\forall n \forall x_1,\ldots, x_n\in M$
\begin{equation}\label{eq5.1*}
 [(x_k)] \le c_1\|(x_k)\|_{RC}
\end{equation}
where
\[
 [(x_k)] \overset{\sst \text{def}}{=} \inf\left\{\|\Phi\|_{L_\infty({\bb P};M)}\mid \int \vp_k\Phi\ d{\bb P} = x_k, k=1,\ldots, n\right\}.
\]
The property $\int \vp_k \Phi\ d{\bb P} = x_k$  $(k=1,\ldots, n)$ can be reformulated as $\Phi = \sum \vp_kx_k+\Phi'$ with $\Phi'\in [\vp_k]^\bot \otimes M$. Note for further use that we have
\begin{equation}\label{eq5.5}
\left\|\left(\sum x^*_kx_k\right)^{1/2}\right\| \le \left\|\left(\int \Phi^*\Phi\ d{\bb P}\right)^{1/2}\right\|_M.
\end{equation}
Indeed \eqref{eq5.5} is immediate by orthogonality because
\[
\sum x^*_kx_k \le \sum x^*_kx_k + \int \Phi^{\prime *}\Phi' \ d{\bb P} = \int \Phi^*\Phi\ d{\bb P}.
\]
By an obvious iteration, it suffices to show the following.

\n {\bf Claim:}\ If $\|(x_k)\|_{RC}<1$, then there is $\Phi$ in $L_\infty({\bb P};M)$ with $\|\Phi\|_{L_\infty({\bb P};M)}\le c_1/2$ such that if $\hat x_k \overset{\sst\text{def}}{=} \int \vp_k \Phi\ d{\bb P}$ we have $\|(x_k-\hat x_k)\|_{RC} < \frac12$. 
\ms 

Since the idea is crystal clear in this case, we will first prove this claim with $c_1 = 4\sqrt 3$ and only after that indicate the refinement that yields $c_1=\sqrt 3$.
It is easy to reduce the claim to the case when $x_k=x^*_k$ (and we will find $\hat x_k$ also self-adjoint). Let $S = \sum \vp_kx_k$. We set
\begin{equation}\label{eq5.6}
 \Phi = S1_{\{|S|\le 2\sqrt 3\}}.
\end{equation}
Here we use a rather bold notation: we denote by $1_{\{|S|\le \lambda\}}$ (resp. $1_{\{|S|> \lambda\}}$) the spectral projection   corresponding to
the set $[0,\lambda]$ (resp. $(\lambda,\infty)$) in the spectral decomposition of the
(self-adjoint) operator $|S|$.
Note $\Phi = \Phi^*$ (since we assume $S=S^*$). Let $F = S-\Phi = S1_{\{|S|>2\sqrt 3\}}$. By \eqref{eq5.4}
\[
 \left\|\left(\int S^4\ d{\bb P}\right)^{1/2}\right\| \le 3^{1/2} \left\|\sum x^2_k\right\| < 3^{1/2}.
\]
A fortiori 
\[
 \left\|\left(\int F^4\ d{\bb P}\right)^{1/2}\right\| < 3^{1/2}
\]
but since $F^4 \ge F^2(2\sqrt 3)^2$ this implies
\[
 \left\|\left(\int F^2\ d{\bb P}\right)^{1/2}\right\| < 1/2.
\]
Now if $\hat x_k = \int\ \Phi \vp_k\ d{\bb P}$, i.e.\ if $\Phi = \sum \vp_k\hat x_k + \Phi'$ with $\Phi' \in [\vp_k]^\perp \otimes M$ we find $F = \sum \vp_k(x_k-\hat x_k)-\Phi'$ and hence by \eqref{eq5.5} applied with $F$ in place of $\Phi$ (note $\Phi,F,\hat x_k$ are all self-adjoint)
\[
\|(x_k-\hat x_k)\|_{RC} = \|(x_k-\hat x_k)\|_C \le \left\|\left(\int F^2\ d{\bb P}\right)^{1/2}\right\| < 1/2,
\]
which proves our claim.

To obtain this with $c_1=\sqrt 3$ instead of $4\sqrt 3$ one needs to do the truncation \eqref{eq5.6} in a slightly less brutal way:\ just set
\[
 \Phi = S1_{\{|S|\le c\}}  + c1_{\{S>c\}} - c1_{\{S<-c\}}
\]
where $c=\sqrt 3/2$. A simple calculation then shows that $F=S-\Phi  = f_c(S)$ where $|f_c(t)|=1_{(|t|>c)}  (|t|-c)$. Since $|f_c(t)|^2 \le (4c)^{-2} t^4$ for all $t>0$ we have
\[
 \left(\int F^2 \ d{\bb P}\right)^{1/2} \le (2\sqrt 3)^{-1} \left(\int S^4\ d{\bb P}\right)^{1/2}
\]
and hence by \eqref{eq5.4}
\[
 \left\|\left(\int F^2\ d{\bb P}\right)^{1/2}\right\| < 1/2
\]
so that we obtain the claim as before using \eqref{eq5.5}, whence \eqref{eq5.1} with $c_1=\sqrt 3$. The proof of \eqref{eq5.1+} and \eqref{eq5.3} (i.e. the complex cases) with the constant $  c^{\bb C}_1= \tilde c^{\bb C}_1 = \sqrt 2$ follows entirely parallel steps, but  the reduction to   the self-adjoint case is not so easy, so the calculations are  slightly more involved. The relevant analogue of \eqref{eq5.4} in that case is \eqref{eq5.4+}.
\end{proof}

\begin{rem}
In \cite{Har} (see also \cite{P77,P7}) versions of the non-commutative Khintchine inequality in $L_4$ are proved for the sequences of functions of the form $\{e^{int}\mid n\in\Lambda\}$ in $L_2({\bb T})$ that satisfy the $Z(2)$-property in the sense that there is a constant $C$ such that
\[
 \sup_{n\in {\bb Z}} \text{card}\{(k,m) \in {\bb Z}^2\mid k\ne m,\   k-m\in \Lambda\} \le C.
\]
It is easy to see that the method from \cite{HM2} (to deduce the $L_1$-case from the $L_4$-one)
described in the proof of Theorem~\ref{thm5.1} applies equally well to the $Z(2)$-subsets of $\bb Z$ or of  arbitrary (possibly non-commutative) discrete groups. 
\end{rem}

Although we do not want to go too far in that direction, we cannot end this section without describing the non-commutative Khintchine inequality for values of $p$ other than $p=1$.

Consider a generalized (possibly non-commutative) measure space $(M,\tau)$ (recall we assume $\tau$ semi-finite). The space $L_p(M,\tau)$ or simply $L_p(\tau)$ can then be described as a complex interpolation space, i.e.\ we can use as a definition the isometric identity (here $1<p<\infty$)
\[
 L_p(\tau) = (L_\infty(\tau), L_1(\tau))_{\frac1p}.
\]
The case $p=2$ is of course obvious:\ for any $x_1,\ldots, x_n$ in $L_2(\tau)$ we have
\[
 \left(\int\left\|\sum \vp_kx_k\right\|^2_{L_2(\tau)} \ d{\bb P}\right)^{1/2} = \left(\sum \|x_k\|^2_{L_2(\tau)}\right)^{1/2},
\]
but the analogous 2-sided inequality for
\[
 \left(\int\left\|\sum \vp_kx_k\right\|^p_{L_p(\tau)} d{\bb P}\right)^{1/p}
\]
with $p\ne 2$ is not so easy to describe, in part because it depends very much whether $p<2$ or $p>2$ (and moreover the case $0<p<1$ is still open!). 

Assume $1\le p<\infty $ and $x_j\in L_p(\tau)$ $(j=1,\ldots, n)$. We will use the notation
\[
 \|(x_j)\|_{p,R} = \left\|\left(\sum x_jx^*_j\right)^{1/2}\right\|_p \quad \text{and}\quad \|(x_j)\|_{p,C} = \left\|\left(\sum x^*_jx_j\right)^{1/2}\right\|_p.
\]

\begin{rem}\label{rnot=c} Let $x= (x_j) $. Let $R(x)=\sum e_{1j}\otimes x_j$ (resp. $C(x)=\sum e_{j1}\otimes x_j$)
denote the row (resp. column) matrix with entries
enumerated as $x_1,\cdots,x_n$. We may clearly view
$R(x)$ and $C(x)$ as elements of $L_p(M_n\otimes M, {\rm tr}\otimes \tau)$. Computing $|R(x)^*|=(R(x)R(x)^*)^{1/2}$ and  $|C(x)|=(C(x)^*C(x))^{1/2}$, we find
$$\|(x_j)\|_{p,R} =\|   R(x)\|_{L_p(M_n\otimes M, {\rm tr}\otimes \tau)}\quad{\rm and}\quad \|(x_j)\|_{p,C} =\|   C(x)\|_{L_p(M_n\otimes M, {\rm tr}\otimes \tau)}.$$
This shows in particular that  $(x_j)\mapsto \|(x_j)\|_{p,R}$ and $(x_j)\mapsto\|(x_j)\|_{p,C}$ are norms on the space of finitely supported sequences in $L_p(\tau)$. Whenever $p\not= 2$, the latter norms are different. Actually, they
are not even equivalent with constants independent of $n$. Indeed 
a simple example is provided by the choice of
$M=M_n$  (or $M=B(\ell_2)$) equipped with the usual trace and  $x_j=e_{j1}$ $(j=1,\ldots, n)$. We have then
$\|(x_j)\|_{p,R} =n^{1/p}$ and $\|(x_j)\|_{p,C} =\|({}^t x_j)\|_{p,R}=n^{1/2}$.
  
\end{rem}

Now assume $1\le p\le 2$.  In that case we define
\[
|||(x_j)|||_p = \inf\{\|(y'_j)\|_{p,R} + \|(y''_j)\|_{p,C}\}
\]
where the infimum runs over all possible decompositions of $(x_j)$ of the form $x_j = y'_j+y''_j$, $y'_j,y''_j\in L_p(\tau)$. In sharp contrast, when $2\le p<\infty$ we define
\[
 |||(x_j)|||_p = \max\{\|(x_j)\|_{p,R}, \|(x_j)\|_{p,C}\}.
\]
The classical Khintchine inequalities (see \cite{H1+,Sz}) say that for any $0 < p  <\infty$ there are positive constants $A_p,B_p$
such that for any scalar sequence $x\in \ell_2$ we have
$$A_p (\sum |x_j|^2)^{1/2} \le  (\int |\sum \vp_jx_j |^p \ d{\bb P} )^{1/p} \le B_p (\sum |x_j|^2)^{1/2}.$$ 
The non-commutative Khintchine inequalities that follow are due to Lust-Piquard (\cite{LP1}) in case $1 < p \ne 2<\infty$.

\begin{thm}[Non-commutative Khintchine in $L_p(\tau)$]\label{thm5.3}
\begin{itemize}
\item[\rm (i)] If $1\le p\le 2$ there is a constant $c_p>0$ such that for any finite set $(x_1,\ldots, x_n)$ in $L_p(\tau)$ we have
\end{itemize}
\begin{equation}\label{eq5.8}
 \frac1{c_p}|||(x_j)|||_p \le \left(\int\left\|\sum \vp_jx_j\right\|^p_p \ d{\bb P}\right)^{1/p} \le |||(x_j)|||_p.
\end{equation}
\begin{itemize}
\item[\rm (ii)] If $2\le p<\infty$, there is a constant $b_p$ such that for any $(x_1,\ldots, x_n)$ in $L_p(\tau)$
\end{itemize}
\begin{equation}\label{eq5.9}
 |||(x_j)|||_p \le \left(\int\left\|\sum \vp_jx_j\right\|^p_p \ d{\bb P}\right)^{1/p} \le b_p|||(x_j)|||_p.
\end{equation}
Moreover, \eqref{eq5.8} and \eqref{eq5.9} also hold  for the sequences $(s_j)$, $(g^{\bb R}_j)$ or $(g^{\bb C}_j)$.
We will denote the corresponding constants respectively on one hand by $c_p^\CC$, $\tilde c_p$ and $\tilde c_p^\CC$, and on the other hand by $b_p^\CC$, $\tilde b_p$ and $\tilde b_p^\CC$.
\end{thm}

\begin{dfn}\label{cotype}
A Banach space $B$ is called ``of cotype 2'' if there is a constant $C>0$ such that for all finite sets $(x_j)$ in $B$ we have
\[
 \left(\sum\|x_j\|^2\right)^{1/2} \le C\left\|\sum \vp_jx_j\right\|_{L_2(B)}.
\]
 It is called ``of type 2'' if the inequality is in the reverse direction.
\end{dfn}

From the preceding Theorem, one recovers easily the known result from \cite{TJ} that $L_p(\tau)$ is of cotype 2 (resp.\ type 2) whenever $1\le p\le 2$ (resp.\ $2\le p<\infty$). In particular, the trace class $S_1$ is of cotype 2. Equivalently, the projective tensor
product $\ell_2\widehat \otimes \ell_2 $ is of cotype 2. However,
it is a long standing open problem whether $\ell_2\widehat \otimes \ell_2\widehat \otimes \ell_2 $ is of cotype 2 or  any cotype $p<\infty$ (cotype $p$ is the same as the preceding definition
but with  $(\sum\|x_j\|^p)^{1/p}$ in place of $(\sum\|x_j\|^2)^{1/2}$).

\n {\bf Remarks:}
\begin{itemize}
 \item[(i)] In the commutative case, the preceding result is easy to obtain by integration from the classical Khintchine inequalities, for which the best constants are known (cf.\ \cite{Sz,H1+,Sa}), they are respectively $c_p=2^{\frac1p-\frac12}$ for $1\le p\le p_0$ and $c_p = \|g^{\bb R}_1\|_p=$ for $p_0\le p<\infty$, where $1<p_0<2$ is determined by the equality
\[
 2^{\frac1{p_0}-\frac12} = \|g^{\bb R}_1\|_{p_0}.
\]
Numerical calculation yields $p_0=1.8474...$!
\item[(ii)] In the non-commutative case, say when $L_p(\tau) = S_p$ (Schatten $p$-class) not much is known on the best constants, except for the case $p=1$ (see \S \ref{sec9} for that). Note however that if $p$ is an even integer the (best possible) type 2 constant of $L_p(\tau)$ was computed already by Tomczak--Jaegermann in \cite{TJ}. One can deduce from this (see \cite[p. 193]{P4}) that the constant $b_p$ in \eqref{eq5.9} is $O(\sqrt p)$ when $p\to \infty$, which is at least the optimal order of growth, see also \cite{Buch2} for exact values of $b_p$
for $p$ an even integer.
Note however that, by \cite{LPP},  $c_p$ remains bounded when $1\le p\le 2$.
\item[(iii)] Whether \eqref{eq5.8} holds for $0<p<1$ is an open problem, although many natural consequences of this have been verified (see \cite{X1,P7}). See \cite{Kal,APe} for somewhat related topics.
\end{itemize}

The next result is meant to illustrate the power of \eqref{eq5.8} and \eqref{eq5.9}. We first introduce a problem. Consider a complex matrix $[a_{ij}]$ and let $(\vp_{ij})$ denote independent random signs so that $\{\vp_{ij}\mid i,j\ge 0\}$ is an i.i.d.\ family of $\{\pm 1\}$-valued variables with ${\bb P}(\vp_{ij} = \pm 1)=1/2$.\ms

\n {\bf Problem.} Fix $0<p<\infty$. Characterize the matrices $[a_{ij}]$ such that 
\[
 [\vp_{ij}a_{ij}] \in S_p
\]
for almost all choices of signs $\vp_{ij}$. 

Here is the solution:

\begin{cor}\label{cor5.4}
When $2\le p<\infty$, $[\vp_{ij} a_{ij}]\in S_p$ a.s.\ iff we have both
\[
 \sum\nolimits_i\left(\sum\nolimits_j |a_{ij}|^2\right)^{p/2}<\infty\quad \text{\rm and}\quad \sum\nolimits_j\left(\sum\nolimits_i |a_{ij}|^2\right)^{p/2}<\infty.
\]
When $0<p\le 2$, $[\vp_{ij}a_{ij}]\in S_p$ a.s.\ iff there is a decomposition $a_{ij} = a'_{ij} + a''_{ij}$ with $\sum_i(\sum_j|a'_{ij}|^2)^{p/2}<\infty$ and $\sum_j(\sum_i|a''_{ij}|^2)^{p/2}<\infty$.
\end{cor}

This was proved in \cite{LP1} for $1<p<\infty$ (\cite{LPP} for $p=1$) as an immediate consequence of Theorem \ref{thm5.3} (or Theorem \ref{thm5.1}) applied to the series
$\sum \vp_{ij} a_{ij} e_{ij} $, together with known general facts about random series of vectors on a Banach space (cf.\ e.g.\ \cite{Ka,IN}). The case $0<p<1$ was recently obtained in \cite{P7}.

\begin{cor}[Non-commutative Marcinkiewicz--Zygmund]\label{cor5.5}
Let $L_p(M_1,\tau_1), L_p(M_2,\tau_2)$ be generalized (possibly non-commutative) $L_p$-spaces, $1\le p<\infty$. Then there is a constant $K(p)$ such that any bounded linear map $u\colon \ L_p(\tau_1)\to L_p(\tau_2)$ satisfies the following inequality:\ for any $(x_1,\ldots, x_n)$ in $L_p(\tau_1)$
\[
 |||(ux_j)|||_p \le K(p)\|u\|\ |||(x_j)|||_p.
\]
We have $K(p) \le c_p$ if $1\le p\le 2$ and $K(p) \le b_p$ if $p\ge 2$.
\end{cor}

\begin{proof}
 Since we have trivially for any fixed $\vp_j= \pm 1$
\[
 \left\|\sum \vp_j ux_j\right\|_p \le \|u\| \left\|\sum \vp_jx_j\right\|_p,
\]
the result is immediate from Theorem \ref{thm5.3}.
\end{proof}

\begin{rmk}
The case $0<p<1$ is still open.
\end{rmk}

\begin{rmk}
 The preceding corollary is false (assuming $p\ne 2$) if one replaces $|||~~|||_p$ by either $\|(\cdot)\|_{p,C}$ or $\|(\cdot)\|_{p,R}$. Indeed the transposition $u\colon \ x\to {}^tx$ provides a counterexample.
See   Remark \ref{rnot=c}.
\end{rmk}

\section{Maurey factorization}\label{secm}

We refer the reader to \S \ref{sechb} for $p$-summing operators.
Let $(\Omega,\mu)$ be a measure space. In his landmark paper \cite{Ros} Rosenthal made the following observation: let $u\colon \ L_\infty(\mu)\to X$ be a operator into a Banach space $X$ that is 2-summing 
 with $\pi_2(u)\le c$, or equivalently satisfies \eqref{eq3.1pie}. If $u^*(X^*)\subset L_1(\mu)\subset L_\infty(\mu)^*$, then the probability $\lambda$ associated to \eqref{eq3.1pie} that is a priori in $L_\infty(\mu)^*_+$ can be chosen absolutely continuous with respect to $\mu$, so we can take $\lambda=f\cdot\mu$ with $f$ a probability density on $\Omega$. Then we find $\forall x\in X$ $\|u(x)\|^2 \le c\int |x|^2 f\ d\mu$. Inspired by this, Maurey developed in his thesis \cite{Ma2} an extensive factorization theory for operators either with range $L_p$ or with domain a subspace of $L_p$. Among his many results, he showed that for any subspace $X\subset L_p(\mu)$ $p>2$ and any operator $u\colon \ E\to Y$ into a Hilbert space $Y$ (or a space of cotype 2), there is a probability density $f$ as above such that $\forall x\in X$ $\|u(x)\| \le C(p)\|u\| (\int|x|^2 f^{1-2/p}d\mu )^{1/2}$ where the constant $C(p)$ is independent of $u$. 

Note that multiplication by $f^{\frac12-\frac1p}$ is an operator of norm $\le 1$ from $L_p(\mu)$ to $L_2(\mu)$. For general subspaces $X\subset L_p$, the constant $C(p)$ is unbounded when $p\to \infty$ but if we restrict to $X=L_p$ we find a bounded constant $C(p)$. In  fact if we take $Y=\ell_2$ and let $p\to\infty$ in that case we obtain roughly the little GT (see Theorem~\ref{thm2.1} above). It was thus natural to ask whether non-commutative analogues of these results hold. This question was answered first by the Khintchine inequality for non-commutative $L_p$ of \cite{LP}. Once this is known, the analogue of Maurey's result follows by the same argument as his, and
the   constant $C(p)$ for this generalization remains of course unbounded when $p\to\infty$. The case of a bounded operator $u\colon \ L_p(\tau)\to Y$ on a non-commutative $L_p$-space (instead of a subspace) turned out to be much more delicate:\ The problem now was to obtain $C(p)$ bounded when $p\to\infty$, thus providing a generalization of the non-commutative little GT. This question was completely settled in \cite{LP2} and in a more general setting in \cite{LPX}.

One of Rosenthal's main achievements in \cite{Ros} was the proof that for any reflexive subspace $E\subset L_1(\mu)$ the adjoint quotient mapping $u^*\colon \ L_\infty (\mu)\to E^*$ was $p'$-summing for some $p'<\infty$, and hence the inclusion $E\subset L_1(\mu)$ factored through $L_p(\mu)$ in the form $E\to L_p(\mu)\to L_1(\mu)$ where the second arrow represents a multiplication operator by a nonnegative function in $L_{p'}(\mu)$. The case when $E$ is a Hilbert space could be considered essentially as a corollary of the little GT, but reflexive subspaces definitely required  deep new ideas.

Here again it was natural to investigate the non-commutative case. First in  \cite{P03} the author proved a non-commutative version of the so-called Nikishin--Maurey factorization through weak-$L_p$ (this also  improved the commutative factorization). However, this version was only partially satisfactory.
A further result is given in \cite{Ran}, but  a fully satisfactory non-commutative analogue of Rosenthal's result was finally achieved by Junge and Parcet in \cite{JPa}. These authors went on to study many related problems involving non-commutative Maurey factorizations, (also in the operator space framework see \cite{JPa2}).

\section{Best constants (Non-commutative case)}
\label{sec9}

Let $K'_G$ (resp. $k'_G$) denote the best possible constant
in \eqref{eq3.2+} (resp.  \eqref{eq3.2++}). Thus  \eqref{eq3.2+} and  \eqref{eq3.2++}
show that $K'_G\le 2 $ and $k'_G\le 2 $. Note that,
by the same simple argument as for \eqref{kK}, we have
$$k'_G\le K'_G.$$
We will show in this section (based on \cite{HI})
that, in sharp contrast with the  commutative case,
we have
$$k'_G= K'_G=2.$$
By the same reasoning as for Theorem \ref{thm2.1}
the best constant in \eqref{eq4.1bis} is equal to
$\sqrt{k'_G}$, and hence $\sqrt{2}$ is the best possible constant in \eqref{eq4.1bis}
(and a fortiori $1$ is optimal in \eqref{eq4.1}).

 The next Lemma is the non-commutative version
 of Lemma \ref{lem2.1}.
 \begin{lem}\label{lem9.1}
Let $(M,   \tau)$ be a von Neumann algebra equipped with a normal, faithful, semi-finite trace $\tau$. Consider   $g_j\in L_1(\tau)$,
 $x_j\in M$ ($1\le j\le N$) and  positive numbers $a,b$. Assume that
$(g_j)$ and $(x_j)$ are biorthogonal (i.e.
$\langle g_i,x_j \rangle=\tau( x_j^*g_i)  =0$ if $i\not=j$ and $=1$
otherwise) and such that
$$\forall (\alpha_j)\in \CC^N\quad a(\sum |\alpha_j|^2)^{1/2}\ge \|\sum \alpha_j g_j\|_{L_1(\tau)} \ {\rm and}\ \max\left\{ \left\|\sum x^*_jx_j\right\|_M^{1/2}, \left\|\sum x_jx^*_j\right\|_M^{1/2}\right\} \le b\sqrt{N}.$$
Then
$ b^{-2}a^{-2}\le k'_G$ (and a fortiori $ b^{-2}a^{-2}\le K'_G$).
\end{lem}
 \begin{proof} We simply repeat the argument used for Lemma \ref{lem2.1}.
\end{proof}

 \begin{lem}[\cite{HI}]\label{lem9.2} Consider an integer $n\ge1$. Let $N=2n+1$ and $d=\binom{2n+1}{n}=\binom{2n+1}{n+1}$.
  Let $\tau_d$ denote the normalized trace on the space $M_d$
of $d \times d$ (complex) matrices. There 
are $x_1,\cdots,x_N$ in $M_d$ such that
$\tau_d(x_i^*x_j)=\delta_{ij}$ for all $i,j$ (i.e. $(x_j)$ is orthonormal in $L_2(\tau_d)$), satisfying
$$\sum x_j^*x_j=\sum x_jx_j^*= {N} I\ 
{\rm and \ hence: }
    \max\left\{ \left\|\sum x^*_jx_j\right\|^{1/2}, \left\|\sum x_jx^*_j\right\|^{1/2}\right\}\le \sqrt{N},$$ and moreover such that
 \begin{equation}\label{eq9.111}   \forall (\alpha_j)\in \CC^N\quad 
 \|\sum \alpha_j x_j\|_{L_1(\tau_d)} =((n+1)/(2n+1))^{1/2}(\sum |\alpha_j|^2)^{1/2}.\end{equation}
\end{lem}

 \begin{proof}  Let $H$ be a $(2n+1)$-dimensional Hilbert space with orthonormal basis $e_1,\cdots,e_{2n+1}$. We will implicitly use 
 analysis familiar in the context of the antisymmetric
Fock space. Let $H^{\wedge k}$ denote the antisymmetric $k$-fold tensor product. Note
that $\dim(  H^{\wedge k})=\dim(  H^{\wedge (2n+1-k)})= \binom{2n+1}{k}$ and hence $\dim(  H^{\wedge  (n+1)})=\dim(  H^{\wedge n})=d$. 
Let $$c_j\colon H^{\wedge n}\to H^{\wedge (n+1)} \quad {\rm and} \quad c_j^* \colon H^{\wedge (n+1)}\to H^{\wedge n}$$ be
respectively the ``creation" operator 
associated to $e_j$ ($1\le j\le 2n+1$) defined by $c_j (h)=e_j\wedge h$ and   its adjoint
  the so-called   ``annihilation" operator. 
For any ordered subset $J\subset [1,\cdots,N]$, say
$J=\{j_1,\dots,j_k\}$ with $j_1<\cdots<j_k$ we denote
$e_J=e_{j_1}\wedge\cdots\wedge e_{j_k}$.
Recall that $\{e_J\mid |J|=k\}$ is an orthonormal basis of $H^{\wedge k}$. It is easy to check that
$c_jc^*_j$ is the orthogonal projection from  $H^{\wedge n+1}$ onto
${\rm span}\{e_J\mid |J|=n+1, j\in J\}$, while
$c_j^*c_j$ is the orthogonal projection from  $H^{\wedge n}$ onto
${\rm span}\{e_J\mid |J|=n, j\not\in J\}$.
In addition, 
   for each $J$ with $|J|=n+1$, we have $\sum_{j=1}^{2n+1} c_jc^*_j (e_J)=|J|e_J=(n+1)e_J$, and similarly if $|J|=n$ we have
$\sum_{j=1}^{2n+1} c^*_jc_j (e_J)=(N-|J|) e_J=(n+1)e_J$. (Note that $\sum_{j=1}^{2n+1} c_jc^*_j$ is the celebrated ``number operator"). Therefore,
\begin{equation}\label{eq9.1}
  \sum\nolimits_{j=1}^{2n+1} c_jc^*_j= (n+1)I_{H^{\wedge (n+1)}}\quad {\rm and}\quad \sum\nolimits_{j=1}^{2n+1} c^*_jc_j= (n+1)I_{H^{\wedge (n)}}.
\end{equation}
In particular this implies
$\tau_d(\sum\nolimits_{j=1}^{2n+1} c^*_jc_j)=n+1$, and since, by symmetry $\tau_d( c^*_jc_j)=\tau_d( c^*_1c_1)$ for all $j$,
we must have  $\tau_d( c^*_1c_1)=(2n+1)^{-1}(n+1)$. Moreover, since $c^*_1c_1$ is a projection,
we also find $\tau_d (| c_1|)=\tau_d (| c_1|^2)=\tau_d( c^*_1c_1)=(2n+1)^{-1}(n+1)$ (which can also be checked by elementary combinatorics).\\
There is a well known ``rotational invariance" in this context
(related to what physicists call ``second quantization"), from which we may extract the following fact.
Recall $N=2n+1$. 
Consider $(\alpha_j)\in \CC^N$ and let $h=\sum \bar \alpha_j e_j\in H$. Assume $\|h\|=1$. 
Then roughly  ``creation by $h$" is equivalent to creation by $e_1$, i.e. to
$c_1$. More precisely, Let $U$ be any unitary operator on $H$ such that
$Ue_1=h$. Then $U^{\wedge k}$ is unitary on  $H^{\wedge k}$ for any $k$, and it is easy to check that
$(U^{\wedge (n+1)})^{-1}(\sum \alpha_j c_j) U^{\wedge n}=   c_1$.
  This implies
   \begin{equation}\label{eq9.2}\tau_d( | \sum \alpha_j c_j|^2)=\tau_d( |  c_1|^2) 
   \quad {\rm and}\quad \tau_d( | \sum \alpha_j c_j|)=\tau_d( |  c_1|).\end{equation}
   Since $\dim(  H^{\wedge n+1})=\dim(  H^{\wedge n})=d$, we may (choosing orthonormal bases) identify $c_j$ with
  an element of $  M_d$. 
  Let then $x_j=c_j(\frac {2n+1}{n+1})^{1/2}$. By  the first equality in \eqref{eq9.2}, $(x_j)$ are orthonormal in $L_2(\tau_d)$ and 
  by \eqref{eq9.1} and   the rest of \eqref{eq9.2} we have the announced result.
  \end{proof}
 The next statement and \eqref{eq2+} imply,  somewhat surprisingly, that  $K'_G\not= K_G$.
 \begin{thm}[\cite{HI}]\label{thm9.1}
$$   k'_G= K'_G=2.$$
\end{thm}
\begin{proof}  By Lemma \ref{lem9.2}, the assumptions of
Lemma \ref{lem9.1} hold with
$a^2=(n+1)/(2n+1)\to 1/2$  and $b=1$. Thus we have
$k'_G\ge (ab)^{-2}=2$,   and hence $K'_G\ge k'_G\ge 2$.  But by Theorem \ref{thm3.1} we already know
that $ K'_G\le 2$.
\end{proof}

The next result is in sharp contrast with the commutative case: for the classical Khintchine inequality
in $L_1$ for the Steinhaus variables $(s_j)$, Sawa (\cite{Sa}) showed that the best constant
is $1/\|   g^{\bb C}_1\|_1=2/\sqrt{\pi}< \sqrt2$ (and of course this   obviously holds as well for the
sequence $(g^{\bb C}_j)$). 

 \begin{thm}[\cite{HM1}]\label{thm9.2} The best constants in \eqref{eq5.1+} and \eqref{eq5.3} are given by
${ c}^{\bb C}_1={\tilde c}^{\bb C}_1=\sqrt2$.
\end{thm}
\begin{proof} By Theorem \ref{thm5.1} we already know the upper bounds, so it suffices to prove
${ c}^{\bb C}_1\ge \sqrt2$ and ${\tilde c}^{\bb C}_1\ge\sqrt2$.
The reasoning done to prove \eqref{eq2.4} can be repeated
to show that $k_G'\le { c}^{\bb C}_1$ and $k_G'\le {\tilde c}^{\bb C}_1$,
thus yielding the lower bounds. Since it is instructive to exhibit the
culprits realizing the equality case, we include a direct deduction.
By \eqref{c+rdual}  and   Lemma \ref{lem9.2} we have
$
N\le |||(x_j)|||_1   \max\{\|(x_j)\|_C ,\|(x_j)\|_R\} \le |||(x_j)|||_1 N^{1/2}$,
and hence $N^{1/2}\le  |||(x_j)|||_1 $. But by \eqref{eq9.111} we have
$\int \|\sum s_j x_j\|_{L_1(\tau_d)}= ((n+1)/(2n+1))^{1/2} N^{1/2}$ so we must have
 $ N^{1/2}\le  |||(x_j)|||_1 \le { c}^{\bb C}_1 ((n+1)/(2n+1))^{1/2} N^{1/2}$
 and hence we obtain $((2n+1)/(n+1))^{1/2} \le { c}^{\bb C}_1$,
  which, letting $n\to \infty$,  yields $\sqrt2 \le  { c}^{\bb C}_1$. The proof that 
  ${\tilde c}^{\bb C}_1\ge\sqrt2$ is identical since by   the strong law of large numbers
  $(\int N^{-1}\sum_1^N | g^{\bb C}_j |^2)^{1/2} \to 1$ when $n\to \infty$.
 \end{proof}

\begin{rem} Similar calculations apply for Theorem \ref{thm5.3}: one finds that
for any $1\le p\le 2$,  ${ c}^{\bb C}_p\ge 2^{1/p-1/2}$ and
also $\tilde { c}^{\bb C}_p\ge 2^{1/p-1/2}$. Note that these lower bounds
are trivial for the sequence $(\vp_j)$, e.g. the fact that
 $c_1\ge \sqrt 2$ is trivial already from the classical case
(just consider $\vp_1+\vp_2$ in \eqref{eq2.1}),
but then in that case the proof of Theorem \ref{thm5.1}
from \cite{HM1} only gives us $c_1\le \sqrt 3$. 
So the best value of $c_1$ is still unknown !
\end{rem}

\section{$\pmb{C^*}$-algebra tensor products, Nuclearity}\label{sec10}

Recall first that a $*$-homomorphism $u\colon \ A\to B$ between  two $C^*$-algebras
is an algebra homomorphism respecting the involution, i.e. such that   $u(a^*)=u(a)^*$
 for all $a\in A$. Unlike Banach space morphisms, these maps are somewhat ``rigid":  we have
  $\|u\|=1$ unless $u=0$ and $u$ injective ``automatically" implies
   $u$   isometric.

The (maximal and minimal) tensor products of $C^*$-algebras were introduced in the late 1950's (by Guichardet and Turumaru). The underlying idea most likely was   inspired by Grothendieck's work on the injective and projective Banach space tensor products.

Let $A,B$ be two $C^*$-algebras. Note that their algebraic tensor product $A\otimes B$ is a $*$-algebra. For any $t = \sum a_j\otimes b_j\in A\otimes B$ we define
  \[
 \|t\|_{\max} = \sup\{\|\varphi(t)\|_{B(H)}\}
\]
where the sup runs over all $H$ and all $*$-homomorphisms $\varphi\colon \ A\otimes B\to B(H)$. Equivalently, we have 

\begin{equation}\label{eq-max}
 \|t\|_{\max} = \sup\left\{\left\|\sum \sigma(a_j) \rho(b_j)\right\|_{B(H)}\right\}
\end{equation}
where the sup runs over all $H$ and all possible pairs of $*$-homomorphisms $\sigma\colon \ A\to B(H)$, $\rho\colon \ B\to B(H)$ (over the same $B(H)$) with commuting ranges. 

Since we have automatically $\|\sigma\| \le 1$ and $\|\rho\|\le 1$, we note that $\|t\|_{\max} \le \|t\|_\wedge$. The completion of $A\otimes B$ equipped with $\|\cdot\|_{\max}$ is denoted by $A\otimes_{\max} B$ and is called the maximal tensor product. This enjoys the following ``projectivity'' property:\ If $I\subset A$   is a closed 2-sided (self-adjoint) ideal, 
then $I\otimes_{\max}B\subset A\otimes_{\max}B$
(this is special to ideals) and the quotient $C^*$-algebra $A/I$   satisfies
\begin{equation}\label{eq10.1}
 (A/I) \otimes_{\max} (B) \simeq (A\otimes_{\max}B)/(I\otimes_{\max}B).
\end{equation}
The minimal tensor product can be defined as follows. Let $\sigma\colon \ A\to B(H)$ and $\rho\colon \ B\to B(K)$ be isometric $*$-homomorphisms. We set $\|t\|_{\min} = \|(\sigma \otimes\rho)(t)\|_{B(H\otimes_2 K)}$. It can be shown that this is independent of the choice of such a pair $(\sigma,\rho)$. We have
\[
 \|t\|_{\min} \le \|t\|_{\max}.
\]
Indeed, in the unital case just observe $a\to \sigma(a)\otimes 1$ and $b\to 1\otimes \rho(b)$ have commuting ranges. (The general case is similar using approximate units.)

The  completion of $(A\otimes B, \|\cdot\|_{\min})$ is denoted by $A\otimes_{\min} B$ and is called the minimal tensor product. It enjoys the following ``injectivity'' property:\ Whenever $A_1\subset A$ and $B_1\subset B$ are $C^*$-subalgebras, we have
\begin{equation}\label{eq10.2}
 A_1 \otimes_{\min} B_1 \subset A \otimes_{\min}B\quad \text{(isometric embedding).}
\end{equation}
However, in general $\otimes_{\min}$ (resp.\ $\otimes_{\max}$) does not satisfy the ``projectivity'' \eqref{eq10.1} (resp.\ ``injectivity'' \eqref{eq10.2}).

By a $C^*$-norm on $A\otimes B$, we mean any norm $\alpha$ adapted to the $*$-algebra structure and in addition such that $\alpha(t^*t) = \alpha(t)^2$. Equivalently this means that there is for some $H$ a pair of $*$-homomorphism $\sigma\colon \ A\to B(H)$, $\rho \colon \  B\to B(H)$ with commuting ranges such that $\forall t = \sum a_j\otimes b_j$
\[
 \alpha(t) = \left\|\sum \sigma(a_j)\rho(b_j)\right\|_{B(H)}.
\]
It is known that necessarily
\[
 \|t\|_{\min} \le \alpha(t) \le \|t\|_{\max}.
\]
Thus $\|\cdot\|_{\min}$ and $\|\cdot\|_{\max}$ are respectively the smallest and largest $C^*$-norms on $A\otimes B$. The comparison with the Banach space counterparts is easy to check:
\begin{equation}\label{comp} \|t\|_{\vee}\le 
 \|t\|_{\min}   \le \|t\|_{\max}\le \|t\|_{\wedge}.
\end{equation}

\begin{rmk}
We recall that a $C^*$-algebra admits a unique $C^*$-norm. Therefore two $C^*$-norms that are equivalent on $A\otimes B$ must necessarily be identical (since the completion has a unique $C^*$-norm). 
\end{rmk}

\begin{dfn}\label{dfn10.1}
 A pair of $C^*$-algebras $(A,B)$ will be called a nuclear pair if $A\otimes_{\min} B = A \otimes_{\max} B$ i.e.
\begin{equation}
 \|t\|_{\min} = \|t\|_{\max}.\tag*{$\forall t\in A\otimes B$}
\end{equation}
Equivalently, this means that there is a unique $C^*$-norm on $A\otimes B$.
\end{dfn}

\begin{dfn}\label{dfn10.2}
A $C^*$-algebra $A$ is called nuclear if for any $C^*$-algebra $B$, $(A,B)$ is a nuclear pair.
\end{dfn}

The reader should note of course the analogy with Grothendieck's previous notion of a ``nuclear locally convex space.'' Note however that a Banach space $E$ is nuclear in his sense (i.e.\ $E\widehat\otimes F = E \check \otimes F$ for any Banach space $F$) iff it is finite dimensional.

The following classical result is due independently to Choi-Effros and Kirchberg (see e.g. \cite{BO}).
\begin{thm}\label{cek}
A unital $C^*$-algebra $A$ is nuclear
iff there is a net of finite rank maps of the form $A
{\buildrel v_\alpha \over \longrightarrow} M_{n_\alpha } {\buildrel
w_\alpha \over \longrightarrow} A$ where $v_\alpha ,w_\alpha $ are 
maps with $\|v_\alpha \|_{cb}\le 1$, $\|w_\alpha \|_{cb}\le 1$, which tends
pointwise to the identity.
\end{thm}

For example, all commutative $C^*$-algebras, the algebra $K(H)$ of all compact operators on a Hilbert space $H$, the Cuntz algebra or more generally all ``approximately finite dimensional'' $C^*$-algebras are nuclear. \\
Whereas for counterexamples, let ${\bb F}_N$ denote the free group on $N$ generators, with $2\le N\le \infty$. (When $N=\infty$ we mean of course countably infinitely many generators.)
The first   non-nuclear $C^*$-algebras were found using free groups:\ both the full and reduced $C^*$-algebra of ${\bb F}_N$ are \emph{not} nuclear. Let us recall their definitions.\\
For any  discrete group $G$, the full (resp.\ reduced) $C^*$-algebra of $G$, denoted by $C^*(G)$ (resp.\ $C^*_\lambda(G)$) is defined as the $C^*$-algebra generated by
 the universal unitary representation of $G$  (resp.\ the left regular representation of $G$ acting by translation on $\ell_2(G)$).  

It turns out that ``nuclear'' is the analogue for $C^*$-algebras of ``amenable'' for groups. Indeed, from work of Lance (see \cite{Ta3}) it is known that $C^*(G)$ or $C^*_\lambda(G)$ is nuclear iff $G$ is amenable (and in that case $C^*(G)=C^*_\lambda(G)$). More generally, an ``abstract'' $C^*$-algebra $A$ is nuclear iff it is amenable as a Banach algebra (in B.E.\ Johnson's sense). This means by definition that any bounded derivation $D\colon \ A\to X^*$ into a dual $A$-module is inner. The fact that amenable implies nuclear was proved by Connes as a byproduct of his deep investigation of injective von~Neumann algebras:\ $A$ is nuclear iff the von~Neumann algebra $A^{**}$ is injective, i.e.\ iff (assuming $A^{**} \subset B(H)$) there is a contractive projection $P\colon\ B(H)\to A^{**}$, cf.\ Choi--Effros \cite{CE}. The converse implication nuclear $\Rightarrow$ amenable was proved by Haagerup in \cite{H2+}. This crucially uses the non-commutative GT (in fact since nuclearity implies the approximation property, the original version of \cite{P1} is sufficient here). 
 For emphasis, let us   state it:
 
%We will outline Haagerup's argument in a subsequent version of this paper.

\begin{thm}\label{thm10.3}
 A $C^*$-algebra is nuclear iff it is amenable (as a Banach algebra).
\end{thm}

Note that this implies that nuclearity passes to quotients (by a closed two sided ideal) but that is
not at all easy to see directly on the definition of nuclearity.

While the meaning of nuclearity for a $C^*$-algebra seems by now fairly well understood, it is not so for pairs, as reflected by  Problem  \ref{prbl10.4} below, proposed and emphasized by Kirchberg \cite{Kir}. 
In this context, Kirchberg \cite{Kir} proved the following striking result (a simpler proof was given in \cite{P}):
\begin{thm}\label{thm7.5}
The pair $(C^*({\bb F}_\infty), B(\ell_2))$ is a nuclear pair. 
\end{thm}

Note that any separable $C^*$-algebra is a quotient of $C^*({\bb F}_\infty)$, so $C^*({\bb F}_\infty)$ can be viewed as ``projectively universal.'' Similarly, $B(\ell_2)$ is ``injectively universal'' in the sense that any separable $C^*$-algebra embeds into $B(\ell_2)$. 

Let $A^{op}$ denote the ``opposite'' $C^*$-algebra of $A$, i.e.\ the same as $A$ but with the product in reverse order. For various reasons, it was conjectured by Lance that $A$ is nuclear iff the pair $(A,A^{op})$ is nuclear. This conjecture was supported by the case $A = C^*_\lambda(G)$ and many other special cases where it did hold. Nevertheless a remarkable counterexample was found by Kirchberg in \cite{Kir}. (See \cite{P2} for the Banach analogue of this counterexample, an infinite dimensional (i.e. non-nuclear!)  Banach space $X$ such that
the injective and projective tensor norms are equivalent
on $X\otimes X$.) Kirchberg then wondered whether
one could simply take either $A=B(H)$ (this was disproved in \cite{JP} see Remark \ref{rembb} below)
or $A=C^*({\bb F}_\infty)$ (this is the still open Problem \ref{prbl10.4} below).
Note that if either  $A=B(H)$ or $A=C^*({\bb F}_\infty))$, we have  $A\simeq A^{op}$, so the
pair $(A,A^{op})$ can be replaced by $(A,A)$. Note also that, for that matter, ${\bb F}_\infty$ can be replaced by ${\bb F}_2$ or any non Abelian free group (since ${\bb F}_\infty\subset {\bb F}_2$).

For our exposition, it will be convenient to adopt the following definitions (equivalent to the more standard ones by \cite{Kir}).

\begin{dfn}
Let $A$ be a $C^*$-algebra. We say that $A$ is WEP if $(A,C^*({\bb F}_\infty))$ is a nuclear pair. We say that $A$ is LLP if $(A,B(\ell_2))$ is a nuclear pair. We say that $A$ is QWEP if it is a quotient (by a closed, self-adjoint, 2-sided ideal) of a WEP $C^*$-algebra. \end{dfn}
Here WEP stands for weak expectation property (and LLP for local lifting property). 
Assuming $A^{**}\subset B(H)$,    $A$ is WEP iff there is
a completely positive $T\colon \ B(H)\to A^{**}$
(``a weak expectation") such that $T(a)=a$ for all $a$ in $A$.

Kirchberg actually proved the following generalization of Theorem \ref{thm7.5}:
\begin{thm}\label{thm7.5bis}
If $A$ is LLP and $B$ is WEP then  $(A,B)$ is a nuclear pair.
\end{thm}
\begin{rem}\label{rembb} S. Wasserman proved in \cite{Wa} that 
$B(H)$ and actually any von Neumann algebra $M$
(except for  a very special class) are \emph{not} nuclear.
The exceptional class is formed of all finite direct sums
of tensor products of a commutative algebra with a matrix algebra  (``finite type I"). The proof in \cite{JP} that $(B(H),B(H))$ (or $(M,M)$) is \emph{not} a nuclear pair is based
on Kirchberg's idea that, otherwise,  the set formed of all the finite dimensional operator spaces, equipped with the distance $\delta_{\text{cb}}(E,F)=\log d_{\text{cb}}(E,F)$ ($d_{\text{cb}}$ is defined in \eqref{eqdist} below), would have to  be a separable metric space. The original proof in
\cite{JP} that the latter metric space (already for $3$-dimensional  operator spaces) is actually non-separable used   GT for exact operator spaces
(described in \S \ref{sec13bis} below), but soon after,
several different, more efficient 
proofs were found (see  chapters 21 and 22 in  \cite{P4}). 

\end{rem}
The next problem is currently perhaps the most important open question
in Operator Algebra Theory. Indeed, as Kirchberg proved (see the end of \cite{Kir}),
this is also equivalent to a fundamental question raised by Alain Connes \cite{Co} on von Neumann algebras:
whether any separable ${\rm II}_1$-factor 
(or more generally any separable finite von Neumann algebra) embeds into an ultraproduct of matrix algebras.
We refer the reader to \cite{Oz} for an excellent exposition of this topic.

\begin{prbl}\label{prbl10.4}
Let $A = C^*({\bb F}_N)$ $(N>1)$. Is there a unique $C^*$-norm on $A\otimes A$? Equivalently, is $(A,A)$ a nuclear pair?
\end{prbl}
Equivalently:
\begin{prbl}\label{prbl10.4bis}
Is every $C^*$-algebra QWEP ?
\end{prbl}

Curiously, there seems to be a connection between  GT and  Problem \ref{prbl10.4}. Indeed, like GT, the latter problem can be phrased as the identity of two tensor products  on $\ell_1\otimes \ell_1$,  or  more precisely
on $M_n\otimes \ell_1\otimes \ell_1$, which might be related to the operator space version of GT presented in the sequel.
To explain this, we first reformulate Problem \ref{prbl10.4}.

Let $A = {C}^*({\bb F}_\infty)$. Let $(U_j)_{j\ge 1}$ denote the unitaries of ${C}^*({\bb F}_\infty)$ that correspond to the free generators of $A$. For convenience of notation we set $U_0=I$. It is easy to check that the closed  {linear} span $E\subset A$ of $\{U_j\mid j\ge 0\}$ in $A$ is isometric to $\ell_1$ and that $E$ generates $A$ as a ${C}^*$-algebra.

Fix $n\ge 1$. Consider a finitely supported family $a = \{a_{jk}\mid j,k\ge 0\}$ with $a_{jk}\in M_n$ for all $j,k\ge 0$. We denote, again for $A = {C}^*({\bb F}_\infty)$:
\begin{align*}
 \|a\|_{\min} = \left\|\sum a_{jk}\otimes U_j\otimes U_k\right\|_{M_n(A\otimes_{\min}A)} \quad{\rm and}\quad
\|a\|_{\max} = \left\|\sum a_{jk}\otimes U_j \otimes U_k\right\|_{M_n(A\otimes_{\max}A)}.
\end{align*}
Then, on one hand, going back to the definitions, it is easy to check that
\begin{equation}\label{10.3}
 \|a\|_{\max} = \sup\left\{\left\|\sum a_{jk} \otimes u_jv_k\right\|_{M_n(B(H))}\right\}
\end{equation}
where the supremum runs over all $H$ and all possible unitaries  $u_j,v_k$ on the same Hilbert space $H$ such that $u_jv_k = v_ku_j$ for all $j,k$. On the other hand, using the known fact that $A$ embeds into a direct sum of matrix algebras (due to M.D. Choi, see e.g. \cite[\S 7.4]{BO}), one can check that
\begin{equation}\label{10.4}
 \|a\|_{\min} = \sup\left\{\left\|\sum a_{jk}\otimes u_jv_k\right\|_{M_n(B(H))}\right\}
\end{equation}
where the sup is as in \eqref{10.3} except that we restrict it to  all \emph{finite dimensional}
Hilbert spaces  $H$. 

\n We may ignore the restriction $u_0=v_0=1$ because we can always replace 
$(u_j,v_j)$ by $(u_0^{-1} u_j, v_j v_0^{-1})$ without changing either \eqref{10.3} or  \eqref{10.4}.

The following is implicitly in \cite{P}.
\begin{pro}\label{simK}
Let $A = {C}^*({\bb F}_\infty)$. The following assertions are equivalent:
\begin{itemize}
\item[\rm (i)] $A\otimes_{\min} A = A\otimes_{\max} A$ (i.e.\ Problem \ref{prbl10.4} has a positive solution).
\item[\rm (ii)] For any $n\ge 1$ and any $\{a_{jk}\mid j,k\ge 0\}\subset M_n$ as above the norms \eqref{10.3} and \eqref{10.4} coincide i.e.\ $\|a\|_{\min} = \|a\|_{\max}$.
\item[\rm (iii)] The identity $\|a\|_{\min} = \|a\|_{\max}$ holds for all $n\ge 1$ but merely for all families $\{a_{jk}\}$ in $M_n$ supported in the union of $\{0\}\times \{0,1,2\}$ and $\{0,1,2\}\times \{0\}$.
\end{itemize}
\end{pro}

\begin{thm}[\cite{Ts1}]\label{ts1}
 Assume $n=1$ and $a_{jk}\in {\bb R}$ for all $j,k$. Then $$\|a\|_{\max} = \|a\|_{\min},$$ and  these norms coincide with the $H'$-norm of
$
 \sum a_{jk}e_j\otimes e_k$  {in}  $ \ell_1 \otimes \ell_1$, that we denote by $\|a\|_{H'}$.\\
Moreover, these are all equal to
\begin{equation}\label{tsi2}
\sup \|\sum a_{jk} u_jv_k\|
\end{equation}
where the sup runs over all $N$ and all self-adjoint unitary $N\times N$  matrices
 such that $u_jv_k=v_ku_j$ for all $j,k$.
\end{thm}

\begin{proof}
  Recall  (see \eqref{eq0.6}) that $\|a\|_{H'}\le 1$ iff for any unit vectors $x_j,y_k$ in a Hilbert space we have
\[
 \left|\sum a_{jk} \langle x_j,y_k\rangle\right| \le 1.
\]
Note that, since $a_{jk}\in \bb R$, whether we work  with real or complex Hilbert spaces does not affect this condition.
The resulting $H'$-norm is the same.
We have trivially
\[
 \|a\|_{\min} \le \|a\|_{\max} \le \|a\|_{H'},
\]
so it suffices to check $\|a\|_{H'} \le \|a\|_{\min}$. Consider unit vectors $x_j,y_k$ in a {\it real} Hilbert space $H$. We may assume that $\{a_{jk}\}$ is supported in $[1,\ldots, n]\times [1,\ldots, n]$ and that $\dim(H)=n$. From classical facts on ``spin systems'' (Pauli matrices, Clifford algebras and so on), we claim that there are self-adjoint  unitary matrices  $X_j,Y_k$ (of size $2^n$) such that $X_jY_k = -Y_kX_j$ for all $j,k$ and a (vector) state  $F$   such that $F(X_jY_k) = i\langle x_j,y_k\rangle\in i{\bb R}$.
 Indeed, let $H={\bb R}^n, \hat H={\bb C}^n$ and let ${\cl F}={\bb C}\oplus \hat H\oplus {\hat H}^{\wedge 2}\oplus\cdots$ denote the   ($2^n$-dimensional) antisymmetric Fock space associated
 to $\hat H$ with vacuum vector $\Omega $ ($\Omega\in {\cl F}$ is the unit in ${\bb C}\subset {\cl F}$). 
  For any $x,y\in H$, let $c(x),c(y)\in B( {\cl F} )$   be the creation operators defined
  by $c(x) t=x\wedge t$.
  Let $Y=c(y)+c(y)^*$ and $X=(c(x)-c(x)^*)/i$. Then $X,Y$ anticommute and $\langle XY\Omega, \Omega\rangle=i\langle x, y\rangle$. So applying this to $x_j,y_k$ yields the claim. Let $Q = \left(\begin{smallmatrix} 0&1\\ 1&0\end{smallmatrix}\right)$, ${  P} = \left(\begin{smallmatrix} 0&i\\ -i&0\end{smallmatrix}\right)$. Note $QP = -PQ$ and $(QP)_{11} = -i$. Therefore if we set $u_j = X_j\otimes Q$, $v_k=Y_k\otimes P$,
 and $f = F\otimes e_{11}$,
 we find self-adjoint unitaries such that   $u_jv_k =v_ku_j$ for all $j,k$ and $f(u_jv_k) = \langle x_j,y_k\rangle\in {\bb R}$. Thus we obtain by \eqref{10.4}
\[
 \left|\sum a_{jk}\langle x_j,y_k\rangle\right| = \left|f\left(\sum a_{jk}u_jv_k\right)\right| \le \|a\|_{\min},
\]
and hence $\|a\|_{H'} \le \|a\|_{\min}$. This proves    
$\|a\|_{H'} =  \|a\|_{\min}$ but also
  $\|\sum a_{jk}e_j\otimes e_k\|_{H'}\le \eqref{tsi2}$. Since, by \eqref{10.4},  we have $\eqref{tsi2}\le \|a\|_{\min}$, the proof is complete.
\end{proof}

The preceding equality $\|a\|_{\max} = \|a\|_{\min}$ seems open  when $a_{jk}\in {\bb C}$  (i.e.\ even the case $n=1$ in Proposition \ref{simK} (ii) is open). However, we have
\begin{pro}\label{protsi3}
Assume $n=1$ and $a_{jk}\in {\bb C}$ for all $j,k\ge 0$ then $\|a\|_{\max} \le K^{\bb C}_G\|a\|_{\min}$.
\end{pro}

\begin{proof}
 By \eqref{comp} we have
\[
 \sup\left\{\left|\sum a_{jk}s_jt_k\right| \ \bigg| \ s_j,t_k\in {\bb C}\quad |s_j| = |t_k| = 1\right\}=\|a\|_{\vee} \le \|a\|_{\min}.
\]
By Theorem \ref{thm1.2bis} we have for any unit vectors $x,y$ in $H$
\begin{align*}
 \left|\sum a_{jk} \langle u_jv_kx,y\rangle\right| &= \left|\sum a_{jk} \langle v_kx,u^*_jy\rangle\right|\le K^{\bb C}_G\|a\|_{\vee}
\le K^{\bb C}_G\|a\|_{\min}.
\end{align*}
Actually this holds even without the assumption that the $\{u_j\}$'s commute with the $\{v_k\}$'s.
\end{proof}
\begin{rmk} Although we are not aware of a proof, we believe that the  equalities  $\|a\|_{H'}=\|a\|_{\min}$
and  $\|a\|_{H'}=\|a\|_{\max}$  in Theorem \ref{ts1}   do not extend to the complex case. However,
if $[a_{ij}]$ is a $2\times 2$  matrix,  they  do  extend because, by \cite{Dav,Ton}  for $2\times 2$  matrices    in the complex case \eqref{eq1.2bis}  happens to be valid with $K=1$
 (while in the real case, for $2\times 2$  matrices, the best constant is $\sqrt{2}$).
\end{rmk}
%\begin{rem} Let us restrict to a $2\times 2$ matrix $[a_{jk}]$ 
%$(0\le j,k \le 1)$ with $a_{jk}\in {\bb C}$. In that case, we   can reduce all considerations to
%$C^*(F_N)$ with $N=1$, i.e. the commutative $C^*$-algebra
%$C^*({\bb Z})$. The $2$-dimensional $\ell_1$-space spanned (over $\bb C$)  by $(U_0,U_1)$, 
% is then equipped
%with the operator space structure induced by the latter commutative (and hence nuclear)  $C^*$-algebra. This implies $\|a\|_{\min}=\|a\|_{\vee}$ (and 
%$\|a\|_{\min}=\|a\|_{\max}$). Therefore, the best constant $K$ in the inequality $\|a\|_{H'}\le K\|a\|_{\min}$ is equal to the constant $K_G^{\bb C}(2)>1$ introduced in    \eqref{eq2+D}. In particular,  the equality $\|a\|_{H'}=\|a\|_{\min}$ in Theorem \ref{ts1} does not extend to the complex case.\end{rem}
See \cite{Rad1,Rad2,DI,JuPa,Fr,CD} for related contributions.

\section{Operator spaces, c.b. maps, Minimal tensor product}\label{sec11}

The theory of operator spaces is rather recent. It is customary to date its birth
with the 1987 thesis of Z.J. Ruan. This started a broad series of foundational investigations
mainly by Effros-Ruan and Blecher-Paulsen. See \cite{ER,P4}. We will
try to merely summarize the developments that relate to our main theme, meaning GT.

We start by recalling a few basic facts. First, a general unital $C^*$-algebra can be viewed as the non-commutative analogue of the space $C(\Omega)$ of continuous functions on a compact set $\Omega$. Secondly,  any Banach space $B$ can be viewed isometrically as a closed subspace of $C(\Omega)$:\ just take for $\Omega$ the closed unit ball of $B^*$ and consider the isometric embedding taking $x\in B$ to the function $\omega\to \omega(x)$.

\begin{dfn}\label{dfn11.1}
An operator space $E$ is a closed subspace $E\subset A$ of a general (unital if we wish) $C^*$-algebra.
\end{dfn}

With the preceding two facts in mind, operator spaces appear naturally as ``non-commutative Banach spaces.'' But actually the novelty in Operator space theory in not so much in the ``spaces'' as it is in the morphisms. Indeed, if $E_1\subset A_1$, $E_2\subset A_2$ are operator spaces, the natural morphisms $u\colon \ E_1\to E_2$ are the ``completely bounded'' linear maps that are defined as follows. First note that if $A$ is a $C^*$-algebra, the algebra $M_n(A)$ of $n\times n$ matrices with entries in $A$ is itself a $C^*$-algebra and hence admits a specific (unique) $C^*$-algebra norm. The latter norm can be described a bit more concretely when one realizes (by Gelfand--Naimark) $A$ as a closed self-adjoint subalgebra of $B(H)$ with norm induced by that of $B(H)$. In that case, if $[a_{ij}]\in M_n(A)$ then the matrix $[a_{ij}]$ can be viewed as a single operator acting naturally on $H\oplus\cdots\oplus H$ ($n$ times) and its norm is precisely the $C^*$-norm of $M_n(A)$. In particular the latter norm is independent of the embedding (or ``realization'') $A\subset B(H)$.

As a consequence, if $E\subset A$ is any closed subspace, then the space $M_n(E)$ of $n\times n$ matrices with entries in $E$ inherits the norm induced by $M_n(A)$. Thus, we can associate to an operator space $E\subset A$, the sequence of Banach spaces $\{M_n(E)\mid n\ge 1\}$.

\begin{dfn}\label{dfn11.2}
 A linear map $u\colon \ E_1\to E_2$ is called completely bounded (c.b.\ in short) if
\begin{equation}\label{eq11.0}
 \|u\|_{\text{cb}} \overset{\sst \text{def}}{=} \sup_{n\ge 1} \|u_n\colon \ M_n(E_1)\to M_n(E_2)\| < \infty
\end{equation}
where for each $n\ge 1$, $u_n$ is defined by $u_n([a_{ij}]) = [u(a_{ij})]$. One denotes by $CB(E_1,E_2)$ the space of all such maps.
\end{dfn}

The following factorization theorem (proved by Wittstock, Haagerup and Paulsen independently in the early 1980's) is crucial for the theory. Its origin goes back to major works by Stinespring (1955) and Arveson (1969) on completely positive maps.

\begin{thm}\label{thm8.3}
Consider $u\colon \ E_1\to E_2$. Assume $E_1\subset B(H_1)$ and $E_2\subset B(H_2)$. Then $\|u\|_{\text{cb}} \le 1$ iff there is a Hilbert space ${\cl H}$, a representation $\pi\colon \ B(H_1)\to B({\cl H})$ and operators $V, W\colon \ H_2\to {\cl H}$ with $\|V\| \|W\|\le 1$ such that 
\begin{equation}
 u(x) = V^*\pi(x)W.\tag*{$\forall x\in E_1$}
\end{equation}
\end{thm}

We refer the reader to \cite{ER,P4,Pau} for more background. 

We say that $u$ is a complete isomorphism (resp.\ complete isometry) if $u$ is invertible and $u^{-1}$ is c.b.\ (resp.\ if $u_n$ is isometric for all $n\ge 1$). We say that $u\colon \ E_1\to E_2$ is a completely isomorphic embedding if $u$ induces a complete isomorphism between $E_1$ and $u(E_1)$.\ms

The following non-commutative analogue of the Banach--Mazur distance has proved quite useful, especially to compare finite dimensional operator spaces. When $E,F$ are completely isomorphic we set
\begin{equation}\label{eqdist}
 d_{\text{cb}}(E,F) = \inf\{\|u\|_{\text{cb}} \|u^{-1}\|_{\text{cb}}\}
\end{equation}
where the infimum runs over all possible isomorphisms $u\colon \ E\to F$. We also set $d_{\text{cb}}(E,F) = \infty$ if $E,F$ are not completely isomorphic.

\n {\bf Fundamental examples:}\ Let us denote by $\{e_{ij}\}$ the matrix units in $B(\ell_2)$ (or in $M_n$). Let 
\begin{align*}
C = \ovl{\rm span}[e_{i1}\mid i\ge 1]\  \text{(``column space'')} \quad \text{and}\quad
R = \ovl{\rm span}[e_{1j}\mid j\ge 1]\ \text{(``row space'')}.
\end{align*}
We also define the $n$-dimensional analogues:
\[
 C_n = \text{span}[e_{i1}\mid 1\le i\le n]\quad \text{and}\quad R_n = \text{span}[e_{1j}\mid 1\le j\le n].
\]
Then $C\subset B(\ell_2)$ and $R\subset B(\ell_2)$ are very simple but fundamental examples of operator spaces. Note that $R\simeq \ell_2\simeq C$ as Banach spaces but $R$ and $C$ are not completely isomorphic. Actually they are ``extremely'' non-completely isomorphic:\ one can even show that
\[
 n = d_{\text{cb}}(C_n,R_n) = \sup\{d_{\text{cb}}(E,F)\mid \dim(E) = \dim(F)=n\}.
\]
\begin{rem} It can be shown that the map on $M_n$ that takes a matrix
to its transposed has cb-norm $=n$ (although it is isometric). 
In the opposite direction, the norm of a Schur multiplier
on $B(\ell_2)$ is equal to its cb-norm. As noticed early on by Haagerup
(unpublished), this follows easily from Proposition \ref{pro2.7ter}. \end{rem}
When working with an operator space $E$ we rarely use a specific embedding $E\subset A$, however, we crucially use the spaces $CB(E,F)$ for different operator spaces $F$. By \eqref{eq11.0} the latter are entirely determined by the knowledge of the sequence of normed (Banach) spaces
\begin{equation}\label{eq11.2}
 \{M_n(E)\mid n\ge 1\}.
\end{equation}
This is analogous to the fact that knowledge of a normed (or Banach space before completion) boils down to that of a vector space equipped with a norm. In other words, we view the sequence of norms in \eqref{eq11.2} as the analogue of the norm. Note however, that not any sequence of norms on $M_n(E)$ ``comes from'' a ``realization"   of $E$ as operator space (i.e.\ an embedding $E\subset A$). The sequences of norms that do so have been characterized by  Ruan: 

\begin{thm}[Ruan's Theorem]\label{thm11.3}
 Let $E$ be a vector space.
Consider for each $n$ a   norm $\alpha_n$ on the vector space $M_n(E)$. Then the sequence of norms $(\alpha_n)$ comes from a linear embedding
of $E$ into a $C^*$-algebra iff the following two properties hold:
\begin{itemize} 
\item $\forall n, \forall x\in M_n(E)\qquad \forall a,b\in M_n$
\qquad
$\alpha_n(a.x.b)\le \|a\|_{M_n} \alpha_n( x ) 
\ \|b\|_{M_n}.$
\item    $\forall n,m$
 $\forall x\in M_n(E)\ \forall y\in M_m(E)\qquad
\alpha_{n+m}(x\oplus y) =\max\{\alpha_n( x )
,\alpha_m(y)\}, $
 where we denote by  $x\oplus y$ the $(n+m) \times
(n+m)$ matrix defined by  
$ x\oplus y =\begin{pmatrix} x&0\\
0&y\end{pmatrix}.$
\end{itemize} 
\end{thm}

Using this, several important constructions involving operator spaces can be achieved, although they do not make sense a priori when one considers concrete embeddings $E\subset A$. Here is a list of the main such constructions.

\begin{thm}\label{thm8.4}
 Let $E$ be an operator  space.
\begin{itemize}
 \item[\rm (i)] Then there is a $C^*$-algebra $B$ and an isometric embedding $E^*\subset B$ such that for any $n\ge 1$
\end{itemize}
\begin{equation}\label{eq11.3}
M_n(E^*) \simeq CB(E,M_n)\quad \text{isometrically.}
\end{equation}
\end{thm}

\begin{itemize}
 \item[(ii)] Let $F\subset E$ be a closed subspace. There is a $C^*$-algebra ${\cl B}$ and an isometric embedding $E/F\subset {\cl B}$ such that for all $n\ge 1$
\end{itemize}
\begin{equation}\label{eq11.4}
 M_n(E/F) = M_n(E)/M_n(F).
\end{equation}
\begin{itemize}
 \item[(iii)] Let $(E_0,E_1)$ be a pair of operator spaces, assumed compatible for the purpose of interpolation (cf.\ \cite{BeL,P4}). Then for each $0<\theta<1$ there is a $C^*$-algebra $B_\theta$ and an isometric embedding $(E_0,E_1)_\theta\subset B_\theta$ such that for all $n\ge 1$
\end{itemize}
\begin{equation}\label{eq11.5}
M_n((E_0,E_1)_\theta) = (M_n(E_0), M_n(E_1))_\theta.
\end{equation}
 
  In each of the identities \eqref{eq11.3}, \eqref{eq11.4}, \eqref{eq11.5} the right-hand side makes natural sense, the content of the theorem is that in each case the sequence of norms on the right hand side ``comes from'' a concrete operator space structure on $E^*$ in (i), on $E/F$ in (ii) and on $(E_0,E_1)_\theta$ in (iii). The proof reduces to the verification that Ruan's criterion applies in each case.
 
 The general properties of c.b. maps, subspaces, quotients and duals mimic the analogous properties for Banach spaces, e.g.\ if $u\in CB(E,F)$, $v\in CB(G,E)$ then $uv\in CB(G,F)$ and \begin{equation}\label{eq11.4bis}\|uv\|_{\text{cb}} \le \|u\|_{\text{cb}} \|v\|_{\text{cb}};\end{equation} also $\|u^*\|_{\text{cb}} = \|u\|_{\text{cb}}$,  and if $F$ is a closed subspace of $E$,  $F^*=E^*/F^\bot$,   $(E/F)^*=F^\bot$ and $E\subset E^{**}$ completely isometrically.

It is natural to revise our terminology slightly:\ by an operator space structure (o.s.s.\ in short) on a vector (or Banach) space $E$ we mean the data of the sequence of the norms in \eqref{eq11.2}. We then ``identify'' the o.s.s.\ associated to two distinct embeddings, say $E\subset A$ and $E\subset B$, if they lead to identical norms in \eqref{eq11.2}. (More rigorously, an o.s.s.\ on $E$ is an equivalence class of embeddings as in Definition \ref{dfn11.1} with respect to the preceding equivalence.)

Thus, Theorem \ref{thm8.4} (i) allows us to introduce a duality for operator spaces:\ $E^*$ equipped with the o.s.s.\ defined in \eqref{eq11.3} is called the o.s.\ dual of $E$. Similarly $E/F$ and $(E_0,E_1)_\theta$ can now be viewed as operator spaces equipped with their respective o.s.s.\ \eqref{eq11.4} and \eqref{eq11.5}.

  The minimal tensor product of $C^*$-algebras induces naturally a tensor product for operator spaces:\ given $E_1\subset A_1$ and $E_2\subset A_2$, we have $E_1 \otimes E_2 \subset A_1 \otimes_{\min} A_2$ so the minimal $C^*$-norm induces a norm on $E_1\otimes E_2$ that we still denote by $\|\cdot\|_{\min}$ and we denote by $E_1\otimes_{\min} E_2$ the completion.
Thus $E_1\otimes_{\min} E_2 \subset A_1 \otimes_{\min} A_2$ is an operator space.

\begin{rem}\label{rem11.5}
It is useful to point out that the norm in $E_1\otimes_{\min}E_2$ can be obtained from the o.s.s.\ of $E_1$ and $E_2$ as follows. One observes that any $C^*$-algebra (and first of all $B(H)$) can be completely isometrically (but \emph{not} as subalgebra) embedded into a direct sum $\bigoplus\nolimits_{i\in I} M_{n(i)}$ of matrix algebras for a suitable family of integers $\{n(i)\mid i\in I\}$. Therefore using the o.s.s.\ of $E_1$ we may assume
\[
 E_1\subset \bigoplus\nolimits_{i\in I} M_{n(i)}
\]
and then we find
\[
 E_1\otimes_{\min} E_2 \subset \bigoplus\nolimits_{i\in I} M_{n(i)}(E_2).
\]
We also note  the canonical (flip) identification:
\begin{equation}\label{eq11.6}
 E_1\otimes_{\min} E_2\simeq E_2 \otimes_{\min}E_1.
\end{equation}
\end{rem}

\begin{rem}\label{rem11.6}
In particular, taking $E_2=E^*$ and $E_1=F$ we find
using \eqref{eq11.3} \[
 F\otimes_{\min} E^* \subset \bigoplus\nolimits_{i\in I} M_{n(i)}(E^*) = \bigoplus\nolimits_{i\in I} CB(E, M_{n(i)}) \subset CB\left(E, \bigoplus\nolimits_{i\in I} M_{n(i)}\right) 
\]
and hence (using \eqref{eq11.6})
\begin{equation}\label{eq11.7}
 E^* \otimes_{\min} F \simeq F \otimes_{\min} E^* \subset CB(E,F)\quad \text{isometrically.}
\end{equation}
More generally, by Ruan's theorem, the space $CB(E,F)$ can be given an o.s.s. by declaring
that $M_n(CB(E,F))=CB(E,M_n(F))$  isometrically. Then \eqref{eq11.7} becomes a completely isometric embedding.

As already mentioned, $\forall u\colon \ E\to F$, the adjoint $u^*\colon \ F^*\to E^*$ satisfies
\begin{equation}\label{eq11.6+}
 \|u^*\|_{\text{cb}} = \|u\|_{\text{cb}}.
\end{equation}
 Collecting these observations, we obtain:
\end{rem}

\begin{pro}\label{pro11.7}
Let $E,F$ be operator spaces and let $C>0$ be a constant. The following properties of a linear map $u\colon \ E\to F^*$ are equivalent.
\begin{itemize}
 \item[\rm (i)] $\|u\|_{\text{cb}} \le C$.
\item[\rm (ii)] For any integers $n,m$, the bilinear map $\Phi_{n,m}\colon \ M_n(E)\times M_m(F) \to M_n \otimes M_m \simeq M_{nm}$ defined by
\[
 \Phi_{n,m}([a_{ij}], [b_{k\ell}]) = [\langle u(a_{ij}), b_{k\ell}\rangle]
\]
satisfies $\|\Phi_{n,m}\| \le C$.\\
Explicitly, for   any finite sums 
$
\sum\nolimits_r a_r\otimes x_r \in M_n(E)$, and $ \sum\nolimits_s b_s\otimes y_s \in M_m(F)$
and for any $n\times m$ scalar matrices $\alpha,\beta$, we have
\end{itemize}
\begin{equation}\label{eq11.9}
 \left|\sum_{r,s} {\rm tr}(a_r\alpha {}^t b_s\beta^*) \langle ux_r,y_s\rangle\right| \le C\|\alpha\|_2 \|\beta\|_2 \left\|\sum a_r\otimes x_r\right\|_{M_n(E)} \left\|\sum b_s\otimes y_s\right\|_{M_m(F)}
\end{equation}
\begin{itemize}
 \item[] where $\|\cdot\|_2$ denotes the Hilbert--Schmidt norm.
 \item[\rm (iii)] For any pair of $C^*$-algebras $A,B$, the bilinear map
\[
 \Phi_{A,B}\colon\ A\otimes_{\min} E \times B \otimes_{\min} F \longrightarrow A\otimes_{\min} B
\]
defined by $\Phi_{A,B}(a\otimes e, b\otimes f) = a\otimes b \  \langle ue,f\rangle$ satisfies $\|\Phi_{A,B}\|\le C$.\\
Explicitly, whenever $\sum^n_1 a_i\otimes x_i\in A\otimes E$, $\sum^m_1 b_j\otimes y_j \in B\otimes F$, we have
\[
 \left\|\sum_{i,j} a_i\otimes b_j \langle u(x_i),y_j\rangle\right\|_{A\otimes_{\min}B} \le C\left\|\sum a_i \otimes x_i\right\|_{\min} \left\|\sum b_j\otimes y_j\right\|_{\min}.\]
\end{itemize}
\end{pro}

Let $H_1,H_2$ be Hilbert spaces. Clearly, the space $B(H_1,H_2)$ can be given a natural o.s.s.:\ one just considers a Hilbert space $H$ such that $H_1\subset H$ and $H_2\subset H$ 
(for instance $H=H_1\oplus H_2$) so that we have $H\simeq H_1\oplus K_1$, $H\simeq H_2 \oplus K_2$ and then we embed $B(H_1,H_2)$ into $B(H)$ via the mapping represented matricially by
\[
 x\mapsto \begin{pmatrix}
           0&0\\ x&0
          \end{pmatrix}.
\]
It is easy to check that the resulting o.s.s.\ does not depend on $H$ or on the choice of embeddings $H_1 \subset H$, $H_2\subset H$.

In particular, using this we obtain o.s.s.\ on $B({\bb C},H)$ and $B(H^*,{\bb C})$ for any Hilbert space $H$. We will denote these operator spaces as follows
\begin{equation}\label{eq.rc}
 H_c = B({\bb C},H)\qquad H_r = B(H^*,{\bb C}).
\end{equation}
The spaces $H_c$ and $H_r$ are isometric to $H$ as Banach spaces, but are quite different as o.s. When $H = \ell_2$ (resp.\ $H=\ell^n_2$) we recover the row and column spaces, i.e.\ we have 
\[
 (\ell_2)_c=C,\quad (\ell_2)_r = R,\quad (\ell^n_2)_c = C_n,\quad (\ell^n_2)_r=R_n.
\]

The next statement incorporates observations made early on
by the founders of operator space theory, namely Blecher-Paulsen and Effros-Ruan.
We refer to \cite {ER,P4} for more references.

\begin{thm}
Let $H,K$ be arbitrary Hilbert spaces. Consider a linear map $u\colon \ E\to F$.
\begin{itemize}
\item[\rm (i)] Assume either $E = H_c$ and $F=K_c$ or $E=H_r$ and $F=K_r$ then
\[
 CB(E,F) = B(E,F)\quad \text{and}\quad \|u\|_{\text{cb}} = \|u\|.
\]
\item[\rm (ii)] If $E=R_n$ (resp.\ $E=C_n$) and if $ue_{1j}=x_j$ (resp.\ $ue_{j1} = x_j$), then
\[
\|u\|_{\text{cb}} = \|(x_j)\|_C \quad (\text{resp. } \|u\|_{\text{cb}} = \|(x_j)\|_R).
\]
\item[\rm (iii)] Assume either $E=H_c$ and $F=K_r$ or $E=H_r$ and $F=K_c$. Then $u$ is c.b.\ iff it is  Hilbert--Schmidt and, denoting by 
$\|\cdot\|_2$   the Hilbert--Schmidt norm, we have 
\[
 \|u\|_{\text{cb}} = \|u\|_2.
\]

\item[\rm (iv)] Assume $F=H_c$. Then $\|u\|_{\text{cb}} \le 1$ iff for all finite sequences $(x_j)$ in $E$ we have 
\end{itemize}
\begin{equation}\label{eq1add}
\left(\sum\|ux_j\|^2\right)^{1/2} \le \|(x_j)\|_C.
\end{equation}
\begin{itemize}
\item[\rm (v)] Assume $F=K_r$. Then $\|u\|_{\text{cb}} \le 1$ iff for all finite sequences $(x_j)$ in $E$ we have
\[
 \left(\sum \|ux_j\|^2\right)^{1/2} \le \|(x_j)\|_R.
\]

\end{itemize}
\end{thm}

\begin{proof}[Proof of {\rm (i)} and {\rm (iv):}]
{\rm (i)} Assume say $E=F=H_c$. Assume $u\in B(E,F) = B(H)$. The mapping $x\to ux$ is identical to the mapping $B({\bb C},H)\to B({\bb C},H)$ of left multiplication by $u$. The latter is clearly cb with c.b.\ norm at most $\|u\|$. Therefore $\|u\|_{\text{cb}} \le \|u\|$, and the converse is trivial.\\
{\rm (iv)} Assume $\|u\|_{\text{cb}}\le 1$. Note that
\[
 \|(x_j)\|_C = \|v\colon \ R_n\to E\|_{\text{cb}}
\]
where $v$ is defined by $ve_{j1} = x_j$. (This follows from the identity $R^*_n \otimes_{\min} E = CB(R_n,E)$ and $R^*_n=C_n$.) Therefore, by \eqref{eq11.4bis} we have
\[
\|uv\colon \ R_n\to F\|_{\text{cb}} \le \|u\|_{cb} \|(x_j)\|_C
\]
and hence by (ii)
\[
 \|uv\|_2= \left(\sum\nolimits_j \|uve_{j1}\|^2\right)^{1/2} = \left(\sum\|x_j\|^2\right)^{1/2} \le \|u\|_{\text{cb}} \|(x_j)\|_C.
\]
This proves the only if part.\\
To prove the converse, assume $E\subset A$ ($A$ being a $C^*$-algebra). Note $\|(x_j)\|^2_C = \sup\{\sum f(x^*_jx_j) \mid f \ {\rm state}\  {\rm on} \ A\}$. Then by Proposition \ref{hb3}, \eqref{eq1add} implies that there is a state $f$ on $A$ such that
\begin{equation}\label{eq2add}
\forall x\in E\qquad \qquad \|ux\|^2 \le f(x^*x) = \|\pi_f(x)\xi_f\|^2
\end{equation}
where $\pi_f$ denotes the (so-called GNS) representation on $A$ associated to $f$ on a Hilbert space $H_f$ and $\xi_f$ is a cyclic unit vector in $H_f$. By \eqref{eq2add} there is an operator $b\colon \ H_f\to K$ with $\|b\|\le 1$ such that $ux = b\pi_f(x)\xi_f$. But then this is precisely the canonical factorization of c.b.\ maps described in Theorem \ref{thm8.3}, but here for the map $x\mapsto ux \in B({\bb C}, K)=K_c$. Thus we conclude $\|u\|_{\text{cb}}\le 1$.
\end{proof}

\begin{rem}\label{rk-rc} Assume $E\subset A$ ($A$ being a $C^*$-algebra). By Proposition \ref{hb3},  (iv) (resp. (v)) holds iff there is a state $f$ on $A$ such that
$$\forall x\in E \quad  \|ux\|^2\le f(x^*x)\quad {\rm ( resp.} \ \forall x\in E \quad \|ux\|^2\le f(xx^*) ).$$
\end{rem}

\begin{rem}\label{rk-ps} The operators $u\colon \ A\to \ell_2$
 ($A$ being a $C^*$-algebra) such that for some constant $c$ there is a state $f$ on $A$ such that
$$\forall x\in A \quad  \|ux\|^2\le c ( f(x^*x)f(xx^*) )^{1/2},$$
have been characterized in \cite{PS} as those that are completely bounded
from $A$ (or from an exact subspace $E\subset A$) to the operator Hilbert space $OH$.
See \cite{P44} for more on this.
\end{rem}

\begin{rem}\label{rk-max} As   Banach spaces,  we have
$c_0^*=\ell_1$, $\ell_1^*=\ell_\infty$  and  more generally 
$L_1^*=L_\infty$. Since the spaces $c_0$ and $L_\infty$ are
$C^*$-algebras, they admit a natural specific o.s.s. (determined by the unique $C^*$-norms on $M_n\otimes c_0$ and $M_n\otimes L_\infty$), therefore, by duality
we may  equip also $\ell_1$ (or $L_1\subset L^*_\infty$)  with a natural specific o.s.s.
called the ``maximal o.s.s.". In the case of $\ell_1$, it is easy to describe: Indeed, it is a simple exercise (recommended to the reader) to check
that the embedding   $\ell_1\simeq \overline{\rm span}\{U_j\}\subset {C}^*({\bb F}_\infty)$, already considered in \S \ref{sec10}
constitutes a completely isometric realization of this operator space $\ell_1=c_0^*$
(recall $(U_j)_{j\ge 1}$ denote the unitaries of ${C}^*({\bb F}_\infty)$ that correspond to the free generators) .
\end{rem}

\section{Haagerup tensor product}\label{sec12}

Surprisingly, operator spaces admit a special tensor product, the Haagerup tensor product, that does not really have any counterpart for Banach spaces. Based on unpublished work of Haagerup related to GT,  Effros and Kishimoto popularized it under this name in 1987. At first only its norm was considered, but somewhat later, its o.s.s.\ emerged as a crucial concept to understand completely positive and c.b.\ \emph{multilinear} maps, notably in fundamental work by Christensen and Sinclair, continued by many authors, namely Paulsen, Smith, Blecher, Effros and Ruan. See \cite{ER,P4} for precise references.

We now formally introduce the Haagerup tensor product $E\otimes_h F$ of two operator spaces.

Assume $E\subset A, F\subset B$. Consider a finite sum $t = \sum x_j\otimes y_j \in E\otimes F$. We define
\begin{equation}\label{eq12.1}
 \|t\|_h = \inf\{\|(x_j)\|_R \|(y_j)\|_C\}
\end{equation}
where the infimum runs over all possible representations of $t$. More generally, given a matrix $t = [t_{ij}]\in M_n(E\otimes F)$ we consider factorizations of the form
\[
 t_{ij} = \sum\nolimits^N_{k=1} x_{ik}\otimes y_{kj}
\]
with $x = [x_{ik}] \in M_{n,N}(E)$, $y = [y_{kj}]\in M_{N,n}(F)$ and we define
\begin{equation}\label{eq12.2}
\|t\|_{M_n(E\otimes_h F)} = \inf\{\|x\|_{M_{n,N}(E)} \|y\|_{M_{N,n}(F)}\}
\end{equation}
the inf being over all $N$ and all possible such factorizations.

It turns out that this defines an o.s.s.\ on $E\otimes_h F$, so there is a $C^*$-algebra $C$ and an embedding $E\otimes_h F\subset C$ that produces the same norms as in \eqref{eq12.2}. This can be deduced from Ruan's theorem, but one can also take $C=A*B$ (full free product of $C^*$-algebras) and use the ``concrete'' embedding
\[
 \sum x_j\otimes y_j\longrightarrow \sum x_j\cdot y_j\in A*B.
\]
Then this is completely isometric. See e.g.\ \cite[\S 5]{P4}.
\\ In particular, since we have an (automatically completely contractive) $*$-homomorphism  $A*B\to A\otimes_{\min}B$, it follows
\begin{equation}\label{eq12.2+}
\|t\|_{M_n(E\otimes_{\min} F)} \le \|t\|_{M_n(E\otimes_h F)} . 
\end{equation}

For any linear map $u\colon \ E\to F$ between operator spaces we denote by $\gamma_r(u)$ (resp.\ $\gamma_c(u)$) the constant of factorization of $u$ through a space of the form $H_r$ (resp.\ $K_c$), as defined in \eqref{eq.rc}. More precisely, we set
\begin{equation}\label{eq12.3}
 \gamma_{r} (u) = \inf\{\|u_1\|_{\text{cb}} \|u_2\|_{\text{cb}}\}\quad {\rm (resp.} \  \gamma_{c} (u) = \inf\{\|u_1\|_{\text{cb}} \|u_2\|_{\text{cb}}\})
\end{equation}
where the infimum runs over all possible Hilbert spaces $H,K$ and all factorizations
\[
 E \overset{\sst u_1}{\longrightarrow} Z \overset{\sst u_2}{\longrightarrow} F
\]
of $u$ through $Z$ with $Z = H_r$ (resp.\ $Z=K_c$). (See also \cite{OP} for a symmetrized version of the Haagerup tensor product, for maps factoring through $H_r\oplus K_c$).

\begin{thm}\label{thm12.1}
Let $E\subset A, F\subset B$ be operator spaces ($A,B$ being $C^*$-algebras). Consider a linear map $u\colon \ E\to F^*$ and let $\varphi\colon \ E\times F\to {\bb C}$ be the associated bilinear form. Then $\|\varphi\|_{(E\otimes_h F)^*}\le 1$ iff each of the following equivalent conditions holds:
\begin{itemize}
 \item[\rm (i)] For any finite  sets $(x_j), (y_j)$ in $E$ and $F$ respectively we have
\[
 \left|\sum \langle ux_j,y_j\rangle\right| \le \|(x_j)\|_R \|(y_j)\|_C.
\]
\item[\rm (ii)] There are states $f,g$ respectively on $A$ and $B$ such that
\end{itemize}
\begin{equation}
 |\langle ux,y\rangle| \le f(xx^*)^{1/2} g(y^*y)^{1/2}.\tag*{$\forall (x,y)\in E\times F$}
\end{equation}
\begin{itemize}
 \item[\rm (iii)] $\gamma_r(u)\le 1$.
\end{itemize}
Moreover if $u$ has finite rank, i.e.\ $\varphi\in E^*\otimes F^*$ then
\begin{equation}\label{eq12.1bis}
 \|\varphi\|_{(E\otimes_h F)^*} = \|\varphi\|_{E^*\otimes_h F^*}.
\end{equation}
\end{thm}
\begin{proof}
 (i) $\Leftrightarrow \|\varphi\|_{(E\otimes_hF)^*}\le 1$ is obvious. (i) $\Leftrightarrow$ (ii) follows by the Hahn--Banach type argument (see \S \ref{sechb}).
 (ii) $\Leftrightarrow$ (iii) follows from     Remark \ref{rk-rc}.
\end{proof}
In particular,  by Theorem \ref{thm1.1}, we deduce
\begin{cor} If $A,B$ are commutative $C^*$-algebras, then  any bounded linear $u:\ A \to B^*$ is c.b.\ 
with $\|u\|_{cb}\le K_G^{\bb C}\|u\|$.
\end{cor}
\begin{rem}\label{rem12.2}
 If we replace $A,B$ by the ``opposite'' $C^*$-algebras (i.e.\ we reverse the product), we obtain induced o.s.s.\ on $E$ and $F$ denoted by $E^{op}$ and $F^{op}$. Another more concrete way to look at this, assuming $E\subset B(\ell_2)$,  is that $E^{op}$  
 consists of the transposed matrices of those in $E$. It is easy to see that $R^{op}=C$ and $C^{op}=R$. Applying Theorem \ref{thm12.1} to $E^{op}\otimes_h F^{op}$, we obtain the equivalence of the following assertions:
\begin{itemize}
 \item[(i)] For any   finite sets $(x_j)$, $ (y_j)$ in $E$ and $F$
\[
 \left|\sum\langle ux_j,y_j\rangle\right|\le \|(x_j)\|_C \|(y_j)\|_R.
\]
 \item[(ii)] There are states $f,g$ on $A,B$ such that
\end{itemize}
\begin{equation}
 |\langle ux,y\rangle|\le f(x^*x)^{1/2} g(yy^*)^{1/2}\tag*{$\forall(x,y)\in E\times F$}
\end{equation}
\begin{itemize}
 \item[(iii)] $\gamma_c(u)\le 1$.
\end{itemize}
\end{rem}

\begin{rem}\label{rem12.3} Since we have trivially $\|u\|_{cb}\le \gamma_r(u) $ and 
$\|u\|_{cb}\le \gamma_c(u) $, the three equivalent properties in Theorem \ref{thm12.1}
 and also the three appearing in the preceding Remark \ref{rem12.2} each imply that $\|u\|_{cb}\le 1$.
\end{rem}
\begin{rem}\label{rem12.3+} The most striking property of the Haagerup tensor product is probably its  selfduality. There is no Banach space
analogue for this. By   ``selfduality" we mean that if either $E$ or $F$ is finite dimensional,
then  $(E\otimes_h F)^*=E^*\otimes_h F^*$ completely isometrically, i.e. these  coincide as operator spaces
and not only (as expressed by \eqref{eq12.1bis}) as Banach spaces.

\end{rem}
\begin{rem}\label{rem12.4} Returning to the notation in \eqref{eq0.5}, let  $E$ and $F$ be \emph{commutative}
$C^*$-algebras (e.g. $E=F=\ell_\infty^n$) with their natural o.s.s.. Then, a simple verification
from the definitions shows that for any $t\in E\otimes F$ we have on one hand $\|t\|_{\min}=\|t\|_\vee$ and 
on the other hand $\|t\|_h=\|t\|_H.$ Moreover, if $u\colon\ E^*\to F$ is the associated linear map,
we have $\|u\|_{cb}=\|u\|$ and     $\gamma_r(u)=\gamma_c(u)=\gamma_2(u).$ Using the selfduality
described in the preceding remark, we find that
for any $t\in L_1\otimes L_1$ (e.g. $t\in \ell^n_1\otimes \ell^n_1$) we have
$\|t\|_h=\|t\|_{H'}.$ 
\end{rem}

\section{The operator Hilbert space OH}\label{ssec1.8}

We call an operator space  Hilbertian if the underlying Banach space 
is isometric to a Hilbert space. For instance $R,C$
are Hilbertian,  but there are also many more 
examples in $C^*$-algebra related theoretical physics, {\it e.g.}
Fermionic or Clifford generators
 of the so-called CAR algebra (CAR stands for
 canonical anticommutation relations), generators of the Cuntz algebras,  
free semi-circular or circular systems in Voiculescu's free probability theory, $\ldots$. None of them 
however is self-dual. Thus, 
the next result came 
somewhat as a surprise.  (Notation:\ if $E$ is an operator space, say $E \subset B(H)$, then
$\ovl  E$ is the complex conjugate of $E$ equipped with the o.s.\ structure 
corresponding to the embedding $\ovl E \subset \ovl{B(H)}= B(\ovl H)$.)

\begin{thm}[\cite{Poh}]
 Let $H$ be an arbitrary Hilbert space. There exists, for a suitable 
${\cl H}$, a Hilbertian operator space $E_H \subset B({\cl H})$ isometric to $H$ such that the 
canonical identification (derived from the scalar product) $E^*_H \to \ovl E_H$ 
is completely isometric. Moreover, the space $E_H$ is unique up to complete 
isometry.
Let $(T_i)_{i\in I}$ be an orthonormal basis in $E_H$. Then, 
for any $n$ and  any finitely
supported family $(a_i)_{i\in I}$ in $M_n$, we have
$\| \sum a_i\otimes T_i\|_{M_n(E_H)} =\| \sum a_i \otimes {\ovl a_i} \|^{1/2}_{M_{n^2}}.$
\end{thm}

When $H=\ell_2$, we denote the space $E_H$ by $OH$ and we call it the ``operator 
Hilbert space''. Similarly, we denote it by $OH_n$ when $H=\ell^n_2$ and by 
$OH(I)$ when $H = \ell_2(I)$.  

The norm of factorization through $OH$, denoted by $\gamma_{oh}$, 
can be studied in analogy with  \eqref{eq0.2}, and the dual norm
is identified in \cite{Poh} in analogy with Corollary \ref{hb3-}.
 
 Although this is clearly  the ``right" operator space analogue
of Hilbert space, we will see shortly that, in the context of
GT, it is upstaged by the space $R \oplus C$.

  The space $OH$ has rather striking 
complex interpolation properties. For instance, we have a
completely isometric identity   $(R,C)_{1/2} \simeq OH$, where  
the pair $(R,C)$ is viewed as ``compatible'' using the transposition map $x\to 
{}^tx$ from $R$ to $C$ which allows to view both $R$ and $C$ as continuously 
injected into a single space (namely here $C$) for the complex  interpolation
method to make sense. \\ Concerning the Haagerup tensor product, for any sets 
$I$ and $J$, we have a completely isometric identity
$$OH(I) \otimes_h OH(J)\simeq OH(I\times J).$$
Finally, we should mention that $OH$ is ``homogeneous'' (an o.s.\ $E$ is called 
homogeneous if any linear map $u\colon \ E\to E$ satisfies $\|u\| = 
\|u\|_{cb}$). While $OH$ is unique, the class of homogeneous Hilbertian operator 
spaces (which also includes $R$ and $C$) is very rich and provides a very fruitful source of examples.
  See \cite{Poh} for more on all this.

\begin{rmk}
A classical application of Gaussian variables (resp.\ of the Khintchine inequality) in $L_p$ is the existence of an isometric (resp.\ isomorphic) embedding of $\ell_2$ into $L_p$ for $0<p\ne 2<\infty$ (of course $p=2$ is trivial). Thus it was natural to ask whether an analogous embedding (completely isomorphic this time) holds for the operator space $OH$. The case $2<p<\infty$ was ruled out early on by Junge who observed that this would contradict \eqref{eq5.9}. The crucial case $p=1$ was solved by Junge in \cite{Ju} (see \cite{P44} for a simpler proof). The case $1<p\le 2$ is included in Xu's paper \cite{X3}, as part of more general embeddings. We refer the reader to Junge and Parcet's \cite{JPa3} for a discussion of completely isomorphic embeddings of $L_q$ into $L_p$ (non-commutative) for $1\le p<q<2$.
\end{rmk}

\section{GT and random matrices}\label{sec13}

The notion of ``exact operator space'' originates in Kirchberg's work on exactness for $C^*$-algebras. To abbreviate, it is convenient to introduce the exactness constant $\text{ex}(E)$ of an operator space $E$ and to define ``exact'' operator spaces as those $E$ such that $\text{ex}(E)<\infty$. We define first when $\dim(E)<\infty$
\begin{equation}
d_{SK}(E) = \inf\{d_{\text{cb}}(E,F)\mid n\ge 1, F\subset M_n\}.
\end{equation}
By an easy perturbation argument, we have
\begin{equation}\label{eq13.1'}
d_{SK}(E) = \inf\{d_{\text{cb}}(E,F)\mid F\subset K(\ell_2)\},
\end{equation}
and this explains our notation. We then set
\begin{equation}\label{eq13.2}
 \text{ex}(E) = \sup\{d_{SK}(E_1)\mid E_1\subset E \quad \dim (E_1)<\infty\}
\end{equation}
and we call ``exact'' the operator spaces $E$ such that $\text{ex}(E)<\infty$.

Note that if $\dim(E)<\infty$ we have obviously $\text{ex}(E) = d_{SK}(E)$. Intuitively, the meaning of the exactness of $E$ is a form of embeddability of $E$ into the algebra ${\cl K} = K(\ell_2)$ of all compact operators on $\ell_2$. But instead of a ``global'' embedding of $E$, it is a ``local'' form of embeddability that is relevant here:\ we only ask for a uniform embeddability of all the finite dimensional subspaces. This local embeddability is of course weaker than the global one. 
By Theorem
\ref{cek} and a simple perturbation argument, any nuclear $C^*$-algebra $A$ is exact and $\text{ex}(A)=1$. A fortiori if $A$ is nuclear and $E\subset A$ then $E$ is exact. Actually, if $E$ is itself a $C^*$-algebra and $\text{ex}(E)<\infty$ then necessarily $\text{ex}(E) =1$. Note however that a $C^*$-subalgebra of a nuclear $C^*$-algebra \emph{need not} be nuclear itself. We refer the reader to \cite{BO}
for an extensive treatment of exact $C^*$-algebras.\\
A typical non-exact operator space (resp. $C^*$-algebra)
is the space $\ell_1$ with its maximal o.s.s. described in
 Remark \ref{rk-max}, (resp. the full $C^*$-algebra of any non-Abelian free group). More precisely,
 it is known (see \cite[p. 336]{P4}) that $\text{ex}(\ell_1^n) = d_{SK}(\ell_1^n)\ge n/(2\sqrt{n-1})$.

In the Banach space setting, there are conditions on Banach spaces $E\subset C(S_1)$, $F\subset C(S_2)$ such that any bounded bilinear form $\varphi\colon \ E\times F\to {\bb R}$ or ${\bb C}$ satisfies the same factorization as in GT. But the conditions are rather strict (see Remark \ref{gtext}). Therefore the next result came as a surprise, because it seemed to have no  Banach space analogue.

\begin{thm}[\cite{JP}]\label{thm13.1}
Let $E,F$ be exact operator spaces. Let $u\colon \ E\to F^*$ be a c.b.\ map. Let $C = \text{ex}(E) \text{ ex}(F)\|u\|_{\text{cb}}$. Then for any $n$ and any finite sets $(x_1,\ldots, x_n)$ in $E$, $(y_1,\ldots, y_n)$ in $F$ we have
\begin{align}\label{eq13.3}
\left|\sum \langle ux_j,y_j\rangle\right| &\le C(\|(x_j)\|_R + \|(x_j)\|_C) (\|(y_j)\|_R + \|(y_j)\|_C),
\intertext{and hence a fortiori}
\sum |\langle ux_j,y_j\rangle| &\le 2C(\|(x_j)\|^2_R + \|(x_j)\|^2_C)^{1/2} (\|(y_j)\|^2_R + \|(y_j)\|^2_C)^{1/2}.\notag
\end{align}
Assume $E\subset A, F\subset B$ ($A,B$ $C^*$-algebra). Then there are states $f_1,f_2$ on $A,g_1,g_2$ on $B$ such that
\begin{equation}\label{eq13.3+} \forall(x,y)\in E\times F\quad
|\langle ux,y\rangle| \le 2C  (f_1(x^*x) + f_2(xx^*))^{1/2} (g_1(yy^*) + g_2(y^*y))^{1/2}.  
\end{equation}
Moreover there is a linear map $\tilde u\colon \ A\to B^*$ with $\|\tilde u \|\le 4C$ satisfying
\begin{equation}\label{extb}
\forall (x,y)\in E\times F \quad  \langle ux,y\rangle=\langle \tilde u x,y\rangle.\qquad\qquad\qquad\qquad\qquad\qquad\qquad\qquad
\end{equation}
\end{thm}

Although the original proof of this requires a much ``softer'' ingredient, it is easier to describe the argument using the following more recent (and much deeper) result, involving random matrices. This brings us back to Gaussian random variables. We will denote by $Y^{(N)}$ an $N\times N$ random matrix, the entries of which, denoted by $Y^{(N)}(i,j)$ are independent  complex Gaussian variables with mean zero and variance $1/N$ so that ${\bb E}|Y^N(i,j)|^2 = 1/N$ for all $1\le i,j\le N$. It is known (due to Geman, see \cite{HT3}) that
\begin{equation}\label{eq13.4}
 \lim_{N\to\infty} \|Y^{(N)}\|_{M_N} = 2 \quad \text{a.s.}
\end{equation}
Let $(Y^{(N)}_1, Y^{(N)}_2,\ldots)$ be a sequence of independent copies of $Y^{(N)}$, so that the family $\{Y^N_k(i,j)\mid k\ge 1, 1\le i,j\le N\}$ is an independent family of $N(0,N^{-1})$ complex Gaussian.

%We may rewrite $  Y^{(N)}$ as $ \sum Y^{(N)}(i,j) e_{ij} $
In \cite{HT3},  a considerable strengthening of \eqref{eq13.4} is proved: for any
finite sequence  $(x_1,\ldots, x_n)$ in $M_k$ with $n,k$ fixed, the norm
of $\sum_{j=1}^n Y^{(N)}_j \otimes x_j $ converges almost surely to a limit that
the authors identify using free probability.
The next result, easier to state, is formally weaker than their result (note that it implies
in particular  $\limsup_{N\to\infty} \|Y^{(N)}\|_{M_N} \le 2$).

\begin{thm}[\cite{HT3}]\label{thm13.2}
Let $E$ be an exact operator space. Then for any $n$ and any $(x_1,\ldots, x_n)$ in $E$ we have almost surely
\begin{equation}\label{eq13.5}
\limsup_{N\to\infty} \left\|\sum Y^{(N)}_j \otimes x_j\right\|_{M_N(E)} \le \text{ex}(E)(\|(x_j)\|_R + \|(x_j)\|_C )\le 2 \text{ ex}(E)\|(x_j)\|_{RC}.
\end{equation}
\end{thm}

\begin{rmk}
The closest to \eqref{eq13.5} that comes to mind in the Banach space setting is the ``type 2'' property. But the surprise is that $C^*$-algebras (e.g.\ commutative ones) can be exact (and hence satisfy \eqref{eq13.5}), while they never satisfy ``type 2'' in the Banach space sense.
\end{rmk}

\begin{proof}[Proof of Theorem \ref{thm13.1}]
Let $C_1 = \text{ex}(E) \text{ ex}(F)$, so that $C = C_1\|u\|_{\text{cb}}$. We identify $M_N(E)$ with $M_N \otimes E$. By \eqref{eq11.9} applied with $\alpha=\beta= N^{-1/2}I$, we have a.s.\ 
\[
\underset{\sst N\to\infty}{\ovl{\rm lim}} \left| N^{-1} \sum\nolimits_{ij} \text{ tr}(Y^{(N)}_i Y^{(N)*}_j) \langle ux_i,y_j\rangle\right| \le \|u\|_{\text{cb}} \underset{\sst N\to\infty}{\ovl{\rm lim}} \left\|\sum Y^{(N)}_i \otimes x_i\right\| \underset{\sst N\to\infty}{\ovl{\rm lim}} \left\|\sum Y^{(N)}_j \otimes y_j\right\|,
\]
and hence by Theorem \ref{thm13.2}, we have
\begin{equation}\label{eq13.6}
 \underset{\sst N\to\infty}{\ovl{\rm lim}} \left|\sum\nolimits_{ij} N^{-1} \text{ tr}(Y^{(N)}_i Y^{(N)*}_j) \langle ux_i,y_j\rangle\right| \le C_1\|u\|_{\text{cb}} (\|(x_i)\|_R + \|(x_i)\|_C) (\|(y_j)\|_R + \|(y_j)\|_C).
\end{equation}
But since
\[
 N^{-1} \text{ tr}(Y^{(N)}_iY^{(N)*}_j) = N^{-1} \sum^N_{r,s=1} Y^{(N)}_i(r,s) \ovl{Y^{(N)}_j(r,s)}
\]
and we may rewrite if we wish
\[
Y^{(N)}_j(r,s) = N^{-\frac12} g^{(N)}_j(r,s)
\]
where $\{g^{(N)}_j(r,s)\mid j\ge 1, 1\le r,s\le N, N\ge 1\}$ is an independent family of complex Gaussian variables each distributed like $g^{\bb C}$, we have by the strong law of large numbers
\begin{align*}
\lim_{N\to\infty} N^{-1} \text{ tr}(Y^{(N)}_iY^{(N)*}_j) &= \lim_{N\to\infty} N^{-2} \sum^N_{r,s=1} g^{(N)}_i(r,s) \ovl{g^{(N)}_j(r,s)}\\
&= {\bb E}(g^{\bb C}_i \ovl{g^{\bb C}_j})\\
&= \begin{cases}
    1&\text{if $i=j$}\\
0&\text{if $i\ne j$}
   \end{cases}~~.
\end{align*}
Therefore, \eqref{eq13.6} yields
\[
\left|\sum \langle ux_j,y_j\rangle\right| \le C_1\|u\|_{\text{cb}} (\|(x_j)\|_R + \|(y_j)\|_C) (\|(y_j)\|_R + \|(y_j)\|_C)\qquad \qed
\]
\renewcommand{\qed}{}\end{proof}

\begin{rem}\label{rem13.3}
In the preceding proof we only used \eqref{eq11.9} with $\alpha,\beta$ equal to a multiple of the identity, i.e.\ we only used the inequality
\begin{equation}\label{eq13.7}
 \left|\sum\nolimits_{ij} \frac1N \text{ tr}(a_ib_j) \langle ux_i,y_j\rangle\right| \le c\left\|\sum a_i\otimes x_i\right\|_{M_N(E)} \left\|\sum b_j\otimes y_j\right\|_{M_N(F)}
\end{equation}
valid with $c=\|u\|_{\text{cb}}$ when $N\ge 1$ and $\sum a_i\otimes x_i \in M_N(E)$, $\sum b_j\otimes y_j \in M_N(F)$ are arbitrary.
\n
Note that this can be rewritten as saying that for any $x=[x_{ij}]\in M_N(E)$ and any $y=[y_{ij}]\in M_N(F)$ we have
\begin{equation}\label{eq13.7+}
\left|\sum\nolimits_{ij} \frac1N   \langle ux_{ij},y_{ji}\rangle\right| \le c \|x\|_{M_N(E)} \|y\|_{M_N(F)}
\end{equation}
\end{rem}

It turns out that \eqref{eq13.7+} or equivalently  \eqref{eq13.7} is actually a weaker property studied in \cite{B2} and \cite{Ito} under the name of ``tracial boundedness.'' More precisely, in \cite{B2} $u$ is called tracially bounded (in short t.b.) if there is a constant $c$ such that \eqref{eq13.7} holds and the smallest such $c$ is denoted by $\|u\|_{\text{tb}}$. The preceding proof shows that \eqref{eq13.3} holds with $\|u\|_{\text{tb}}$ in place of $\|u\|_{\text{cb}}$. This formulation of Theorem \ref{thm13.1} is optimal. Indeed, the converse also holds:
\begin{pro}   Assume $E\subset A, F\subset B$ as before. Let $u\colon \ E\to F^*$ be a linear map. Let $C(u)$ be the best possible
constant $C$ such that \eqref{eq13.3} holds for all finite sequences $(x_j),(y_j)$.
Let $\tilde C(u)=\inf\{ \|\tilde u\|\}$ where the infimum runs over all   $\tilde u\colon \ A\to B^*$ satisfying   \eqref{extb}.
Then 
$$4^{-1} \tilde C(u)\le C(u)\le \tilde C(u).$$
Moreover
$$4^{-1} \|u\|_{\text{tb}} \le  C(u)\le  ex(E) ex(F)\|u\|_{\text{tb}} .$$
In particular, for a linear map $u\colon\ A\to B^*$ tracial boundedness is equivalent to ordinary boundedness,  and in that case
$$4^{-1} \|u\|  \le4^{-1} \|u\|_{\text{tb}} \le  C(u)\le   \|u\| .$$
\end{pro}

\begin{proof} 
That $C(u)\le \tilde C(u)$ is a consequence of \eqref{eq3.2}.
That $4^{-1} \tilde C(u)\le C(u)$ follows from a simple variant of the argument for (iii) in Proposition \ref{hb2} (the factor 4 comes from the use of \eqref{eq13.3+}).
The preceding remarks show that $C(u)\le  ex(E) ex(F)\|u\|_{\text{tb}} .$\\
It remains to show $4^{-1} \|u\|_{\text{tb}} \le  C(u)$. By \eqref{eq13.3} we have
$$\left|\sum\nolimits_{ij} \langle ux_{ij} ,y_{ji} \rangle\right|  \le C(u)(\|(x_{ij} )\|_R + \|(x_{ij} )\|_C) (\|(y_{ji}) \|_R + \|(y_{ji})\|_C).$$
Note that for any fixed $i$ the sum $L_i=\sum_j e_{ij} \otimes x_{ij} $ satisfies
$\|L_i\|=\|\sum_j x_{ij}x^*_{ij}\|^{1/2}\le \|x\|_{M_N(E)}$, 
and hence  $\|\sum_{ij} x_{ij}x^*_{ij}\|^{1/2}\le N^{1/2}\|x\|_{M_N(E)}$.
Similarly  $\|\sum_{ij} x^*_{ij}x_{ij}\|^{1/2}\le N^{1/2}\|x\|_{M_N(E)}$. This implies
$$(\|(x_{ij} )\|_R + \|(x_{ij} )\|_C) (\|(y_{ji}) \|_R + \|(y_{ji})\|_C)\le 4N \|x\|_{M_N(E)} \|y\|_{M_N(F)}$$
and hence we obtain \eqref{eq13.7+}
with $c\le 4C(u)$, which means that $4^{-1} \|u\|_{\text{tb}} \le  C(u)$.
  \end{proof}

\begin{rem}\label{lem13.4}
Consider  Hilbert spaces $H$ and $K$. We denote by  $H_r,K_r$ (resp.\ $H_c,K_c$) the associated row  (resp.\ column) operator spaces. Then for any linear map $u\colon \ H_r\to K^*_r$, $u\colon \ H_r\to K^*_c$, $u\colon \ H_c\to K^*_r$ or $u\colon \ H_c\to K^*_c$ we have
\[
 \|u\|_{\text{tb}} = \|u\|.
\]
Indeed this can be checked by a simple modification of the preceding Cauchy-Schwarz argument.
\end{rem}

The next two statements follow from results  known to Steen Thorbj{\o}rnsen since at least 1999 (private communication). We present our own self-contained derivation of this, for
  use only  in \S \ref{sec16} below.  
  
\begin{thm}\label{thm13.10}
Consider independent copies $Y'_i = Y^{(N)}_i(\omega')$ and
$
 Y''_j = Y^{(N)}_j(\omega'')$    {for} $ (\omega',\omega')\in\Omega\times \Omega$.
 Then, for any $n^2$-tuple of scalars $(\alpha_{ij})$, we have
\begin{equation}\label{eq13.16bis}
 \underset{N\to\infty}{\ovl{\lim}}\left\|\sum \alpha_{ij}Y^{(N)}_i(\omega') \otimes Y^{(N)}_j(\omega'')\right\|_{M_{N^2}} \le 4(\sum|\alpha_{ij}|^2)^{1/2}
\end{equation}
for a.e.\ $(\omega',\omega'')$ in $\Omega\times \Omega$.\end{thm}

\begin{proof}
By (well known) concentration of measure arguments, it is known that \eqref{eq13.4} is essentially the same as the assertion that $\lim_{N\to\infty} {\bb E}\|Y^{(N)}\|_{M_N}=2$. Let $\vp(N)$ be defined by
\[
 {\bb E}\|Y^{(N)}\|_{M_N} = 2+\vp(N)
\]
so that we know $\vp(N)\to 0$.
Again by concentration of measure arguments (see e.g.\ \cite[p.~41]{Led} or \cite[(1.4) or chapter 2]{P02}) there is a constant $\beta$ such that for any $N\ge 1$ and and $p\ge 2$ we have 
\begin{equation}\label{eqg1}
({\bb E}\|Y^{(N)}\|^p_{M_N})^{1/p} \le {\bb E}\|Y^{(N)}\|_{M_N} + \beta(p/N)^{1/2} \le 2+\vp(N) + \beta(p/N)^{1/2}.
\end{equation}
For any $\alpha\in M_n$, we denote
\[
 Z^{(N)}(\alpha)(\omega',\omega'') = \sum\nolimits^n_{i,j=1} \alpha_{ij} Y^{(N)}_i(\omega') \otimes Y^{(N)}_j(\omega'').
\]
Assume $\sum_{ij}|\alpha_{ij}|^2 = 1$. We will show that almost surely
\[
 \lim\nolimits_{N\to\infty} \|Z^{(N)}(\alpha)\| \le 4.
\]
Note that by the invariance of (complex) canonical Gaussian measures under unitary transformations, $Z^{(N)}(\alpha)$ has the same distribution as $Z^{(N)}(u\alpha v)$ for any pair $u,v$ of $n\times n$ unitary matrices. Therefore, if $\lambda_1,\ldots, \lambda_n$ are the eigenvalues of $|\alpha| = (\alpha^*\alpha)^{1/2}$, we have
\[
 Z^{(N)}(\alpha)(\omega',\omega'') \overset{\text{dist}}{=} \sum\nolimits^n_{j=1} \lambda_j Y^{(N)}_j(\omega') \otimes Y^{(N)}_j(\omega'').
\]
We claim that by a rather simple calculation of moments, one can show that for any even integer $p\ge 2$ we have
\begin{equation}\label{eqg2}
 {\bb E} \text{ tr}|Z^{(N)}(\alpha)|^p \le ({\bb E} \text{ tr}|Y^{(N)}|^p)^2.
\end{equation}
Accepting this claim for the moment, we find, a fortiori,  using \eqref{eqg1}:
\[
 {\bb E}\|Z^{(N)}(\alpha)\|^p_{M_N} \le N^2({\bb E}\|Y^{(N)}\|^p_{M_N})^2 \le N^2(2+\vp(N)+ \beta(p/N)^{1/2})^{2p}.
\]
Therefore for any $\delta>0$
\[
 {\bb P}\{\|Z^{(N)}(\alpha)\|_{M_N} > (1+\delta)4\} \le (1+\delta)^{-p} N^2(1+\vp(N)/2 + (\beta/2)(p/N)^{1/2})^{2p}.
\]
Then choosing (say) $p=5(1/\delta) \log( N)$ we find
\[
 {\bb P}\{\|Z^{(N)}(\alpha)\|_{M_N} > (1+\delta)4\} \in O(N^{-2})
\]
and hence (Borel--Cantelli) $\ovl{\lim}_{N\to\infty} \|Z^{(N)}(\alpha)\|_{M_N} \le 4$ a.s.. \\
It remains to verify the claim. Let $Z=Z^N(\alpha)$, $Y=Y^{(N)}$ and $p=2m$. We have
\[
 {\bb E} \text{ tr}|Z|^p = {\bb E} \text{ tr}(Z^*Z)^m = \sum \bar\lambda_{i_1}\lambda_{j_1} \ldots \bar\lambda_{i_m}\lambda_{j_m}({\bb E} \text{ tr}(Y^*_{i_1}Y_{j_1}\ldots Y^*_{i_m}Y_{j_m}))^2.
\]
Note that the only nonvanishing terms in this sum correspond to certain pairings that guarantee that both $\bar\lambda_{i_1}\lambda_{j_1}\ldots \bar\lambda_{i_m}\lambda_{j_m}\ge 0$ and ${\bb E} \text{ tr}(Y^*_{i_1}Y_{j_1}\ldots Y^*_{i_m}Y_{j_m})\ge 0$. Moreover, by H\"older's inequality for the trace we have
\[
 |{\bb E} \text{ tr}(Y^*_{i_1}Y_{j_1}\ldots Y^*_{i_m}Y_{j_m})| \le \Pi({\bb E} \text{ tr}|Y_{i_k}|^p)^{1/p} \Pi({\bb E} \text{ tr}|Y_{j_k}|^p)^{1/p} = {\bb E} \text{ tr}(|Y|^p).
\]
From these observations, we find
\begin{equation}\label{eqg3}
 {\bb E} \text{ tr}|Z|^p \le {\bb E} \text{ tr}(|Y|^p) \sum \bar\lambda_{i_1}\lambda_{j_1}\ldots \bar\lambda_{i_m}\lambda_{j_m} {\bb E} \text{ tr}(Y^*_{i_1}Y_{j_1}\ldots Y^*_{i_m} Y_{j_m})
\end{equation}
but the last sum is equal to ${\bb E} \text{ tr}(|\sum \lambda_jY_j|^p)$ and since $\sum \lambda_jY_j \overset{\text{dist}}{=} Y$   (recall $\sum |\lambda_j|^2 = \sum|\alpha_{ij}|^2=1$) we have
\[
 {\bb E} \text{ tr}\Big(\Big|\sum \alpha_jY_j\Big|^p\Big) = {\bb E} \text{ tr}(|Y|^p),
\]
and hence \eqref{eqg3} implies \eqref{eqg2}.
\end{proof}

\begin{cor}\label{cor13.11}
For any integer $n$ and $\vp>0$, there are $N$ and $n$-tuples of $N\times N$ matrices $\{Y'_i\mid 1\le i\le n\}$ and $\{Y''_j\mid 1\le j\le n\}$ in $M_N$ such that
\begin{align}\label{eq13.15}
 &\sup\left\{\left\|\sum^n_{i,j=1} \alpha_{ij}Y'_i \otimes Y''_j\right\|_{M_{N^2}}\ \Big| \ \alpha_{ij}\in {\bb C}, \ \sum\nolimits_{ij}|\alpha_{ij}|^2 \le 1\right\} \le (4+\vp)\\
\label{eq13.16}
&\min\left\{\frac1{nN} \sum\nolimits^n_1 \text{ \rm tr}|Y'_i|^2, \frac1{nN} \sum\nolimits^n_1 \text{ \rm tr}|Y''_j|^2\right\} \ge 1-\vp. 
\end{align}
\end{cor}

\begin{proof}

Fix $\vp>0$. Let ${\cl N}_\vp$ be a finite $\vp$-net in the unit ball of $\ell^{n^2}_2$. By 
Theorem \ref{eq13.16} we have for almost all $(\omega',\omega'')$  
\begin{equation}\label{eq13.18bis}
   \ovl{\lim}_{N\to\infty} \sup_{\alpha\in {\cl N}_\vp}\left\|\sum\nolimits^n_{i,j=1} \alpha_{ij} Y'_i\otimes Y''_j\right\|_{M_{N^2}} \le 4,
\end{equation}
We may pass from an $\vp$-net to the whole unit ball in \eqref{eq13.18bis} at the cost of an extra factor $(1+\vp)$ and we obtain \eqref{eq13.15}. As for \eqref{eq13.16}, the strong law of large numbers shows that the left side of \eqref{eq13.16} tends a.s.\ to 1. Therefore, we may clearly  find $(\omega',\omega'')$ satisfying both \eqref{eq13.15} and \eqref{eq13.16}.
\end{proof}
\begin{rem} A close examination of the proof 
and concentration of measure arguments show
that the preceding corollary holds
with $N$ of the order of $c(\vp) n^2$.
\end{rem}
\begin{rem}\label{rem13.12}
Using the well known ``contraction principle'' that says that the variables $(\vp_j)$ are dominated by either $(g^{\bb R}_j)$ or $(g^{\bb C}_j)$, it is easy to deduce that Corollary~\ref{cor13.11} is valid for matrices $Y'_i,Y''_j$ with entries all equal to $\pm N^{-1/2}$, with possibly a different numerical constant in place of 4. Analogously, using the polar factorizations  $Y'_i=U'_i|Y'_i|$, $Y''_j=U"_j|Y''_j |$ and noting
that all the factors $U'_i,|Y'_i| ,U"_j,|Y''_j |$ are independent, we
can also (roughly by integrating over the moduli $|Y'_i| , |Y''_j |$) obtain Corollary~\ref{cor13.11} with unitary matrices $Y'_i,Y''_j$ , with a different numerical constant in place of 4. 
\end{rem}

\section{GT for exact operator spaces}\label{sec13bis}

We now abandon tracially bounded maps and turn to c.b.\ maps.
In \cite{PS}, the following fact plays a crucial role.

\begin{lem}\label{lem13.5}
 Assume $E$ and $F$ are both exact. Let $A,B$ be $C^*$-algebras, with either $A$ or $B$ QWEP. Then for any $u$ in $CB(E,F^*)$ the bilinear map $\Phi_{A,B}$ introduced in Proposition \ref{pro11.7} satisfies
\[
\|\Phi_{A,B}\colon \ A\otimes_{\min} E\times B \otimes_{\min} F \to A\otimes_{\max} B\| \le  ex(E) ex(F)\|u\|_{\text{cb}}.
\]
\end{lem}

\begin{proof}
Assume $A=W/I$ with $W$ WEP. We also have, assuming say $B$ separable that $B = C/J$ with $C$ LLP (e.g.\ $C=C^*({\bb F}_\infty)$). The exactness of $E,F$ gives us
\[
 A\otimes_{\min}E = (W\otimes_{\min} E)/(I \otimes_{\min} E) \quad \text{and}\quad B \otimes_{\min} F=(C \otimes_{\min} F)/(J\otimes_{\min}F).
\]
Thus we are reduced to show the Lemma with $A=W$ and $B=C$. But then by Kirchberg's Theorem \ref{thm7.5bis}, we have $A \otimes_{\min} B = A \otimes_{\max} B$ (with equal norms).
\end{proof}

We now come to the operator space version of GT of \cite{PS}. We seem here to repeat the setting of Theorem \ref{thm13.1}. Note however that, by Remark \ref{rem13.3}, Theorem \ref{thm13.1} gives a characterization of \emph{tracially} bounded maps $u\colon \ E\to F^*$, while the next statement characterizes \emph{completely} bounded ones.

\begin{thm}[\cite{PS}]\label{thm13.6}
Let $E,F$ be exact operator spaces. Let $u\in CB(E,F^*)$. Assume $\|u\|_{cb}\le 1$.  Let $C=ex(E) ex(F)$. Then for any finite sets $(x_j)$ in $E,(y_j)$ in $F$ we have
\begin{equation}\label{eq13.8}
\left|\sum \langle ux_j,y_j\rangle\right| \le 2C(\|(x_j)\|_R \|(y_j)\|_C + \|(x_j)\|_C \|(y_j)\|_R).
\end{equation}
Equivalently, assuming $E\subset A$, $F\subset B$ there are states $f_1,f_2$ on $A$ $g_1,g_2$ on $B$ such that
\begin{equation}\label{eq13.9}
\forall (x,y)\in E\times F\qquad\qquad |\langle ux,y\rangle| \le 2C((f_1(xx^*)g_1(y^*y))^{1/2} + (f_2(x^*x) g_2(yy^*))^{1/2}).
\end{equation}
Conversely if this holds for some $C$ then $\|u\|_{\text{cb}} \le 4C$.
\end{thm}

\begin{proof}
The main point is \eqref{eq13.8}. The equivalence of \eqref{eq13.8} and \eqref{eq13.9} is explained below in Proposition \ref{pro14.2}. To prove \eqref{eq13.8}, not surprisingly, Gaussian variables reappear as a crucial ingredient. But it is their analogue in Voiculescu's free probability theory (see \cite{VDN}) that we need, and actually we use them in Shlyakhtenko's generalized form. To go straight to the point, what this theory does for us is this:\ Let $(t_j)$ be an arbitrary finite set with $t_j>0$. Then we can find operators $(a_j)$ and $(b_j)$ on a (separable) Hilbert space $H$ and a unit vector $\xi$ such that
\begin{itemize}
\item[(i)] $a_ib_j = b_ja_i$ and $a^*_ib_j = b_ja^*_i$ for all $i,j$.
\item[(ii)] $\langle a_ib_j\xi,\xi\rangle = \delta_{ij}$.
\item[(iii)] For any $(x_j)$ and $(y_j)$ in $B(K)$ ($K$ arbitrary Hilbert)
\begin{align*}
\left\|\sum a_j\otimes x_j\right\|_{\min} &\le \|(t_jx_j)\|_R + \|(t^{-1}_jx_j)\|_C\\
\left\|\sum b_j\otimes y_j\right\|_{\min} &\le \|(t_jy_j)\|_R + \|(t^{-1}_jy_j)\|_C.
\end{align*}
\item[(iv)] The $C^*$-algebra generated by $(a_j)$ is QWEP.
\end{itemize}
With this ingredient, the proof of \eqref{eq13.8} is easy to complete:\ By  \eqref{eq-max}  and Lemma \ref{lem13.5}, (ii) implies  
\begin{align*}
\left|\sum \langle ux_j,y_j\rangle\right| = \left|\sum_{i,j} \langle ux_i,y_j\rangle \langle a_ib_j\xi, \xi\rangle\right|
&\le \left\|\sum \langle ux_i,y_j\rangle a_i\otimes b_j\right\|_{\max}\\
&\le C\left\|\sum a_j\otimes x_j\right\|_{\min} \left\|\sum b_j\otimes y_j\right\|_{\min}\end{align*}
{and hence (iii) yields}
\[\left|\sum \langle ux_j,y_j\rangle\right| \le C(\|(t_jx_j)\|_R + \|(t^{-1}_jx_j)\|_C) (\|(t_jy_j)\|_R + \|(t^{-1}_jy_j)\|_C).
\]
A fortiori (here we use the elementary inequality $(a+b)(c+d) \le 2(a^2+b^2)^{1/2} (c^2+d^2)^{1/2} \le$ $s^2(a^2+b^2) + s^{-2}(c^2+d^2)$ valid for non-negative numbers $a,b,c,d$ and $s>0$)
\[
 \left|\sum \langle ux_j,y_j\rangle\right| \le C\left(s^2\left\|\sum t^2_jx^*_jx_j\right\| + s^2\left\|\sum t^{-2}_jx_jx^*_j\right\| + s^{-2}\left\|\sum t^2_jy^*_jy_j\right\| + s^{-2} \left\|\sum t^{-2}_jy_jy^*_j\right\|\right).
\]
By the Hahn--Banach argument (see \S \ref{sechb}) this implies the existence of states $f_1,f_2,g_1,g_2$ such that 
\begin{equation}
|\langle ux,y\rangle| \le C(s^2t^2f_1(x^*x) + s^2t^{-2}f_2(xx^*) + s^{-2}t^2g_2(y^*y) + s^{-2}t^{-2}g_1(yy^*)).\tag*{$\forall(x,y)\in E\times F$}
\end{equation}
Then taking the infimum over all $s,t>0$ we obtain \eqref{eq13.9} and hence \eqref{eq13.8}. 
\end{proof}

We now describe briefly the generalized complex free Gaussian variables that we use, following Voiculescu's and Shlyakhtenko's ideas. One nice realization of these variables is on the Fock space. But while it is well known that Gaussian variables can be realized using the symmetric Fock space, here we need the ``full'' Fock space, as follows:\ We assume $H = \ell_2(I)$, and we define
\[
 F(H) = {\bb C}\oplus H\oplus H^{\otimes 2}\oplus\cdots
\]
where $H^{\otimes n}$ denotes the Hilbert space tensor product of $n$-copies of $H$. As usual one denotes by $\Omega$ the unit in ${\bb C}$ viewed as sitting in $F(H)$ (``vacuum vector'').

Given $h$ in $H$, we denote by $\ell(h)$ (resp.\ $r(h)$) the left (resp.\ right)  creation operator, defined by
$\ell(h)x = h\otimes x$ (resp.\ $r(h)x = x\otimes h$). Let $(e_j)$ be the canonical basis of $\ell_2(I)$. Then the family $\{\ell(e_j) + \ell(e_j)^* \mid j\in I\}$ is a ``free semi-circular'' family (\cite{VDN}). This is the free analogue of $(g^{\bb R}_j)$, so we refer to it instead as a ``real free Gaussian'' family. But actually, we need the complex variant. We assume $I = J\times \{1,2\}$ and then for any $j$ in $J$ we set
\[
 C_j = \ell(e_{(j,1)}) + \ell(e_{(j,2)})^*.
\]
The family $(C_j)$ is a ``free circular'' family (cf.\ \cite{VDN}), but we refer to it as a ``complex free Gaussian'' family. The generalization provided by Shlyakhtenko is crucial for the above proof. This uses the positive weights $(t_j)$ that we assume fixed. One may then define
\begin{align}\label{eq13.10}
 a_j &= t_j\ell(e_{(j,1)}) + t^{-1}_j\ell(e_{(j,2)})^*\\
\label{eq13.11}
b_j &= t_jr(e_{(j,2)}) + t^{-1}_jr(e_{(j,1)})^*,
\end{align}
and set $\xi=\Omega$.
Then it is not hard to check that (i), (ii) and (iii) hold. In sharp contrast, (iv) is much more delicate but it follows from known results from free probability, namely the existence of asymptotic ``matrix models'' for free Gaussian variables or their generalized  versions \eqref{eq13.10} and \eqref{eq13.11}.  

\begin{cor}\label{cor13.7}
An operator space $E$ is exact as well as its dual $E^*$ iff there are Hilbert spaces $H,K$ such that $E\simeq H_r\oplus K_c$ completely isomorphically.
\end{cor}

\begin{proof}
By Theorem \ref{thm13.6} (and Proposition \ref{pro14.2} below) the identity of $E$ factors through ${\cl H}_r\oplus {\cl K}_c$. By a remarkable result due to Oikhberg \cite{O}
(see \cite{P44} for a simpler proof), this can happen only if $E$ itself is already completely isomorphic to ${ H}_r \oplus { K}_c$ for subspaces $ H{\subset \cl H}$ and $K\subset {\cl K}$. This proves the only if part. The ``if part'' is obvious since $H_r,K_c$ are exact and $(H_r)^* \simeq H_c$, $(K_c)^*\simeq K_r$.
\end{proof}

We postpone to the next chapter the discussion of the case when $E=A$ and $F=B$
in Theorem \ref{thm13.6}. We will also indicate there a new proof of Theorem \ref{thm13.6} based on Theorem \ref{thm12.1} and the more recent results of \cite{HM2}.

\section{GT for operator spaces}\label{sec14}

In the Banach space case, GT tells us that bounded linear maps $u\colon \ A\to B^*$ ($A,B$ $C^*$-algebras, commutative or not, see \S \ref{sec1} and \S \ref{sec4}) factor (boundedly) through a Hilbert space. In the operator space case, we consider c.b.\ maps $u\colon \ A\to B^*$ and we look for a c.b.\ factorization through some Hilbertian operator space. It turns out that, if $A$ or $B$ is separable,  the relevant Hilbertian space is the direct sum 
\[
R\oplus C
\]
of the row and column spaces introduced in \S  \ref{sec11}, or more generally the direct sum
\[
 H_r\oplus K_c
\]
where $H,K$ are arbitrary Hilbert spaces.

In the o.s.\ context, it is more natural to define the direct sum of two operator spaces $E,F$ as the ``block diagonal'' direct sum, i.e.\ if $E\subset A$ and $F\subset B$ we embed $E\oplus F\subset A\oplus B$ and equip $E\oplus F$ with the induced o.s.s. Note that for all $(x,y)\in E\oplus F$ we have then $\|(x,y)\| = \max\{\|x\|, \|y\|\}$. Therefore the spaces $R\oplus C$ or $H_r\oplus K_c$ are not isometric but only $\sqrt 2$-isomorphic to Hilbert space, but this will not matter much.

In analogy with \eqref{eq13.3}, for any linear map $u\colon \ E\to F$ between operator spaces we denote by    $\gamma_{r\oplus c}(u)$) the constant of factorization of $u$ through a space of the form  $H_r\oplus K_c$. More precisely, we set
\[
 \gamma_{r\oplus c} (u) = \inf\{\|u_1\|_{\text{cb}} \|u_2\|_{\text{cb}}\}
\]
where the infimum runs over all possible Hilbert spaces $H,K$ and all factorizations
\[
 E \overset{\sst u_1}{\longrightarrow} Z \overset{\sst u_2}{\longrightarrow} F
\]
of $u$ through $Z$ with $Z = H_r\oplus K_c$.  
Let us state a first o.s.\ version of GT.

\begin{thm}\label{thm14.1}
Let $A,B$ be $C^*$-algebra. Then any c.b.\ map $u\colon \ A\to B^*$ factors through a space of the form $H_r\oplus K_c$ for some Hilbert spaces $H,K$. If $A$ or $B$ is separable, we can replace $H_r\oplus K_c$ simply by $R\oplus C$. More precisely,   for any such $u$ we have $\gamma_{r\oplus c}(u) \le 2\|u\|_{\text{cb}}$.
\end{thm}

Curiously, the scenario of the non-commutative GT repeated itself:
this was proved in \cite{PS} assuming that either $A$ or $B$ is exact, or assuming that $u$ is suitably approximable by finite rank linear maps. These restrictions were removed in the recent paper \cite{HM2} that also obtained better constants. A posteriori, this means (again!)
that the approximability assumption of \cite{PS} for c.b. maps from $A$ to $B^*$  holds in general. The case of
exact operator \emph{subspaces} $E\subset A,F\subset B$ (i.e. Theorem \ref{thm13.6}) a priori does not follow from the method in \cite{HM2} but,
in the  second proof of Theorem \ref{thm13.6} given at the end of this section, we will show that it can also be derived from the ideas of \cite{HM2}
using Theorem \ref{thm13.1} in place of Theorem \ref{thm3.1}.

Just like in the classical GT, the above factorization is equivalent to a specific inequality which we now describe, following \cite{PS}.
\begin{pro}\label{pro14.2}
Let $E\subset A, F\subset B$ be operator spaces and let $u\colon \ E\to F^*$ be a linear map (equivalently we may consider a bilinear form on $E\times F$). The following assertions are equivalent:
\begin{itemize}
\item[\rm (i)] For any finite sets $(x_j), (y_j)$ in $E,F$ respectively and for any number $t_j>0$ we have
\[
\left|\sum \langle ux_jy_j\rangle\right| \le (\|(x_j)\|_R \|(y_j)\|_C + \|(t_jx_j)\|_C \|(t^{-1}_jy_j)\|_R).
\]
\item[\rm (ii)] There are states $f_1,f_2$ on $A,g_1,g_2$ on $B$ such that
\begin{equation}
|\langle ux,y\rangle| \le (f_1(xx^*) g_1(y^*y))^{1/2} + (f_2(x^*x) g_2(yy^*))^{1/2}.\tag*{$\forall(x,y) \in E\times F$}
\end{equation}
\item[\rm (iii)] There is a decomposition $u = u_1+u_2$ with maps $u_1\colon \ E\to F^*$ and $u_2\colon \ E\to F^*$ such that \begin{equation}
|\langle u_1x,y\rangle| \le (f_1(xx^*) g_1(y^*y))^{1/2} \ {\rm and}\ |\langle u_2 x,y\rangle| \le (f_2(x^*x) g_2(yy^*))^{1/2}.\tag*{$\forall(x,y) \in E\times F$}
\end{equation} 
\item[\rm (iv)] There is a decomposition $u = u_1+u_2$ with maps $u_1\colon \ E\to F^*$ and $u_2\colon \ E\to F^*$ such that $\gamma_r(u_1)\le 1$ and $\gamma_c(u_2)\le 1$. 
\end{itemize}
In addition, the bilinear form associated to $u$ on $E\times F$ extends to one on $A\times B$ that still satisfies {\rm (i)}.
Moreover, these conditions imply $\gamma_{r\oplus c}(u)\le 2$, and conversely $\gamma_{r\oplus c}(u)\le 1$ implies these equivalent conditions.

\end{pro}

\begin{proof}{(Sketch)}
The equivalence between (i) and (ii) is proved by the   Hahn--Banach type argument
in \S \ref{sechb}. (ii) $\Rightarrow$ (iii)  (with the same states) requires a trick of independent interest,
due to the author, see \cite[Prop. 5.1]{X2}. Assume (iii). Then by Theorem \ref{thm12.1} and Remark \ref{rem12.2}, we have  
$\gamma_r(u_1)\le 1$ and $\gamma_c(u_2)\le 1$.
By the triangle inequality this implies (ii). The extension property follows from (iii) in Theorem \ref{hb2}. The last assertion is easy by Theorem \ref{thm12.1} and Remark \ref{rem12.2}.
\end{proof}

The key ingredient for the proof of Theorem \ref{thm14.1} is the ``Powers factor'' $M^\lambda$, i.e.\ the von~Neumann algebra associated to the state
\[
 \varphi^\lambda = \bigotimes_{\bb N} \begin{pmatrix} \frac\lambda{1+\lambda}&0\\ 0&\frac1{1+\lambda}
\end{pmatrix}
\]
on the infinite tensor product of $2\times 2$ matrices. If $\lambda\not =1$ the latter is of `` type III", i.e. does \emph{not admit}
any kind of trace.  We will try
to describe ``from scratch" the main features that are needed for our purposes  in
a somewhat self-contained way, using as little von Neumann Theory
as possible. Here (and throughout this section) $\lambda$ will be a fixed number such that
\[
 0<\lambda<1.
\] 
For simplicity of notation, we will drop the superscript $\lambda$ and denote simply $\varphi,M,N,\ldots$ instead of $\varphi^\lambda,M^\lambda,N^\lambda,\ldots$ but the reader should recall these do depend on $\lambda$. The construction of $M^\lambda$ (based on the classical GNS construction) can be outlined as follows:

With $M_2$ denoting the $2\times 2$ complex matrices, let ${\cl A}_n = M^{\otimes n}_2$ equipped with the $C^*$-norm inherited from the identification with $M_{2^n}$.
Let ${\cl A} = \cup{\cl A}_n$ where we embed  ${\cl A}_n$ into ${\cl A}_{n+1}$ via the isometric map $x\to x\otimes 1$. Clearly, we may equip
$\cl A$ with the norm inherited from the norms of the algebras $\cl A_n$.

Let $\varphi_n = \psi\otimes\cdots\otimes \psi$ ($n$-times) with $\psi = \begin{pmatrix} \frac\lambda{1+\lambda}&0\\ 0&\frac1{1+\lambda}\end{pmatrix}$. We define $\varphi\in {\cl A}^*$ by 
\begin{equation}
 \varphi(a) =  \lim_{n\to\infty} \text{tr}(\varphi_na)\tag*{$\forall a\in {\cl A}$}
\end{equation}
where the limit is actually stationary, i.e.\ we have
$ \varphi(a) = \text{tr}(\varphi_na) 
$ {$\forall a\in {\cl A}_n$}
(here the trace is  meant in $M^{\otimes n}_2\simeq M_{2^n}$). We equip ${\cl A}$ with the inner product:
\begin{equation}
 \langle a,b\rangle = \varphi(b^*a).\tag*{$\forall a,b\in {\cl A}$}
\end{equation}
The space $L_2(\varphi)$ is then defined as the Hilbert space obtained from $({\cl A}, \langle \cdot,\cdot\rangle)$ after  completion.

We observe that $\cl A$ acts on $L_2(\varphi)$ by left multiplication, 
since we have $\|ab\|_{L_2(\varphi)} \le \|  a\|_{\cl A}\|  b\|_{L_2(\varphi)}$ for all $a,b\in \cl A$.
 So
from now on we view
$$\cl A\subset B( L_2(\varphi)). $$
We then let $M$ be the von Neumann algebra generated by $\cl A$, i.e. we set $M= {\cl A}''$ (bicommutant). Recall that, by classical results, the unit ball of
$\cl A$ is dense in that of $M$ for either the weak or strong operator topology
(wot and sot in short).

Let $L$ denote the inclusion map into $B(L_2(\varphi))$. Thus \begin{equation}\label{eq14.10}
 L\colon \ M\to B(L_2(\varphi))
\end{equation}
is  an isometric $*$-homomorphism
extending the action of left multiplication. Indeed, let $b\to \dot b$ denote the dense range inclusion of ${\cl A}$ into $L_2(\varphi)$. Then we have 
\begin{equation}
 L(a)\dot {b} =  \overbrace{ab}^{.} \tag*{$\forall a\in {\cl A}~~\forall b\in {\cl A}$}.
\end{equation}
Let $\xi=\dot 1$. Note $L(a)\xi= \dot a$ 
 and also  
$  \langle L(a)\xi,\xi\rangle = \varphi(a)
$ for all $a$ in $\cl A$.
Thus we can extend $\varphi$ to the whole of $M$ by setting
\begin{equation}\label{eq14.11}\forall a\in  M\qquad\qquad \varphi(a)=\langle L(a)\xi,\xi\rangle .\end{equation}
We wish to also have an action of $M$ analogous to right multiplication. Unfortunately, when $\lambda\ne 1$, $\varphi$ is \emph{not tracial} and hence right multiplication by elements of $M$ is \emph{unbounded} on $L_2(\varphi)$. Therefore we need a modified version of right multiplication, as follows.

For any $a,b$ in ${\cl A}$, let $n$ be such that $a,b\in {\cl A}_n$, we define
\begin{equation}
 R(a)\dot b =  {\overbrace{b(\varphi^{1/2}_na\varphi^{-\frac12}_n)}^.}.
\end{equation}
Note that   this does not depend on $n$ (indeed $\varphi^{1/2}_{n+1}(a\otimes 1)\varphi_{n+1}^{-\frac12} = (\varphi^{1/2}_na\varphi^{-\frac12}_n)\otimes 1)$. \\ A simple verification shows that
\begin{equation}
 \|R(a)\dot b\|_{L_2(\varphi)} \le  \|  a\|_{\cl A}\|\dot b\|_{L_2(\varphi)},\tag*{$\forall a,b\in {\cl A}$}
\end{equation}
and hence this defines $  R(a)\in B(L_2(\varphi))$ by density  of $\cl A$ in $L_2(\varphi)$.  Note that  \begin{equation}\label{x}
\forall a  \in  \cl A \quad \langle R(a) \xi ,\xi\rangle =\langle L(a) \xi ,\xi\rangle=\varphi(a)\quad {\rm and}\quad  R(a^*)=R(a)^* .\end{equation}
Using  this together with the  wot-density of the unit ball of 
$\cl A$ in that of $M$, we can extend $a \mapsto R(a)$ to the whole of $M$.
Moreover, by \eqref{x} we have (note $R(a)R(a^*)=R(a^*a)$)
 \begin{equation}\label{xx}
\forall a  \in  M \quad \langle R(a^*) \xi ,R(a^*)\xi\rangle   =\langle L(a) \xi ,L(a)\xi\rangle=\varphi(a^*a).
\end{equation}
Since left and right multiplication obviously commute we have
\begin{equation}
 R(a_1)L(a_2) = L(a_2)R(a_1)\tag*{$\forall a_1,a_2\in {\cl A}$},
\end{equation} and hence also for all $a_1,a_2\in M$.
Thus  
   we obtain a $*$-anti-homomorphism
\[
 R\colon \ M\to B(L_2(\varphi)),
\]
i.e.\ such that $R(a_1a_2) = R(a_2)R(a_1)$ $\forall a_1,a_2\in M$. Equivalently, let $M^{op}$ denote the von~Neumann algebra that is ``opposite'' to $M$, i.e.\ the same normed space with the same involution but with reverse product  (i.e.\ $a\cdot b=ba$ by definition). If $M\subset B(\ell_2)$, $M^{op}$ can be realized concretely as the algebra $\{{}^tx\mid x\in M\}$. Then $R\colon \ M^{op}\to B(L_2(\varphi))$ is a bonafide $*$-homomorphism. We may clearly do the following identifications 
\[
 M_{2^n} \simeq {\cl A}_n \simeq L({\cl A}_n)  \simeq M_2 \otimes\cdots\otimes M_2 \otimes 1\otimes \cdots
\]
where $M_2$ is repeated $n$ times and 1 infinitely many times.

In particular we view ${\cl A}_n\subset M$, so that we also have $L_2({\cl A}_n, \varphi_{|{\cl A}_n}) \subset L_2(\varphi)$ and $\varphi_{|{\cl A}_n} = \varphi_n$. Let $P_n\colon \ L_2(\varphi)\to L_2({\cl A}_n, \varphi_{|{\cl A}_n})$ be the orthogonal projection.\\ 
Then for any $a$ in $\cl A$, say $a=a_1\otimes \cdots\otimes a_n\otimes a_{n+1}\otimes\cdots$, we have
$$P_n(a)=a_1\otimes \cdots\otimes a_n\otimes 1 \cdots (\psi(a_{n+1})\psi(a_{n+2})\cdots),$$
This behaves   like a conditional expectation, e.g.
for any $a,b\in \cl A_n$, we have
$P_n L(a)R(b)  = L(a)R(b) P_n $. 
Moreover, 
since the operator 
$P_n L(a)_{|L_2({\cl A}_n, \varphi_{|{\cl A}_n})}$
commutes with right multiplications,
for any $a\in M$, there is a unique element $a_n \in \cl A_n$ such that
$\forall x\in L_2({\cl A}_n, \varphi_{|{\cl A}_n}), P_n L(a)x=L(a_n)x$. We will
denote ${\bb E}_n(a)=a_n$. Note that ${\bb E}_n(a^*)={\bb E}_n(a)^*$.
For any $x\in L_2(\varphi)$, by a density argument, we obviously have $P_n(x)\to x$
in $L_2(\varphi)$. Note that for any $a\in M$, we have
$ L(a_n) \xi =P_n (L(a)\xi)\to L(a)\xi $ in $L_2(\varphi)$, thus, using \eqref{xx} we find 
\begin{equation}\label{conv}
\|L(a-a_n) \xi\|_{L_2(\varphi)}=\|R(a^*-a^*_n) \xi\|_{L_2(\varphi)}\to 0.
\end{equation}
Note that, if we identify  (completely isometrically) the elements of ${\cl A}_n$ with left
multiplications on $L_2({\cl A}_n, \varphi_{|{\cl A}_n})$, we may write
   ${\bb E}_n(.)= P_n L(.)_{|L_2({\cl A}_n, \varphi_{|{\cl A}_n})}$.
The latter shows (by Theorem \ref{thm8.3}) that  
  ${\bb E}_n\colon M\to \cl A_n$ is   a completely contractive mapping. Thus we have
\begin{equation}\label{eq14.12}
 \forall n\ge 1\qquad\qquad \|{\bb E}_n\colon \ M\to {\cl A}_n\|_{cb}\le 1.
\end{equation}
Obviously we may write as well:
\begin{equation}\label{eq14.13}
 \forall n\ge 1\qquad \qquad \|{\bb E}_n\colon \ M^{op}\longrightarrow {\cl A}^{op}_n\|_{cb} \le 1.
\end{equation}
(The reader should observe that $\bb E_n$ and $P_n$ are essentially 
the same map, since we have $\overbrace{{\bb E_n}(a)}^{.}=P_n( \dot{a})$ ,
or equivalently $L({\bb E_n}(a)) \xi = P_n( L(a)\xi)$ but the multiple identifications may be confusing.)

Let $N\subset M$ denote the $\varphi$-invariant subalgebra, i.e.
\[
 N = \{a\in M\mid \varphi(ax) = \varphi(xa) \quad \forall x\in M\}.
\]
Obviously, $N$ is a von Neumann subalgebra of $M$. Moreover $\varphi_{|N}$ is a faithful tracial state (and actually a ``vector state"  by \eqref{eq14.11}), so $N$ is a \emph{finite} von~Neumann algebra.\\ We can now state the \emph{key analytic ingredient} for the proof of Theorem \ref{thm14.1}.

\begin{lem}\label{lem14.3}
Let $\Phi\colon \ M\times M^{op}\to {\bb C}$ and $\Phi_n\colon \ M\times M^{op}\to {\bb C}$ be the bilinear forms defined by
\begin{align}
\Phi(a,b) &= \langle L(a) R(b)\xi,\xi\rangle\tag*{$\forall a,b\in M$}\\
\Phi_n(a,b) &= \text{\rm tr}({\bb E}_n(a) \varphi^{1/2}_n{\bb E}_n(b)\varphi^{1/2}_n)\notag
\end{align}
where $\text{\rm tr}$ is the usual trace on $M^{\otimes n}_2 \simeq {\cl A}_n\subset M$. 
\begin{itemize}
 \item[\rm (i)] For any $a,b$ in $M$, we have
\end{itemize}
\begin{equation}\label{eq14.14}
 \forall a,b\in M\qquad\qquad \lim_{n\to\infty} \Phi_n(a,b)= \Phi(a,b).
\end{equation}
\begin{itemize}
\item[\rm (ii)] For any $u$ in ${\cl U}(N)$ (the set of unitaries in $N$) we have
\end{itemize}
\begin{equation}\label{eq14.15}
 \forall a,b\in M\qquad\qquad \Phi(uau^*, ubu^*) = \Phi(a,b)
\end{equation}
\begin{itemize}
\item[\rm (iii)] For any $u$ in $M$, let $\alpha(u)\colon \ M\to M$ denote the mapping defined by $\alpha(u)(a) = uau^*$. Let ${\cl C} = \{\alpha(u)\mid u\in {\cl U}(N)\}$. There is a net of mappings $\alpha_i$ in $\text{\rm conv}({\cl C})$ such that $\|\alpha_i(v) - \varphi^\lambda(v)1\|\to 0$ for any $v$ in $N$.
\item[\rm (iv)] For any $q\in {\bb Z}$, there exists $c_q$ in $M$ (actually in ${\cl A}_{|q|}$) such that $c^*_qc_q$ and $c_qc^*_q$ both belong to $N$ and such that:
\end{itemize}
\begin{equation}\label{eq14.17}
 \varphi(c^*_qc_q) = \lambda^{-q/2},\quad \varphi(c_qc^*_q) = \lambda^{q/2} \quad \text{and}\quad \Phi(c_q,c^*_q)=1.
\end{equation}
\end{lem}
We give below a \emph{direct} proof of this Lemma that is ``almost'' self-contained (except for the use of Dixmier's classical averaging theorem and a certain complex interpolation argument). But we first show how this Lemma yields Theorem \ref{thm14.1}.

\begin{proof}[Proof of Theorem \ref{thm14.1}]
Consider the bilinear form
\[
 \widehat\Phi\colon \ M \otimes_{\min} A \times M^{op} \otimes_{\min} B \longrightarrow {\bb C}
\]
defined (on the algebraic tensor product) by
\[
\widehat\Phi(a\otimes x,b\otimes y) = \Phi(a,b) \langle ux,y\rangle.
\]
We define $\widehat\Phi_n$ similarly. We claim that $\widehat\Phi$ is bounded with
\[
 \|\widehat\Phi\|\le \|u\|_{cb}.
\]
Indeed, by Proposition \ref{pro11.7} (note the transposition sign there), and by \eqref{eq14.12} and \eqref{eq14.13} we have
\begin{align*}
 \left|\widehat\Phi_n\left(\sum a_r\otimes x_r, \sum b_s\otimes y_s\right)\right|
\le~ &\|u\|_{cb} \left\|\sum {\bb E}_n(a_r)\otimes x_r\right\|_{M_{2^n}(A)} \left\|\sum {}^t{\bb E}_n(b_s) \otimes y_s\right\|_{M_{2^n}(B)}\\
\le~ &\|u\|_{cb} \left\|\sum a_r\otimes x_r\right\|_{M\otimes_{\min}A} \left\|\sum b_s \otimes y_s\right\|_{M^{op}\otimes_{\min} B},
\end{align*}
and then by \eqref{eq14.14} we obtain the announced claim $\|\widehat\Phi\|\le \|u\|_{cb}$. \\
But now by Theorem 6.1, there are states $f_1,f_2$ on $M\otimes_{\min} A$, $g_1,g_2$ on $M^{op} \otimes_{\min} B$ such that
\begin{equation}
|\widehat\Phi(X,Y)| \le \|u\|_{cb} (f_1(X^*X) + f_2(XX^*))^{1/2} (g_2(Y^*Y) + g_1(YY^*))^{1/2} \tag*{$\forall X\in M\otimes A, \forall Y\in M\otimes B$}
\end{equation}
In particular, we may apply this when $X = c_q\otimes x$ and $Y = c^*_q\otimes y$. Recall that by \eqref{eq14.17} $\Phi(c_q,c^*_q)=1$. We find
\[
 |\langle ux,y\rangle| |\Phi(c_q,c^*_q)|\le \|u\|_{cb}(f_1(c^*_qc_q\otimes x^*x) + f_2(c_qc^*_q \otimes xx^*))^{1/2} (g_2(c^*_qc_q\otimes y^*y) + g_1(c_qc^*_q\otimes yy^*))^{1/2}.
\]
But then by \eqref{eq14.15} we have for any $i$
\[
 \Phi(\alpha_i(c_q), \alpha_i(c^*_q)) = \Phi(c_q,c^*_q)
\]
and since $c^*_qc_q, c_qc^*_q\in N$, we know by Lemma \ref{lem14.3} and \eqref{eq14.17} that $\alpha_i(c^*_qc_q) \to \varphi(c^*_qc_q)1 = \lambda^{-q/2}1$ while $\alpha_i(c_qc^*_q)\to \varphi(c_qc^*_q)1 = \lambda^{q/2}1$. Clearly this implies $\alpha_i(c^*_qc_q)\otimes x^*x\to \lambda^{-q/2}1 \otimes x^*x$ and $\alpha_i(c_qc^*_q)\otimes xx^*\to \lambda^{q/2}1\otimes xx^*$ and similarly for $y$. It follows that if we denote
\[
 \forall x\in A\qquad \tilde f_k(x) = f_k(1\otimes x)\quad \text{and}\quad \forall y\in B\qquad \tilde g_k(y) = g_k(1\otimes y)
\]
then $\tilde f_k,\tilde g_k$ are states on $A,B$ respectively such that
\begin{equation}
\forall (x,y)\in A\times B \quad |\langle ux,y\rangle|\le \|u\|_{cb}(\lambda^{-q/2}\tilde f_1(x^*x) + \lambda^{q/2}\tilde f_2(xx^*))^{1/2} (\lambda^{-q/2}\tilde g_2(y^*y) + \lambda^{q/2} \tilde g_1(yy^*))^{1/2}.
\end{equation}
Then we find
\[
 |\langle ux,y\rangle|^2 \le \|u\|^2_{cb} (\tilde f_1(x^*x) \tilde g_1(yy^*) + \tilde f_2(xx^*)\tilde g_2(y^*y) + \delta_q(\lambda))
\]
where we set 
\[ \delta_q(\lambda) = \lambda^{-q}\beta+\lambda^q\alpha \quad \text{with}\quad
 \beta = \tilde f_1(x^*x) \tilde g_2(y^*y)\quad \text{and}\quad \alpha = \tilde f_2(xx^*)\tilde g_1(yy^*).
\]
But an elementary calculation shows that (here we crucially use that $\lambda<1$)
\[
 \inf_{q\in {\bb Z}} \delta_q(\lambda) \le (\lambda^{1/2}+\lambda^{-1/2}) \sqrt{\alpha\beta}
\]
so after minimizing over $q\in {\bb Z}$ we find 
\[
 |\langle ux,y\rangle|\le \|u\|_{cb} (\tilde f_1(x^*x)\tilde g_1(yy^*) + \tilde f_2(xx^*) \tilde g_2(y^*y) + (\lambda^{1/2} + \lambda^{-1/2}) \sqrt{\alpha\beta})^{1/2}
\]
and since $C(\lambda) = (\lambda^{1/2}+\lambda^{-1/2})/2 \ge 1$ we obtain $\forall x\in A\ \forall y\in B$
\[
 |\langle ux,y\rangle| \le \|u\|_{cb} C(\lambda)^{1/2} \left((\tilde f_1(x^*x)\tilde g_1(yy^*))^{1/2} + (\tilde f_2(xx^*)\tilde g_2(y^*y))^{1/2}\right).
\]

To finish, we note that $C(\lambda)\to 1$ when $\lambda\to 1$, 
and, by pointwise compactness,  we may assume that the above states
$\tilde f_1,\tilde f_2,\tilde g_1,\tilde g_2$ that (implicitly) depend on $\lambda$ are also  converging pointwise
when  $\lambda\to 1$. Then we obtain the announced result, i.e. the last inequality holds
with the constant 1 in place of $C(\lambda)$.

\end{proof}

\begin{proof}[Proof of Lemma \ref{lem14.3}]
(i) Let $a_n = {\bb E}_n(a)$, $b_n = {\bb E}_n(b)$, $(a,b\in M)$.  We know that $\dot a_n\to \dot  a$, $\dot b _n\to \dot  b$ and also $\overbrace{a^*_n}^{.}\to \overbrace{a^*}^{.}$. Therefore we have by \eqref{conv}
\[
\Phi_n(a,b) = \langle L(a_n)\xi, R(b^*_n)\xi\rangle   \to   \langle L(a)\xi, R(b^*)\xi\rangle = \Phi(a,b).
\]
(iii) Since $\varphi_{|  N}$ is a (faithful normal) tracial state on $N$,
$N$ is a {\it finite} factor, so this follows from Dixmier's classical approximation theorem (\cite{Dix} or \cite[p. 520]{KR}).\\
(iv) The case $q=0$ is trivial, we set $c_0=1$. Let $c = (1+\lambda)^{1/2}\lambda^{-1/4}e_{12} $. We then set, for $q\ge 1$, $c_q=c \otimes\cdots\otimes c\otimes 1\cdots$ where $c$ is repeated $q$ times and for $q<0$ we set $c_q=(c_{-q})^*$. The verification of \eqref{eq14.17} then boils down to the observation that 
\[
\psi(e^*_{12}e_{12}) = (1+\lambda)^{-1},\quad   \psi(e_{12}e^*_{12}) = \lambda(1+\lambda)^{-1} \quad\text{and}\quad \text{tr}(\psi^{1/2}e_{12} \psi^{1/2}e_{21}) = (1+\lambda)^{-1}\lambda^{1/2}.
\]
(ii) By polarization, it suffices to show \eqref{eq14.15}   for $b=a^*$. The proof can be completed using
properties of self-polar forms (one can also use the Pusz-Woronowicz ``purification of states" ideas). We outline
an argument based on complex interpolation.
Consider the pair of Banach spaces $(L_2(\varphi), L_2(\varphi)^\dagger)$ where
$L_2(\varphi)^\dagger  $ denotes the completion of $\cl A$ with respect to the
norm $x\mapsto (\varphi (xx^*))^{1/2}$. Clearly we have natural
continuous inclusions of both these spaces into $M^*$ (given respectively
by $x\mapsto x\varphi $ and $x\mapsto \varphi x$). Let us denote
$$\Lambda(\varphi)=(L_2(\varphi), L_2(\varphi)^\dagger)_{1/2}.$$
Similarly we denote for any $n\ge 1$
$$\Lambda(\varphi_n)=(L_2({\cl A}_n,\varphi_n), L_2(\cl A_n,\varphi_n)^\dagger)_{1/2}.$$
By a rather easy elementary (complex variable) argument one checks     that for any $a_n\in {\cl A}_n$ we have
$$\| a_n\|^2_{\Lambda(\varphi_n)}={\rm tr}(\varphi_n^{1/2} a_n \varphi_n^{1/2} a_n^*)=\Phi_n(a_n,a_n^*).$$
By \eqref{xx} and \eqref{conv}, we know that 
for any $a\in M$, if $a_n={\bb E}_n(a)$,  then $a-a_n$ tends to $0$ in both spaces 
$(L_2(\varphi), L_2(\varphi)^\dagger)$, and hence in the interpolation space
$\Lambda(\varphi)$.  In particular we have $\| a\|^2_{\Lambda(\varphi)}=\lim_n  \| a_n\|^2_{\Lambda(\varphi)}$.
Since ${\bb E}_n(a^*)={\bb E}_n(a)^*$ for any $a\in M$, ${\bb E}_n$ defines a norm one projection
simultaneously on both spaces 
$(L_2(\varphi), L_2(\varphi)^\dagger)$, and hence in  $\Lambda(\varphi)$. This implies
that $\| a_n\|^2_{\Lambda(\varphi)}=\| a_n\|^2_{\Lambda(\varphi_n)}=\Phi_n(a_n,a_n^*).$
Therefore, for any $a \in M$ we find using (i)
\begin{equation}\label{eq13.25}\| a\|^2_{\Lambda(\varphi)}=\Phi(a ,a
^*).\end{equation}
Now for any unitary $u$ in $N$, 
for any $a\in M$, we obviously have  $ \|uau^* \|_{L_2(\varphi)}=\| a\|_{L_2(\varphi)} $  
 and $ \|uau^* \|_{L_2(\varphi)^\dagger}=\| a\|_{L_2(\varphi)^\dagger} $. By the basic interpolation
 principle (applied to $a\mapsto uau^*$ and its inverse) this implies
 $\| a\|_{\Lambda(\varphi)} = \|uau^* \|_{\Lambda(\varphi)}$. Thus by
 \eqref{eq13.25} we conclude $\Phi(a,a^*)=\Phi(uau^*,ua^*u^*)$, and (ii) follows.
\end{proof}

\begin{proof}[Second proof of Theorem \ref{thm13.6}] (jointly   with Mikael de la Salle).
Consider a c.b.\ map $u\colon  E\to F^*$. Consider the bilinear form
\[
\theta_n\colon \ {\cl A}_n\otimes_{\min} E\times {\cl A}^{op}_n \otimes_{\min} F\to {\bb C}
\]
defined by $\theta_n(a\otimes x, b\otimes y) = \Phi_n(a,b)\langle ux,y\rangle$. Note that ${\cl A}_n \otimes_{\min} E$ and ${\cl A}^{op}_n\otimes_{\min} F$ are exact with constants at most respectively $ex(E)$ and $ex(F)$. Therefore Theorem \ref{thm13.1} tells us that $\theta_n$ satisfies \eqref{eq13.3+} with states $f^{(n)}_1,f^{(n)}_2,g^{(n)}_1,g^{(n)}_2$ respectively on $M\otimes_{\min} A$ and $M^{op}\otimes_{\min}B$. Let $\widehat\Phi_n$ and $\widehat\Phi$ be the bilinear forms defined like before but this time on $M\otimes_{\min} E \times M^{op} \otimes_{\min} F$. Arguing as before using \eqref{eq14.12} and \eqref{eq14.13} we find that $\forall X\in M\otimes E$ $\forall Y\in M^{op}\otimes F$ we have
\begin{equation}\label{eq14.18}
 |\widehat\Phi_n(X,Y)|\le 2C(f^{(n)}_1(X^*X) + f^{(n)}_2(XX^*))^{1/2} (g^{(n)}_2(Y^*Y) + g^{(n)}_1(YY^*))^{1/2}.
\end{equation}
Passing to a subsequence, we may assume that $f^{(n)}_1,f^{(n)}_2$, $g^{(n)}_1,g^{(n)}_2$ all converge pointwise to states $f_1,f_2,g_1,g_2$ respectively on $M\otimes_{\min} A$ and $M^{op}\otimes_{\min} B$. Passing to the limit in \eqref{eq14.18} we obtain
\[
 |\widehat\Phi(X,Y)|\le 2C(f_1(X^*X) + f_2(XX^*))^{1/2} (g_2(Y^*Y) + g_1(YY^*))^{1/2}.
\]
We can then repeat word for word the end of the proof of Theorem \ref{thm14.1} and we obtain \eqref{eq13.9}.
\end{proof}

\begin{rmk}
In analogy with Theorem~\ref{thm14.1}, it is natural to investigate whether the Maurey factorization described in \S~\ref{secm} has an analogue for c.b.\ maps from $A$ to $B^*$ when $A,B$ are  non-commutative $L_p$-space and $2\le p<\infty$. This program was completed by Q.~Xu in \cite{X2}. Note that this requires proving a version of Khintchine's inequality for ``generalized circular elements'', i.e.\ for non-commutative ``random variables'' living in $L_p$ over a type III von~Neumann algebra. Roughly this means that all the analysis has to be done ``without any trace''!
\end{rmk}

\section{GT and Quantum mechanics: EPR and Bell's inequality}\label{sec15}

In 1935, Einstein, Podolsky and Rosen (EPR in short) published a famous article
vigorously criticizing the foundations of quantum mechanics (QM in short).
In their view, the quantum mechanical description of reality is ``incomplete".
This suggests that there are,  in reality, {``hidden variables"}
that we cannot measure because our technical means are not yet
powerful enough, but that the statistical character of the (experimentally confirmed) predictions
of quantum mechanics can be explained by this idea, according to which standard quantum mechanics would be the statistical description of the underlying hidden variables.

In 1964, J.S. Bell observed that the hidden variables theory could be
put to the test. He proposed an inequality (now called ``Bell's inequality")
that is a {\it consequence}  of the hidden variables assumption.

Clauser, Holt, Shimony and Holt
(CHSH, 1969),  modified the Bell inequality and 
suggested that experimental verification should be possible.
Many experiments later, there seems to be
  consensus among the experts that   the Bell-CHSH  inequality is {\it violated},
  thus essentially the   {``hidden variables"} theory is invalid,
and in fact the measures tend to agree with the predictions of QM.

We refer the reader to Alain Aspect's papers \cite{As1,As2}
for an account of the experimental saga, and for relevant physical background to    
 the  books \cite[Chapter 6]{Per} or   \cite[Chapter 10]{Aud}
 and also \cite[Complement 5C and 6C]{GAFC}.  For simplicity,
 we will not discuss  here the concept of ``locality" 
 or ``local realism" related to the assumption that the observations are independent (see \cite{As1}).
 See \cite[p. 196]{Aud}
 for an account of the Bohm theory where non-local hidden variables
 are introduced in order to reconcile theory with experiments.
 
  In 1980 Tsirelson observed that GT could be interpreted as
  giving an upper bound for the violation of a (general) Bell inequality,
  and that the {\it violation} of Bell's inequality is closely related to the assertion that
  { $ K_G>1 !$}
  He also found a variant of the CHSH inequality (now called ``Tsirelson's bound"), see \cite{Ts1,Ts2,Ts3,Ts4,FR}.

   The relevant experiment can be schematically described by the following diagram.

\begin{figure}[htpb]
\centering \includegraphics[width=5.5in]{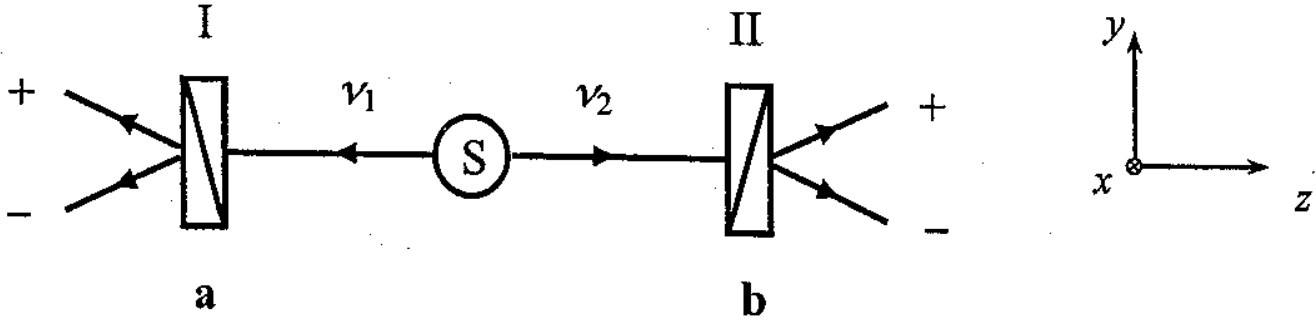}   
\end{figure}

 We have a source that produces two   spin 1/2 particles issued
from the split of a single particle with zero spin (or
equivalently two photons emitted in a radiative atomic cascade).
Such particles are selected because their spin can take only two values
that we will identify with $\pm 1$.
The new particles are sent horizontally in opposite directions
toward two observers, A (Alice)
and B (Bob)
   that are assumed far apart so whatever measuring they do does not influence the
   result of the other. If they   use  detectors oriented along the same direction, say for instance
   in horizontal direction, the two arriving particles will always be found
   with 
   opposite spin components, $+1$ and $-1$ since originally the spin was zero.
   So with certainty the product result will be $-1$. However,
   assume now that Alice   and Bob can use their detector in   different angular   positions,
   that we designate by $i$ ($1\le i\le n$, so there are $n$ positions of their  device).
   Let $A_i=\pm 1$ denote the result of Alice's detector
   and $B_i$  denote the result of Bob's.  
   Again we have $B_i=-A_i$
   Assume now that Bob  uses a detector in a  position $j$,
   while Alice uses position $i\not=j$.
   Then the product $A_iB_j$ is no longer deterministic, the result is randomly equal to $\pm1$,
   but its average, i.e. the covariance of $A_i$ and $B_j$ can be
   measured  (indeed, A and B can repeat the same measurements many times
   and confront their results, not necessarily instantly, say by phone
   afterwards, to compute this average). We will denote it by
   $\xi_{ij}.$
   Here is an outline of Bell's argument:
   
   \n Drawing the consequences of the EPR reasoning, let us introduce
``hidden variables" that can be denoted by a single one $\lambda$ and a probability distribution
   $\rho(\lambda)d\lambda$ so that the covariance of $A_i$ and $B_j$ is
   $$\xi_{ij}=\int A_i(\lambda)B_j(\lambda) \rho(\lambda)d\lambda.$$
 We will now fix a real matrix $[a_{ij}]$ (to be specified later). Then 
 for {\it any} $\rho$ we have
   $$|\sum a_{ij} \xi_{ij}|\le {  {HV}(a)_{\max}}= \sup_{\phi_i=\pm 1\psi_j=\pm 1   } |\sum a_{ij} \phi_i\psi_j|=\| a\|_{\vee} .$$
  Equivalently ${  {HV}(a)_{\max}}$ denotes the maximum value of $|\sum a_{ij} \xi_{ij}|$ over all possible
  distributions $\rho$.
        But Quantum Mechanics predicts
     $$\xi_{ij}=\text{tr}(\rho A_i B_j)$$
     where $A_i$, $B_j$ are self-adjoint unitary operators
     on $H$ ($\dim(H)<\infty$) with spectrum in $\{\pm1\}$
     such that $A_iB_j=B_jA_i$ (this reflects the separation of the two observers) and $\rho$ is a
     non-commutative probability density, i.e. $\rho\ge 0$ is a  trace class operator with $\text{tr}(\rho)=1$.
     This yields
     $$|\sum a_{ij} \xi_{ij}|\le{ {QM}(a)_{\max}}=\sup_{\rho,A_i,B_j} |\text{tr}(\rho\sum a_{ij} A_i B_j)|=\sup_{x\in B_H,A_i,B_j} |\sum a_{ij}\langle A_i  B_j x,x\rangle|.$$
      But the latter norm is familiar: Since  $\dim(H)<\infty$
      and $A_i,B_j$ are commuting self-adjoint unitaries,
       by Theorem \ref{ts1} we have
          $$  \sup_{x\in B_H,A_i,B_j} |\sum a_{ij}\langle A_i  B_j x,x\rangle| =\|a\|_{\min}=\|a\|_{\ell_1\otimes_{H'} \ell_1},$$
    so that 
    $${QM}(a)_{\max}=\|a\|_{\min}=\|a\|_{\ell_1\otimes_{H'} \ell_1}.$$
 Thus, if we set $\|a\|_{\vee}=\|a\|_{\ell_1^n {\buildrel \vee \over \otimes} \ell_1^n}$, then  GT (see  \eqref{eq0.3b}) says 
{ $$ \|a\|_{\vee}\le \|a\|_{\min}\le   K^{\bb R}_G \ \|a\|_{\vee}$$  }
that is precisely equivalent to:
{  $${HV}(a)_{\max}\le {QM}(a)_{\max}\le K^{\bb R}_G \ {HV}(a)_{\max}.$$ }
  But the covariance $\xi_{ij}$ can be physically measured, and hence also $|\sum a_{ij} \xi_{ij}|$
 for a fixed suitable choice of $a$,  so one can obtain an experimental answer for the maximum over all choices of $(\xi_{ij})$
 $$EXP(a)_{\max}$$
 and (for well chosen $a$) it { \it deviates} from the HV value.
 Indeed, if  $a$ is the identity matrix there is no deviation 
 since we know $A_i =-B_i$, and hence $ \sum a_{ij} \xi_{ij}=-n $, but for a {\it non trivial} choice of $a$,
 for instance
 for $$a=\left(\begin{matrix} 1&  \ \ 1\\
 -1&  \ \  1 \end{matrix} \right)
,$$
 a deviation is found, indeed a simple calculation shows that for this particular choice of $a$:
 $$QM(a)_{\max} =\sqrt{2}\ {HV}(a)_{\max}.$$
 In fact the experimental data strongly confirms the
 QM predictions:
  $${HV}(a)_{\max}< EXP(a)_{\max}\simeq  QM(a)_{\max}.$$
{GT} then appears as giving {a bound} for the deviation:
 $${HV}(a)_{\max}< {QM}(a)_{\max}
 \quad {\rm but}\quad {QM}(a)_{\max}\le K^{\bb R}_G\  {HV}(a)_{\max}.$$

 At this point it is quite natural to wonder, as Tsirelson did in \cite{Ts1}, what happens
 in case of three observers A, B and C (Charlie !), or even more. One wonders whether
 the analogous deviation is still bounded or not.  
 This was recently answered   by
  Marius Junge  with Perez-Garcia, Wolf, Palazuelos, Villanueva \cite{J+}.  
  The analogous question for three separated observers $A,B,C$ becomes:  
  Given
  $$a=\sum a_{ijk}e_i\otimes e_j\otimes e_k\in \ell_1^n\otimes\ell_1^n\otimes\ell_1^n\subset C^*({\bb F}_n)\otimes_{\min}C^*({\bb F}_n) \otimes_{\min}C^*({\bb F}_n).$$  
  Is there a constant $K$ such that 
  $${ \| a\|_{\min}\le K \| a\|_{\vee} ?} $$
 The answer is No! In fact,  they get on $\ell_1^{2^{n^2}}\otimes\ell_1\otimes\ell_1$
  $$K\ge c\sqrt{n}, $$  with $c>0$ independent of $n$.
  We give a complete proof of this in the next section.

For more information on the connection between Grothendieck's and Bell's inequalities, see
\cite{KT,AGT,FR,BOV3,JuPa2,Hey,RT,Reg}.

\section{Trilinear counterexamples}\label{sec16}

Many researchers have looked for a trilinear version of GT. The results have been mostly negative (see however \cite{Blei,Blei2}). Note that a bounded trilinear form on $C(S_1)\times C(S_2)\times C(S_3)$ is equivalent to a bounded linear map $C(S_1)\to (C(S_2)\widehat\otimes C(S_3))^*$ and in general the latter will not factor through a Hilbert space. For a quick way to see this fact from the folklore, observe that $\ell_2$ is a quotient of $L_\infty$ and $L_\infty\simeq C(S)$ for some $S$, so that $\ell_2 \widehat\otimes \ell_2$ is a quotient of $C(S) \widehat\otimes C(S)$ and hence $B(\ell_2)$ embeds into $(C(S) \widehat\otimes C(S))^*$. But $c_0$ or $\ell_\infty$ embeds into $B(\ell_2)$ and obviously this embedding does not factor through a Hilbert space (or even through any space not containing $c_0$!). However, with the appearance of the operator space versions, particularly because of the very nice behaviour of the Haagerup tensor product (see \S \ref{sec12}), several conjectural trilinear versions of GT reappeared and seemed quite plausible (see also \S \ref{sec21}  below).

 If $A,B$ are \emph{commutative} $C^*$-algebras it is very easy to see that, for all $t$ in $A\otimes B$
\[
 \|t\|_h = \|t\|_H!
\]
Therefore, the classical GT says that, for all $t$ in $A\otimes B$,  
$
 \|t\|_\wedge \le K_G \|t\|_h$  (see    Theorem \ref{thm1.2bis} above).
Equivalently for all $t$ in $A^*\otimes B^*$ or for all $t$ in $L_1(\mu)\otimes L_1(\mu')$ (see eqref{eq0.3b} )
\[
 \|t\|_{H'} = \|t\|_{h'} \le K_G\|t\|_{\vee}.
\]
But then a sort of miracle happens:\ $\|\cdot\|_h$ is self-dual, i.e.\ whenever $t\in E^*\otimes F^*$ ($E,F$ arbitrary o.s.) we have (see Remark \ref{rem12.4})
\[
 \|t\|_{h'} = \|t\|_h.
\]
Note however that here $\|t\|_h$ means $\|t\|_{E^*\otimes_h F^*}$ where $E^*,F^*$ must be equipped with the dual o.s.s.\ described in \S \ref{sec11}.
In any case, if $A^*,B^*$ are equipped with their dual o.s.s.\ we find $\|t\|_h \le K_G\|t\|_{\vee}$ and a fortiori $\|t\|_{\min} \le K_G\|t\|_{\vee}$, for any $t$ in $A^*\otimes B^*$. In particular, let $A = B = c_0$ so that $A^* = B^* = \ell_1$ (equipped with its dual o.s.s.). We obtain

\begin{cor}\label{cor16.1}
 For any $t$ in $\ell_1\otimes \ell_1$
\begin{equation}\label{eq20.1}
\|t\|_{\min} \le K_G\|t\|_{\vee}.
\end{equation}\end{cor}
Recall here
(see Remark \ref{rk-max}) that $\ell_1$ 
with its maximal o.s.s. (dual to the natural one for $c_0$) can be ``realized'' as an operator space inside $C^*({\bb F}_\infty)$:\ One simply maps the $n$-th basis vector $e_n$ to the unitary $U_n$   associated to the $n$-th free generator.

It then becomes quite natural to wonder whether the trilinear version of Corollary \ref{cor16.1} holds. This was disproved by a rather delicate counterexample due to Junge (unpublished), a new version of which was later published in \cite{J+}.

\begin{thm}[Junge]\label{thm16.2}
The injective norm and the minimal norm are not equivalent on $\ell_1\otimes \ell_1\otimes \ell_1$.
\end{thm}

\begin{lem}\label{lem16.1}
Let $\{\vp_{ij}\mid 1\le i, j\le n\}$ be an i.i.d.\ family representing $n^2$ independent choices of signs $\vp_{ij} =\pm1$. Let ${\cl R}_n \subset L_1(\Omega,{\cl A},{\bb P})$ be their linear span and let ${ w}_n\colon \ {\cl R}_n\to M_n$ be the linear mapping taking $\vp_{ij}$ to $e_{ij}$. We equip ${\cl R}_n$ with the o.s.s.\ induced on it by the maximal o.s.s.\ on $L_1(\Omega,{\cl A},{\bb P})$ (see Remark~\ref{rk-max}). Then $\|{ w}_n\|_{cb} \le 3^{1/2}n^{1/2}$.
\end{lem}

\begin{proof}
Consider a function $\Phi\in L_\infty({\bb P}; M_n)$ (depending only on $(\vp_{ij})$ $1\le i,j\le n)$ such that $\int \Phi\vp_{ij} = e_{ij}$. Let $\widetilde{{ w}}_n\colon \ L_1\to M_n$ be the operator defined by $\widetilde{{ w}}_n(x) = \int x\Phi\ d{\bb P}$, so that $\widetilde{{ w}}_n(\vp_{ij}) = e_{ij}$. Then $\widetilde{{ w}}_n$ extends ${ w}_n$ so that
\[
 \|{ w}_n\|_{cb}\le \|\widetilde{{ w}}_n\|_{cb}.
\]
By \eqref{eq11.3} we have
\[
 \|\widetilde{{ w}}_n\|_{cb} = \|\Phi\|_{L^\infty({\bb P})\otimes_{\min} M_n}= \|\Phi\|_{L^\infty({\bb P}; M_n)}.
\]
But now by \eqref{eq5.1*}
\[
 \inf\left\{\|\Phi\|_\infty\ \Big| \  \int \Phi \vp_{ij}=e_{ij}\right\} \le 3^{1/2} \|(e_{ij})\|_{RC}
\]
and it is very easy to check that
\[
 \|(e_{ij})\|_R = \|(e_{ij})\|_C = n^{1/2}
\]
so we conclude that $\|{ w}_n\|_{cb}\le 3^{1/2}n^{1/2}$.
\end{proof}

\begin{rem}\label{rem16.1} Let ${\cl G}_{n}\subset L_1$ denote the linear span of a family
of $n^2$ independent standard complex Gaussian random variables $\{g^{\bb C}_{ij}\}$.
Let ${ W}_n\colon \ {\cl G}_{n}\to M_n$ be the linear map taking
$g^{\bb C}_{ij}$ to $e_{ij}$. An identical argument to the preceding one
shows that $\|{ W}_n\|_{cb}\le 2^{1/2}n^{1/2}$.
\end{rem}

\begin{proof}[Proof of Theorem \ref{thm16.2}]
We will define below an element $T\in {\cl R}_n\otimes {\cl R}_N \otimes {\cl R}_N \subset L_1\otimes L_1\otimes L_1$, and we note immediately that (by injectivity of $\|\cdot\|_{\min}$) we have
\begin{equation}\label{eq16.14}
\|T\|_{{\cl R}_n\otimes_{\min} {\cl R}_N \otimes_{\min} {\cl R}_N} = \|T\|_{L_1\otimes_{\min} L_1 \otimes_{\min} L_1}.
\end{equation}
Consider the natural isometric isomorphism
\[
 \varphi_N\colon \ M_N\otimes_{\min} M_N\longrightarrow (\ell^N_2 \otimes_2 \ell^N_2) \stackrel{\vee}{\otimes} (\ell^N_2 \otimes_2 \ell^N_2)
\]
defined by
\[
 \varphi_N(e_{pq}\otimes e_{rs}) = (e_p\otimes e_r) \otimes (e_q\otimes e_s).
\]
Let $Y'_i,Y''_j$ be the matrices in $M_N$ appearing in Corollary~\ref{cor13.11}. Let $\chi_N\colon \ (\ell^N_2 \otimes_2\ell^N_2) \stackrel{\vee}{\otimes} (\ell^N_2 \otimes_2 \ell^N_2)\longrightarrow {\cl R}_N \stackrel{\vee}{\otimes} {\cl R}_N$ be the linear isomorphism defined by 
\[
 \chi_N((e_p\otimes e_r) \otimes (e_q\otimes e_s)) = \vp_{pr} \otimes \vp_{qs}.
\]
Since $\chi_N$ is the tensor product of two maps from $\ell^N_2 \otimes_2 \ell^N_2$ to ${\cl R}_N$ and
\begin{equation}\label{eq16.5}
 \left\|\sum \alpha_{pr}{\vp}_{pr}\right\|_{{\cl R}_N}= \left\|\sum \alpha_{pr}{\vp}_{pr}\right\|_1 \le \left\|\sum \alpha_{pr}{\vp}_{pr}\right\|_2 = \left(\sum|\alpha_{pr}|^2\right)^{1/2}\end{equation}
for all scalars $(\alpha_{pr})$, we have $\|\chi_N\|\le 1$ and hence $\|\chi_N\varphi_N\colon \ M_N \otimes_{\min} M_N\to {\cl R}_N \stackrel{\vee}{\otimes} {\cl R}_N\|\le 1$. We define $T\in {\cl R}_n\otimes {\cl R}_N \otimes {\cl R}_N$ by
\[
 T = \sum   \vp_{ij} \otimes (\chi_N\varphi_N) (Y'_i\otimes Y''_j).
\]
By Corollary~\ref{cor13.11} we can choose $Y'_i,Y''_j$ so that (using \eqref{eq16.5} with $ {\cl R}_n$ in place of $ {\cl R}_N$)
\begin{equation}\label{eq16.6}
 \|T\|_{{\cl R}_n\stackrel{\vee}{\otimes} {\cl R}_N\stackrel{\vee}{\otimes} {\cl R}_N}\le 4+\vp.
\end{equation}
But by the preceding Lemma and by \eqref{eq11.4bis}, we have
\begin{equation}\label{eq16.7}
 \|(w_n\otimes w_N \otimes w_N)(T)\|_{M_n\otimes_{\min} M_N \otimes_{\min} M_N}\le 3^{3/2}n^{1/2}N \|T\|_{{\cl R}_n\otimes_{\min} {\cl R}_N\otimes_{\min} {\cl R}_N}.
\end{equation}
Then $$(w_N \otimes w_N) (\chi_N\varphi_N)(e_{pq}\otimes e_{rs}) = e_{pr}\otimes e_{qs}$$ so that if we rewrite $(w_n\otimes w_N\otimes w_N)(T)\in M_n\otimes_{\min} M_N \otimes_{\min} M_N$ as an element $\widehat T$ of\\ $(\ell^n_2 \otimes_2 \ell^N_2 \otimes_2 \ell^N_2)^* \stackrel{\vee}{\otimes} (\ell^n_2 \otimes_2\ell^N_2\otimes_2\ell^N_2)$ then we find
\[
 \widehat T = \left(\sum\nolimits_{ipq} e_i\otimes e_p \otimes e_qY'_i(p,q)\right) \otimes \left(\sum\nolimits_{jrs} e_j\otimes e_r \otimes e_s Y''_j(r,s)\right).
\]
Note that since $\widehat T$ is of rank 1
\[
 \|\widehat T\|_\vee = \left(\sum\nolimits_{ipq}|Y'_i (p,q)|^2\right)^{1/2} \left(\sum\nolimits_{jrs}|Y''_j(r,s)|^2\right)^{1/2}
\]
and hence by \eqref{eq13.16}
\[
 \|(w_n\otimes w_N\otimes w_N)(T)\|_{M_n\otimes_{\min} M_N \otimes_{\min} M_N} = \|\widehat T\|_\vee \ge (1-\vp)nN.
\]
Thus we conclude by \eqref{eq16.6} and \eqref{eq16.7} that
\[
 \|T\|_{L_1\stackrel{\vee}{\otimes} L_1\stackrel{\vee}{\otimes} L_1} \le 4+\vp \quad \text{but}\quad \|T\|_{L_1\otimes_{\min}  L_1\otimes_{\min} L_1} \ge 3^{-3/2}(1-\vp)n^{1/2}.
\]
This completes the proof with $L_1$ instead of $\ell_1$. Note that we obtain a tensor $T$ in 
\[
 \ell^{2^{n^2}}_1 \otimes \ell^{2^{N^2}}_1 \otimes \ell^{2^{N^2}}_1.\eqno \qed
\]
\renewcommand{\qed}{}\end{proof}
\begin{rem} Using \cite[p. 16]{P7}, one can reduce the dimension  $2^{n^2}$ to one $\simeq {n^4}$ in the preceding example.
\end{rem}
\begin{rem} The preceding construction establishes the following fact: there is a constant $c>0$ such that for any positive integer $n$ there is a
finite rank  map
$u\colon\ \ell_\infty \to \ell_1\otimes_{\min} \ell_1 $,
associated to a tensor $t\in \ell_1\otimes \ell_1\otimes \ell_1$,  such that the map $u_n
\colon\ M_n(\ell_\infty) \to M_n(\ell_1\otimes_{\min} \ell_1)$ defined
by $u_n([ a_{ij}])=[u(a_{ij})]$ satisfies
\begin{equation}\label{tri}\|u_n\|\ge c\sqrt{n} \| u \|.\end{equation}
A fortiori of course $\|u\|_{cb}=\sup_m\|u_m\|\ge c\sqrt{n}\| u \|.$
Note that by \eqref{eq11.7} $\|u\|_{cb}=\|t\|_{\ell_1\otimes_{\min} \ell_1\otimes_{\min} \ell_1}$ and  by GT (see \eqref{eq20.1})
we have $\|t\|_{\ell_1\stackrel{\vee}\otimes  \ell_1\stackrel{\vee}\otimes \ell_1}\le\|u\| = \|t\|_{\ell_1\stackrel{\vee}\otimes  (\ell_1 \otimes_{\min} \ell_1)}\le
K^{\bb C}_G \|t\|_{\ell_1\stackrel{\vee}\otimes  \ell_1\stackrel{\vee}\otimes \ell_1} $.\\
In this form (as observed in \cite{J+})
 the estimate \eqref{tri} is \emph{sharp}. \\ Indeed, there is a constant $c'$,  such that  
$\|u_n\|\le c'\sqrt{n} \|u\|$
for any operator space $E$ and any
map $u\colon\ \ell_\infty \to E$.  This can be deduced from the fact that the identity of $M_n$ admits a factorization
through an $n^2$-dimensional quotient $Q$
(denoted by $E_n^2$ in \cite[p. 910]{HP1}) of $L_\infty$ of the form $I_{M_n}=ab$ where
$b\colon\ M_n\to Q$ and $a\colon\ Q\to M_n$ satisfy $\|b\|\|a\|_{cb}\le  c'\sqrt{n}$.
\end{rem}
\begin{rmk}
 In \cite{J+}, there is a different proof  of \eqref{eq13.15} that uses  the fact that the von~Neumann algebra of ${\bb F}_\infty$ embeds into an ultraproduct (von~Neumann sense) of matrix algebras. In this approach,
 one can use, instead of random matrices, the
 residual finiteness of free groups.
  This   leads to the following substitute for Corollary~\ref{cor13.11}: Fix $n$ and $\vp>0$, then for all $N\ge N_0(n,\vp)$ there are $N\times N$ matrices $Y'_i,Y''_j$ and $\xi_{ij}$ such that $Y'_i,Y''_j$ are unitary (permutation matrices) satisfying:
\begin{equation}
 \left\|\sum\nolimits^n_{i,j=1} \alpha_{ij}(Y'_i\otimes Y''_j + \xi_{ij})\right\|_{M_{N^2}} \le (4+\vp) \left(\sum |\alpha_{ij}|^2\right)^{1/2},\tag*{$\forall (\alpha_{ij})\in {\bb C}^{n^2}$}
\end{equation}
and
\begin{equation}
 N^{-1} \text{ tr}(\xi^2_{ij}) < \vp.\tag*{$\forall i,j=1,\ldots, n$}
\end{equation}
This has the advantage of a deterministic choice of $Y'_i,Y''_j$ but the inconvenient that $(\xi_{ij})$ is not explicit, so only the ``bulk of the spectrum'' is controlled. An entirely explicit and non-random
example proving Theorem \ref{thm16.2} is apparently still unknown.
\end{rmk}

\section{Some open problems}\label{sec21}

By \cite{JP}, the minimal and maximal $C^*$-norms are not equivalent on the
algebraic tensor product $B(\ell_2)\otimes B(\ell_2)$. Curiously, nothing more on this is known:

\begin{prbl}\label{prbl21.0} Show that there are at least 3 (probably infinitely many)
mutually inequivalent $C^*$-norms   on  $B(\ell_2)\otimes B(\ell_2)$.
\end{prbl}
In Theorem~\ref{thm14.1}, we obtain an equivalent norm on $CB(A,B^*)$, and we can claim that this elucidates the \emph{Banach space} structure of $CB(A,B^*)$ in the same way as GT does when $A,B$ are commutative $C^*$-algebras (see Theorem~\ref{thm1.1}). But since the space $CB(A,B^*)$ comes naturally equipped with an o.s.s.\ (see Remark~\ref{rem11.6}) it is natural to wonder whether we can actually \emph{improve} Theorem~\ref{thm14.1}  to also describe the o.s.s.\ of $CB(A,B^*)$.
In this spirit, the natural conjecture
(formulated by David Blecher in the collection of problems \cite{B1})
 is as follows:

\begin{prbl}\label{prbl21.1}
Consider a jointly c.b.\ bilinear form $\varphi\colon \ A\times B\to B(H)$ (this means that $\varphi\in CB(A, CB(B,B(H)))$ or equivalently that $\varphi\in CB(B,CB(A,B(H)))$ with the obvious abuse).

Is it true that $\varphi$ can be decomposed as $\varphi = \varphi_1+ {}^t\varphi_2$ where $\varphi_1\in CB(A\otimes_h B, B(H))$ and $\varphi_2 \in CB(B\otimes_h A, B(H))$ where ${}^t\varphi_2(a\otimes b) = \varphi_2(b\otimes a)$?
\end{prbl}

Of course (as can be checked using
 Remark \ref{rem12.3}, suitably generalized) the converse holds: any   map $\varphi$ with such a decomposition is c.b. 

As far as we know this problem  is open even in the commutative case, i.e.\ if $A,B$ are both commutative. The problem clearly reduces to the case $B(H)=M_n$ with a constant $C$ independent of $n$ such that $\inf\{\|\varphi_1\|_{cb} + \|\varphi_2\|_{cb}\} \le C\|\varphi\|_{cb}$. In this form (with an absolute constant $C$), the case when $A,B$ are matrix algebras is also open. Apparently, this might even be true when $A,B$ are exact operator spaces, in the style of Theorem~\ref{thm13.6}. 

In a quite different direction, the trilinear counterexample in \S \ref{sec16} does not rule out a certain trilinear version of the o.s.\ version of GT that we will now describe:

Let $A_1,A_2,A_3$ be three $C^*$-algebras. Consider a (jointly) c.b.\ trilinear form $\varphi\colon \ A_1\times A_2\times A_3\to {\bb C}$, with c.b.\ norm $\le 1$. This means that for any $n\ge 1$ $\varphi$ defines a trilinear map of norm $\le 1$
\[
 \varphi[n]\colon \ M_n(A_1)\times M_n(A_2) \times M_n(A_3)\longrightarrow M_n \otimes_{\min} M_n \otimes_{\min} M_n \simeq M_{n^3}.
\]
Equivalently, this means that 
$
 \varphi\in CB(A_1, CB(A_2,A^*_3))
$
or more generally this is the same as
\[
 \varphi\in CB(A_{\sigma(1)}, CB(A_{\sigma(2)}, A^*_{\sigma(3)}))
\]
for any permutation $\sigma$ of $\{1,2,3\}$, and the obvious extension of Proposition \ref{pro11.7} is valid.

\begin{prbl}\label{prbl21.2}
Let $\varphi\colon \ A_1\times A_2\times A_3\to {\bb C}$ be a (jointly) c.b.\ trilinear form as above. Is it true that $\varphi$ can be decomposed as a sum
\[
 \varphi = \sum\nolimits_{\sigma\in S_3} \varphi_\sigma
\]
indexed by the permutations $\sigma$ of $\{1,2,3\}$, such that for each such $\sigma$, there is a $$\psi_\sigma\in CB(A_{\sigma(1)} \otimes_h A_{\sigma(2)} \otimes_hA_{\sigma(3)}, {\bb C})$$ such that
\[
 \varphi_\sigma(a_1\otimes a_2\otimes a_3) = \psi_\sigma(a_{\sigma(1)} \otimes a_{\sigma(2)} \otimes a_{\sigma(3)})\ ?
\]
\end{prbl}

Note that such a statement would obviously imply the o.s.\ version of GT given as Theorem~\ref{thm14.1} as a special case.

\begin{rmk} See \cite[p.104]{Blei2} for an open question of the same flavor
as the preceding one, for bounded trilinear forms on  commutative $C^*$-algebras, but involving factorizations through $L_{p_1}\times L_{p_2} \times L_{p_3}$  with $(p_1,p_2,p_3)$ 
such that $p^{-1}_1+p^{-1}_2+p^{-1}_3= 1$ varying depending on the trilinear form $\varphi$.
\end{rmk}

\section{GT in graph theory and computer science}\label{sec17}

The connection of GT with computer science 
stems roughly  from the following remark. Consider for a real matrix $[a_{ij}]$
the problem of computing the maximum
of $\sum a_{ij} \vp_i \vp'_j$ over all choices of signs $\vp_i=\pm 1$
$\vp'_j=\pm 1$. This maximum over these $2^N\times   2^N$ choices
is ``hard" to compute in the sense that no polynomial time algorithm is known
to solve it. However, by convexity, this maximum is clearly the same 
as    \eqref{eq0.4}. Thus GT says that our maximum,  that is trivially less than \eqref{eq0.6},
is also  larger than \eqref{eq0.6} divided by
  $K_G$  , and, as incredible as it may seem at first glance,
  \eqref{eq0.6} itself  is solvable in polynomial time! Let us now describe this
more carefully.

In \cite{AMMN}, the Grothendieck constant of a (finite) graph ${\cl G}=(V,E)$ is introduced, as the smallest constant $K$ such that, for every $a\colon \ E\to {\bb R}$, we have
\begin{equation}\label{sec17eq1}
\sup_{f\colon  V\to S} \sum_{\{s,t\}\in E} a(s,t) \langle f(s), f(t)\rangle \le K \sup_{f\colon  V\to \{-1,1\}} \sum_{\{s,t\}\in E} a(s,t) f(s) f(t)
\end{equation}
where $S$ is the unit sphere of $H=\ell_2$. Note that we may replace $H$ by the span of the range of $f$ and hence we may always assume $\dim(H) \le |V|$. We will denote by $K({\cl G})$ the smallest such $K$. Consider for instance the complete bipartite graph ${\cl C}{\cl B}_n$  on vertices $V = I_n\cup J_n$ with $I_n = \{1,\ldots, n\}$, $J_n = \{n+1,\ldots, 2n\}$ with $(i,j)\in E\Leftrightarrow i\in I_n$, $j\in J_n$. In that case \eqref{sec17eq1} reduces to \eqref{eq1.2bis} and, recalling \eqref{eq2+Dpre},  we have
\begin{equation}
 K(\cl{CB}_n) = K^{\bb R}_G(n,n)\quad{\rm and}\quad
\sup\nolimits_{n\ge 1} K(\cl{CB}_n) = K^{\bb R}_G.
\end{equation}
If ${\cl G}
 = (V',E')$ is a subgraph of ${\cl G}$ (i.e.\ $V'\subset V$ and $E'\subset E$) then obviously
\[
 K({\cl G}') \le K({\cl G}).
\]
Therefore, for any bipartite graph ${\cl G}$ we have
\[
 K({\cl G}) \le K^{\bb R}_G.
\]
However, this constant does not remain bounded for general (non-bipartite) graphs. In fact, it is known (cf. \cite{NRT} and \cite{Meg} independently) that there is an absolute constant $C$ such that for any ${\cl G}$ with no selfloops (i.e. $(s,t) \notin E$ when $s=t$) 
\begin{equation}\label{sec17eq2}
 K({\cl G}) \le C(\log(|V|)+1).
\end{equation}
Moreover by \cite{AMMN} this logarithmic growth is asymptotically optimal,
thus improving   Kashin and Szarek's lower bound \cite{KS} 
that   answered   a question raised by Megretski \cite{Meg} (see the above Remark \ref{rem-kso}).
Note however that the $\log(n)$ lower bound found in  \cite{AMMN} for the complete graph on $n$ vertices
was somewhat non-constructive,  and  more recently,     explicit examples
were produced in \cite{ABHKS} but yielding  only a  $\log(n)/\log\log(n)$ lower bound.

 Let us now describe briefly the motivation behind these notions.
 One major breakthrough was    Goemans and Williamson's paper
 \cite{GW} that uses semidefinite programming (ignorant  readers like the author
 will find   \cite{Lov}   very illuminating).
 The origin can be traced back to a famous problem
 called MAX CUT. By definition, a cut in a graph $G=(V,E)$
 is a set of edges connecting a subset $S\subset V$
of the vertices to the complementary subset $V- S$. The MAX CUT problem
 is to find a cut with maximum cardinality.

 MAX CUT is known to be hard in general, in precise technical terms
it is NP-hard. Recall that this implies that $P=NP$ would follow if it could be solved
in polynomial time.
 Alon and Naor \cite{AN} proposed to compare MAX CUT to another
 problem that they called  the CUT NORM problem:\ We are given a real matrix $(a_{ij})_{ {  i\in {\cl R}}, {j\in  {\cl C}}}$ we want to compute efficiently
\[
 Q = \max_{{ I\subset  {\cl R}},{J\subset  {\cl C}}} \left|\sum_{{ i\in I},{j\in J}} a_{ij}\right|,
\] 
and to find a  pair $(I,J)$ realizing the maximum.
Of course the connection to GT is that this quantity $Q$ is such that
$$Q \le Q' \le 4 Q \quad{\rm where} \quad Q' = \sup\limits_{x_i,y_j\in \{-1,1\}}\sum a_{ij}x_i y_j.$$ So roughly computing $Q$ is reduced to computing $Q'$, and finding $(I,J)$ to finding a pair of choices of signs.
Recall that $S$ denotes the unit sphere in Hilbert space.
Then precisely Grothendieck's inequality   \eqref{eq1.2bis}  tells us that 
$$ \frac1{K_G} Q'' \le  Q'\le Q''  \quad{\rm where} \quad Q'' = \sup\limits_{x_i,y_j\in S} \sum a_{ij}\langle x_i,y_j\rangle.$$ The point is that computing $Q'$ in polynomial time is not known and very unlikely to become known (in fact it would imply $P=NP$) while the problem of computing $Q''$ turns out to be \emph{feasible}: Indeed, it falls into the category of \emph{semidefinite programming} problems and these are known to be solvable,
within an 
additive error of $\vp$,  in polynomial time (in the length of the input and in the logarithm of $1/\vp$), because one can exploit Hilbert space geometry, namely
``the ellipsoid method"  (see \cite{GLS1,GLS2}).  
Inspired by earlier work by Lov\'asz-Schrijver and Alizadeh (see \cite{GW}),
Goemans and Williamson 
introduced, in the context of   MAX CUT, a geometric method to
analyze the connection of a problem such as $Q''$ with the original
combinatorial problem such as $Q'$. In this context, $Q''$ is called a
``relaxation'' of the corresponding problem $Q'$. Grothendieck's inequality
furnishes a new way to approach harder problems such as the CUT NORM problem $Q''$. The CUT NORM problem is indeed harder than MAX CUT since Alon and Naor \cite{AN}
showed that MAX CUT can be cast a special case of CUT NORM; this implies in
particular that CUT NORM is also NP hard.\\
 It should be noted that by known results: \emph{There exists $  \rho<1$  such that even 
being able to compute $Q'$ up to a factor $>\rho$ in polynomial time would imply $P=NP$}. The analogue of this fact (with $\rho=16/17$) for MAX CUT    goes back to H{\aa}stad \cite{Has} (see also \cite{PY}).
Actually, according to  \cite{Ragh,RaghS}, for any $0<K<K_G$, 
assuming
a strengthening of $P\not=NP$ called the ``unique games conjecture", it is NP-hard to compute  
any quantity $q$ such that $K^{-1}q \le  Q'$. While, for  $K>K_G$, we can take
$q=Q''$ and then compute a solution in polynomial time by semidefinite programming.
So in this framework $K_G$  seems connected to the $P=NP$ problem ! \\
%We refer the reader to \cite{Lov} for an introduction to semidefinite programs.
 But the CUT NORM (or MAX CUT) problem is not just a maximization problem as the
preceding ones   of computing $Q'$ or $Q''$. There one wants to find
the vectors with entries $\pm 1$ that achieve the maximum or 
 realize a value greater than a fixed factor $\rho$ times the maximum. The
 semidefinite programming produces $2n$ vectors in the $n$-dimensional Euclidean sphere, so there is
 a delicate  ``rounding problem" to obtain instead vectors in $\{-1,1\}^n$. In \cite{AN},
 the authors transform a known proof of GT into 
 one that solves efficiently this rounding problem. They obtain
 a deterministic polynomial time algorithm that produces
 $x,y\in \{-1,1\}^n$ such that $ \sum a_{ij}\langle x_i,y_j\rangle \ge 0.03  \ Q''$ (and a fortiori $\ge 0.03  \ Q'$).
 Let $\rho=2\log(1+\sqrt{2})/\pi>0.56$ be the inverse of Krivine's upper bound for $K_G$.
 They also obtain a randomized algorithm for that value of $\rho$.
 Here randomized means that the integral vectors are random vectors
  so that the expectation of $ \sum a_{ij}\langle x_i,y_j\rangle$ is $\ge \rho  \ Q''$.

The papers \cite{AMMN,AN} ignited a lot of interest in the computer science literature.
Here are a few samples:
In \cite{AB} the Grothendieck constant of the random graph on $n$ vertices
is shown to be of order $\log(n)$ (with probability tending to $1$ as $n\to \infty$), answering a question left open in \cite{AMMN}. 
In  \cite{KNS}, the ``hardness" of computing the norm of a matrix
on $\ell_p^n\times \ell_p^n$ is evaluated depending on the value
of $2\le p\le \infty$. The Grothendieck related case is of course $p=\infty$.
%maximal   several known proofs of GT as (polynomial time) algorithms for solving $Q''$.
In \cite{CW}, the authors study a generalization of MAX CUT that they call MAXQP.
In  \cite{Ragh,RaghS}, a (programmatic)  approach is proposed to compute $K_G$ efficiently.
More references are \cite{  KhN, LinialS,  Tr}.
\bigskip

\section{Appendix: The Hahn-Banach argument}\label{sechb}

Grothendieck used duality of tensor norms and doing that, of course he used the Hahn-Banach extension theorem repeatedly. However, he did not use it in the way that
is described below, e.g. to describe the $H'$-norm. At least in \cite{Gr1}, he   systematically passes to the  ``self-adjoint" and ``positive definite"  setting and uses positive linear forms.
This explains why he obtains only $\|.\|_{H}\le 2\|.\|_{H'}$, and, although he suspects that it is true without it, he cannot remove the factor 2. 
This was done later on, in Pietsch's (\cite{Pie}) and Kwapie\'n's work (see in particular \cite{Kw1}).
 Pietsch's factorization theorem for $p$-absolutely summing
operators became widely known. That was proved using a form of the Hahn-Banach
separation, that has become routine to Banach space specialists since then.
Since this kind of argument is used repeatedly in the paper, we append this section to it.

We start with a variant of the famous min-max lemma.

\begin{lem}\label{hb1}
Let $S$ be a set and
let ${\cal F}\subset
\ell_\infty(S)$ be a convex cone of real valued functions on $S$ such that
$$\forall\ f \in {\cal F}\qquad \sup_{s\in S} f(s) \ge 0.$$
Then there is a net $(\lambda_\alpha)$ of finitely supported probability
measures on $S$ such that
$$\forall\ f\in {\cal F}\qquad \lim \int fd\lambda_\alpha \ge 0.$$
\end{lem}
\begin{proof}
Let  $\ell_\infty(S,\bb R)$ denote the space
all bounded {\it real valued} functions on $S$
with its usual norm.
  In $\ell_\infty(S,\bb R)$ the set ${\cal F}$ is
disjoint from the set
$C_{-}=\{\varphi\in \ell_\infty(S,\bb R)\mid \sup\varphi
< 0\}$. Hence by the Hahn-Banach theorem (we
separate the convex set  ${\cal F}$ and the convex
open set $C_{-}$) there is  a non zero $\xi\in
\ell_\infty(S,\bb R)^*$
 such that $\xi(f)\ge 0$ \
$\forall \ f\in {\cal F}$ and
$\xi(f)\le 0$ \
$\forall \ f\in C_{-}$. 
Let $M\subset \ell_\infty(S,\bb R)^*$ be the cone of all
finitely supported (nonnegative) measures
on $S$ viewed as functionals on $\ell_\infty(S,\RR)$.
Since we have $\xi(f)\le 0$ \
$\forall \ f\in C_{-}$, $\xi$ must be in the
bipolar of $M$ for the duality 
of the pair $(\ell_\infty(S,\RR),
\ell_\infty(S,\RR)^*)$. Therefore, by the bipolar
theorem,  
$\xi$ is the limit for the topology
$\sigma(\ell_\infty(S,\RR)^*,\ell_\infty(S,\RR))$ of a
net of finitely supported  (nonnegative) measures
$\xi_\alpha$ on
$S$. We have
for any $f$ in $\ell_\infty(S,\RR)$, $\xi_\alpha
(f)\to \xi
(f)$ and this holds in particular if $f=1$, thus
(since
$\xi$ is nonzero) we may assume
$\xi_\alpha (1)> 0$, hence  if we set
$\lambda_\alpha(f)=\xi_\alpha(f)/\xi_\alpha(1)$
we obtain the announced result.
\end{proof}

The next statement is meant to illustrate the way the preceding Lemma  is used,
but many variants are possible.   

Note that
if $B_1,B_2$ below are unital and commutative, we can identify them
with $C(T_1),C(T_2)$ for compact sets $T_1,T_2$.
A state on $B_1$ (resp. $B_2$) then corresponds to a (Radon) probability measure
on $T_1$ (resp. $T_2$).

\begin{pro}\label{hb2}
Let $B_1,B_2$ be
$C^*$-algebras, let $F_1 \subset B_1$  and
$F_2\subset B_2$ be two linear subspaces, and
let $\varphi\colon \ F_1\times  F_2\to \CC$ be
a bilinear form. The following are equivalent:
\begin{itemize}
\item[\rm (i)] For any finite sets
$(x^j_1)$ and 
$(x^j_2)$ in $F_1$ and $F_2$ respectively, we have
$$\left|\sum \varphi(x^j_1,x^j_2)\right| \le \left\|\sum x^j_1 
x^{j*}_1\right\|^{1/2} \left\|\sum x^{j*}_2 x^j_2\right\|^{1/2}.$$
\item[\rm (ii)] There are states $f_1$ and $f_2$ on $B_1$ and $B_2$ respectively such
that 
$$  |\varphi(x_1,x_2)|
\le (f_1(x_1x^*_1) f_2(x^*_2x_2))^{1/2}.\leqno
\forall (x_1,x_2) \in F_1\times F_2 $$
\item[\rm (iii)] The form $\varphi$   extends to a bounded bilinear form $\tilde  \varphi$
satisfying {\rm (i),(ii)} on $B_1\times B_2$. 
\end{itemize}
\end{pro}

\begin{proof} Assume (i).
First observe
that by the arithmetic/geometric mean
inequality we have for  any $a,b\ge 0$
$$(ab)^{1/2} = \inf_{t>0} \{2^{-1}(ta+(b/t))\}.$$
In particular we have
$$\left\|\sum x^j_1x^{j*}_1\right\|^{1/2} \left\|\sum x^{j*}_2x^j_2\right\| 
^{1/2} \le 2^{-1}\left(\left\|\sum x^j_1x^{j*}_1\right\| + \left\|\sum x^{j*}_2 
x^j_2\right\|\right).$$
Let $S_i$ be the set of states on $B_i$ $(i=1,2)$ and let $S = S_1\times S_2$. 
The last inequality implies
$$\left|\sum\varphi (x^j_1,x^j_2)\right| \le 2^{-1} \sup_{f=(f_1,f_2)\in S} 
\left\{f_1\left(\sum x^j_1x^{j*}_1\right) + f_2\left(\sum x^{j*}_2x^j_2\right) 
\right\}.$$
Moreover, since the right side does not change if we replace $x^1_j$ by 
$z_jx^1_j$ with $z_j\in \CC$ arbitrary such that $|z_j|=1$, we may assume that 
the last inequality holds with $\sum|\varphi(x^j_1,x^j_2)|$ instead of 
$\left|\sum \varphi(x^j_1,x^j_2)\right|$.
 Then let ${\cal F} \subset 
\ell_\infty(S,\RR)$ be the convex cone formed of all
possible functions $F\colon \  S\to {\RR}$ of
the form
$$F(f_1,f_2) = \sum\nolimits_j 2^{-1} f_1(x^j_1x^{j*}_1) + 2^{-1} f_2(x^{j*}_2x^j_2) - 
|\varphi(x^j_1,x^j_2)|.$$
By the preceding Lemma, there is a net ${\cal U}$ of
probability measures 
$(\lambda_\alpha)$ on $S$ such that for any $F$ in  ${\cal F}$ we have 
$$\lim_{\cal U} \int F(g_1,g_2)
d\lambda_\alpha(g_1,g_2) \ge 0.$$ We may as
well assume that ${\cal U}$ is an ultrafilter.
Then if we set
$$f_i = \lim_{\cal U} \int g_i
d\lambda_\alpha(g_1,g_2) \in S_i$$ (in the
weak-$*$ topology $\sigma(B^*_i,B_i)$), we
find that for any choice of 
$(x^j_1)$ and $(x^j_2)$ we have
$$\sum\nolimits_j 2^{-1}f_1(x^j_1x^{j*}_1) + 2^{-1}f_2(x^{j*}_2x^j_2) - |\varphi(x^j_1, 
x^j_2)|\ge 0.$$
In particular $\forall x_1\in F_1,
\forall x_2\in F_2$
$$2^{-1}(f_1(x_1x^*_1) + f_2(x^*_2x_2)) \ge |\varphi(x_1,x_2)|.$$
By the homogeneity of $\varphi$, this implies
$$\inf_{t>0} \{2^{-1}(tf_1(x_1x^*_1) +
f_2(x^*_2x_2)/t)\} \ge  |\varphi(x_1,x_2)|$$
and we obtain the desired conclusion (ii) using our initial observation on the 
geometric/arith\-metic mean inequality. This shows (i)   implies (ii). The converse is obvious.
To show that (ii) implies (iii), let $H_1$ (resp. $H_2$) be the Hilbert space obtained  from
equipping $B_1$ (resp. $B_2$) with the scalar product $\langle x,y\rangle =f_1(y^*x)$ (resp. 
$=f_2(y^*x)$) (``GNS construction"), let $J_k\colon B_k\to H_k$ 
($k=1,2$) be the canonical inclusion, and let ${\cl H}_k={J_k(E_k)}\subset H_k$. By
(ii) $\varphi$ defines a bilinear form of norm $\le 1$ on ${\cl H}_1 \times{\cl H}_2$.
Using the   orthogonal projection from, say, $H_1$ to the closure of ${\cl H}_1$, the latter extends to
a form $\psi$ of norm $\le 1$ on $H_1 \times H_2$, and then
$\tilde \varphi (x_1,x_2)=\psi (J_1(x_1), J_2(x_2))$ is the desired extension.
\end{proof}
Let $T_j$ denote the unit ball of $F_j^*$ equipped with the (compact)
weak-$*$ topology.
Applying the preceding to the embedding $F_j\subset B_j=C(T_j)$
and recalling \eqref{eq0.1h}, we find

\begin{cor}\label{hb3-} Let  $\varphi\in \ (F_1\otimes  F_2)^* $. The following are equivalent:
\begin{itemize}
 \item[\rm (i)] We have $  |\varphi(t)|\le \|t\|_H$ $(\forall t \in F_1\otimes F_2) $.

\item[\rm (ii)] There are probabilities 
$\lambda_1$ and $\lambda_2$ on $T_1$ and $T_2$ respectively such
that 
$$  |\varphi(x_1,x_2)|
\le (  \int |x_1|^2 d\lambda_1 \int |x_2|^2 d\lambda_2)^{1/2}.\leqno
\forall (x_1,x_2) \in F_1\times F_2 $$
\item[\rm (iii)] The form $\varphi$   extends to a bounded bilinear form $\tilde  \varphi$
satisfying {\rm (i),(ii)}   on $B_1\otimes B_2$. 
\end{itemize}
\end{cor}

\begin{rem}\label{rem4.3bis}
Consider a bilinear form $\psi$ on $\ell^n_\infty \otimes \ell^n_\infty$ defined by $\psi(e_i\otimes e_j) = \psi_{ij}$.    
The preceding Corollary implies that $\|\psi\|_{H'}\le 1$ iff there are $(\alpha_i)$, $(\beta_j)$ in the unit sphere of $\ell^n_2$ such that for any $x_1,x_2$ in $\ell^n_2$
\[
 \left|\sum \psi_{ij}x_1(i)x_2(j)\right|\le \left(\sum |\alpha_i|^2|x_1(i)|^2\right)^{1/2} \left(\sum |\beta_j|^2 |x_2(j)|^2\right)^{1/2}.
\]
Thus $\|\psi\|_{H'}\le 1$ iff we can write $\psi_{ij} = \alpha_i a_{ij}\beta_j$ for some matrix $[a_{ij}]$ of norm $\le 1$ on $\ell^n_2$,
and some $(\alpha_i)$, $(\beta_j)$ in the unit sphere of $\ell^n_2$.
\end{rem}

Here is another variant:
\begin{pro}\label{hb3}
Let $B  $ be a
$C^*$-algebra, let $F  \subset B $    be a linear subspace, and
let $ u\colon\ F \to E$ be
a linear map into a Banach space $E$.
Fix numbers $a,b\ge 0$. The following are equivalent:
\begin{itemize}
\item[\rm (i)] For any finite sets
$(x_j)$   in $F $, we have
$$\left(\sum \| ux_j\|^2\right)^{1/2}\le    \left(a\left\|\sum x_j 
x_{j}^*\right\| + b \left\|\sum x_j^* 
x_{j}\right\|\right)^{1/2} .$$
\item[\rm (ii)] There are states $f ,g$   on $B $   such
that 
$$\|ux\|  
\le (af(xx^*)+bg(x^*x) )^{1/2}.\leqno
\forall  x  \in F   $$
\item[\rm (iii)] The map $u$ extends to a   linear map $\tilde u\colon\ B\to E$
satisfying {\rm(i)} or {\rm(ii)} on the whole of $B$.
\end{itemize}
\end{pro}
\begin{proof} Let $S$ be the set of pairs $(f,g)$ of states
on $B$. Apply Lemma \ref{hb1} to the family of functions on $S$ of the form
$(f,g) \mapsto a\sum f(x_j x_{j}^*)+b\sum f(x_j^* x_{j}) -\sum \| ux_j\|^2$.
The extension in  {\rm(iii)} is obtained using the orthogonal projection 
onto the space spanned by $F$ in the Hilbert space associated
to the right hand side of  {\rm(ii)}. We leave the easy details to the reader.
\end{proof}
\begin{rem}\label{rkpie}
When $B$ is commutative, i.e. $B=C(T)$ with $T$ compact, the inequality in
{\rm(i)} becomes, letting $c=(a+b)^{1/2}$
 \begin{equation}\label{eq3.1pie}\left(\sum \| ux_j\|^2\right)^{1/2}\le    c\left\|\sum |x_j|^2
 \right\|^{1/2} . \end{equation}
 Using $F\subset C(T)$ when $T$  is  equal to the unit ball of $F^*$, we find
  the ``Pietsch factorization" of $u$: there is a probability $\lambda$ on $T$
  such that $\|u(x)\|^2\le c^2\int |x|^2 d\lambda \ \forall x\in C(T)$.
 An operator satisfying this for some $c$ is called 2-summing and
 the smallest constant $c$ for which this holds is denoted by $\pi_2(u)$.
 More generally,  if $0<p<\infty$, $u$ is called $p$-summing with $\pi_p(u)\le c$ if there
 is a probability $\lambda$ on $T$
  such that $\|u(x)\|^p\le c^p\int |x|^p d\lambda \ \forall x\in C(T)$,
  but we       really use only $p=2$ in this paper.\\
  See \cite {PauR} for a recent use of this Pietsch factorization for
  representations of $H^\infty$.
\end{rem}

\n\textbf{Acknowledgment.} 
This paper is partially based on the author's notes for a minicourse
at IISc Bangalore (10th discussion meeting in Harmonic Analysis, Dec. 2007) and 
a lecture at IHES    (Colloque en l'honneur d'Alexander Grothendieck, Jan. 2009). I thank the organizers for the resulting stimulation.
  I am  very grateful to
  Mikael de la Salle  for  useful conversations,   and for various improvements.  Thanks
  are also due to  Alain Aspect, Kate Juschenko, Hermann K\"onig, Assaf Naor,  \'Eric Ricard, Maurice Rojas,
  Boris Tsirelson  for useful correspondence.

\end{document}